\documentclass{amsart}

\usepackage[all]{xy}

\usepackage[a4paper, top=1.91cm, bottom=1.91cm, left=3cm, right=3cm, heightrounded, marginparwidth=3.5cm, centering]{geometry}
\usepackage{hyperref}
\usepackage{cleveref}
\usepackage{calligra,mathrsfs}

\usepackage{verbatim}
\usepackage{tikz-cd}
\usepackage{bbold}
\usepackage{amssymb}
\usepackage{amsaddr}
\usepackage{todonotes}
\usepackage[T1]{fontenc}

\usepackage[bbgreekl]{mathbbol}
\usepackage{amsfonts,amsthm}
\usepackage{stmaryrd}
\usepackage{scalerel}

\DeclareSymbolFontAlphabet{\mathbb}{AMSb}
\DeclareSymbolFontAlphabet{\mathbbl}{bbold}
\newcommand{\prism}{{\scaleobj{1.25}{\mathbbl{\Delta}}}}

\DeclareRobustCommand{\SkipTocEntry}[5]{}

\theoremstyle{plain}
\newtheorem{theorem}{Theorem}[section]
\newtheorem{proposition}[theorem]{Proposition}
\newtheorem{lemma}[theorem]{Lemma}
\newtheorem*{claim*}{Claim}
\newtheorem{corollary}[theorem]{Corollary}

\newtheorem{assumption}[theorem]{Assumption}
\newtheorem{question}[theorem]{Question}
\theoremstyle{definition}
\newtheorem{construction}[theorem]{Construction}
\newtheorem{Setup}[theorem]{Setup}
\newtheorem{definition}[theorem]{Definition}
\newtheorem{example}[theorem]{Example}
\newtheorem{remark}[theorem]{Remark}

\newcommand{\an}{\mathrm{an}}

\newcommand{\C}{\mathbb{C}}

\newcommand{\can}{\mathrm{can}}

\newcommand{\cone}{\mathrm{cone}}

\newcommand{\cycl}{\mathrm{cycl}}

\newcommand{\dR}{\mathrm{dR}}

\newcommand{\et}{\mathrm{{\acute{e}t}}}
\newcommand{\F}{\mathbb{F}}

\newcommand{\Fsm}{\mathrm{F\text{-}sm}}

\newcommand{\Gal}{\mathrm{Gal}}
\newcommand{\Gm}{\mathbb{G}_m}
\newcommand{\Ga}{\mathbb{G}_a}

\newcommand{\Symvw}{{\mathcal S}_p}
\newcommand{\Symv}{\mathcal S}
\newcommand{\Hom}{\mathrm{Hom}}
\newcommand{\HT}{\mathrm{HT}}

\newcommand{\Hsm}{\mathrm{H\text{-}sm}}

\newcommand{\Id}{\mathrm{Id}}

\renewcommand{\inf}{\mathrm{inf}}

\newcommand{\Lft}{\mathcal{L}\mathrm{ft}}
\newcommand{\LS}{\mathrm{LS}}

\newcommand{\N}{\mathbb{N}}

\newcommand{\calPerf}{\mathcal{P}erf}

\newcommand{\Q}{\mathbb{Q}}

\newcommand{\rig}{\mathrm{rig}}

\newcommand{\sm}{\mathrm{sm}}
\newcommand{\Spa}{\mathrm{Spa}}

\newcommand{\Spf}{\mathrm{Spf}}
\newcommand{\Spec}{\mathrm{Spec}}

\newcommand{\WCart}{\mathrm{WCart}}
\newcommand{\Z}{\mathbb{Z}}

\renewcommand{\O}{\mathcal{O}}

\newcommand{\wt}{\widetilde}

\newcommand{\tf}{[\tfrac{1}{p}]}
\newcommand{\X}{{\mathcal X}}

\newcommand{\Varphi}{\psi}
\newcommand{\Varpsi}{\psi'}

\newcommand{\Higgs}{\mathrm{Higgs}}

\newcommand*\isomarrow{%
	\xrightarrow{\raisebox{-0.35em}{\smash{\ensuremath{\sim}}}}
}

\begin{document}

	\title[Hodge--Tate stacks and \MakeLowercase{v}-perfect complexes on smooth rigid spaces]{Hodge--Tate stacks and non-abelian $p$-adic Hodge theory of \MakeLowercase{v}-perfect complexes on rigid spaces}
	
	\author[J. Ansch\"utz, B. Heuer, A.-C. Le Bras]{Johannes Ansch\"utz, Ben Heuer, Arthur-C\'esar Le Bras}
	
	\begin{abstract}
          Let $X$ be a quasi-compact quasi-separated $p$-adic formal scheme  that is smooth either over a perfectoid $\Z_p$-algebra or over some ring of integers of a $p$-adic field.
          We construct a fully faithful functor from perfect complexes on the Hodge--Tate stack of $X$ up to isogeny to perfect complexes on the v-site of the generic fibre of $X$.
          Moreover, we describe perfect complexes on the Hodge--Tate stack in terms of certain derived categories of Higgs, resp.\ Higgs--Sen modules. This leads to a derived $p$-adic Simpson functor. 
	\end{abstract}
	
	\maketitle

\section{Introduction}
\label{sec:introduction-1}

\subsection{v-perfect complexes via the Hodge--Tate stack}
Let $p$ be a prime and let $\X$ be an adic space over $\Q_p$. Let $\X_v$ be the v-site of $\X$, consisting of perfectoid spaces over $\X$ endowed with the v-topology. Let $\mathcal{O}_{\X_v}$ be the structure sheaf of $\X_v$. The starting point of this article is the following:

\begin{question} \label{sec:introduction-2-question-introduction}
 How can one describe the category $\mathcal{P}erf(\X_v):=\mathcal{P}erf(\X_v,\mathcal{O}_{\X_v})$ of perfect complexes on $\X_v$ in terms of data that only involve the analytic or \'etale topology of $\X$?
\end{question}

If $\X$ is perfectoid, then $\mathcal{P}erf(\X_v)\cong \mathcal{P}erf(\X_\an)$ \cite[Theorem~2.1]{anschutz2021fourier}. In contrast, for rigid analytic $\X$, vector bundles on $\X_v$ are related to Higgs bundles: in this paper, we study the following cases.

\begin{enumerate}
	\item $\X=X^\rig$ is the adic generic fiber of a quasi-compact, quasi-separated (qcqs) smooth formal scheme over $\mathcal O_{\mathbb C_p}$. More generally, we consider qcqs $p$-adic formal schemes $X$ that are ``smoothoid'', i.e.\ locally $X$ is smooth over $\Spf(R_0)$ for a perfectoid $\Z_p$-algebra $R_0$. 
	\item $\X=X^\rig$ for some qcqs smooth $p$-adic formal scheme over a complete $p$-adic discrete valuation ring $\O_K$ with perfect residue field $k$.\footnote{In fact, our methods are strong enough to also handle the case that $k$ is only assumed to be $p$-finite, i.e., $[k:k^p]<\infty$, cf.\ \Cref{sec:fully-faithf-arithm-remark-on-imperfect-residue-field-case}} We call such $X$ ``arithmetic''.
        \end{enumerate}
        
         Bhatt--Lurie and Drinfeld have associated to any $p$-adic formal scheme $Z$  a $p$-adic formal stack $Z^\HT\to Z$ called the ``Hodge--Tate stack'' (\cite[Construction 3.7]{bhatt2022prismatization}, \cite{drinfeld2020prismatization}). The stack $X^\HT$ will be the key to our approach to \Cref{sec:introduction-2-question-introduction}. Using that for any perfectoid $Z$ the natural map $Z^\HT\to Z$ is an isomorphism, we construct in \Cref{sec:form-reduct-prov} a natural pullback morphism
\[
\alpha_X^\ast \colon \mathcal{P}erf(X^\HT)\tf\to \mathcal{P}erf(\X_v)
\]
from the category of perfect complexes on $X^\HT$ up to isogeny. Our first key result is the following:

\begin{theorem}[{\Cref{sec:tori-over-perfectoid}, \Cref{sec:smooth-form-schem-galois-cohomo-in-arithmetic-case}}]
	\label{sec:introduction-2-statement-main-theorem}
	If $X$ is qcqs smoothoid as in (1) or arithmetic as in (2), then $\alpha_X^\ast$ is fully faithful.
      \end{theorem}
      
	On the other hand, we explain that $\mathcal{P}erf(X^\HT)\tf$ can be described in terms of  Hodge-theoretic data on $\X_\an$, like Higgs bundles or Sen modules. For instance, for smoothoid $X$ over a perfectoid base ring $R_0$, we introduce a category of ``Higgs perfect complexes'' on $\X$ and show that any lift $\tilde{X}$ of $X$ to $A_2(R_0):=A_{\rm inf}(R_0)/\ker(\theta)^2$ induces a fully faithful functor  $\beta_{\tilde{X}}$ from $\calPerf(X^\HT)\tf$ into this category. 
	In combination, this realizes a ``derived $p$-adic Simpson functor'' via the diagram:
	\[
	\begin{tikzcd}[row sep=0.1cm]
		& \calPerf(X^{\HT}
		)\tf \arrow[ld, "\alpha_X^\ast"'] \arrow[rd, "\beta_{\tilde{X}}"] &                      \\
		\calPerf(\mathcal X_v) \arrow[rr, "\text{$p$-adic Simpson}", dotted,<->] &                                                                    & \big\{\text{Higgs-perfect complexes on }\mathcal X\big\}
	\end{tikzcd}\]      
	Besides giving a partial answer to \Cref{sec:introduction-2-question-introduction} in this case, this provides a fruitful new geometric perspective on Faltings' $p$-adic non-abelian Hodge theory \cite{faltings2005p}. Indeed, we use it to prove:
	\begin{enumerate}
		\item new derived versions of the local and global $p$-adic Simpson functor for small Higgs bundles, generalising these from vector bundles to perfect complexes (\Cref{sec:smoothoid-case-1-derived-local-p-adic-simpson-correspondence-introduction,t:intro-local-p-adic-Simpson-functor-geometric}), 
		\item a derived version of Sen theory in families in the arithmetic setting, (\Cref{sec:arithmetic-case-1-derived-local-p-adic-simpson-in-arithmetic-case}).
	\end{enumerate}
	This leads to a uniform geometric approach to Sen theory and $p$-adic Simpson in this context.
	As an application, the generalisation from vector bundles to perfect complexes formally implies the comparison of cohomology in each case.    We now  describe each of these results in more detail.

      \subsection{The smoothoid case}
\label{sec:smoothoid-case}
Let $X$ be a qcqs smoothoid $p$-adic formal scheme. For simplicity, we assume  that $X$ lives over a perfectoid base ring $R_0$. Let $\Omega^1_{X}:=\Omega^1_{X|R_0}$ be the sheaf of $p$-completed K\"ahler differentials, which is finite  locally free, cf.\ \Cref{sec:fully-faithf-smooth-differentials-are-smoothoid}. Let $(A_0,I_0)$ be the perfect prism associated with $R_0$, i.e., $A_0/I_0\cong R_0$. We denote by $\{1\}$ the Breuil--Kisin twist $I_0/I_0^2\otimes_{R_0}(-)$.

We then define a relative formal group scheme $\mathcal{T}_{X}^\sharp\{1\}\to X$ as the PD-envelope of the zero section of the (twisted) tangent bundle $\mathcal{T}_{X}\{1\}$ of $X$ relatively over $R_0$, i.e., locally on $X=\Spf(R)$,
  \[
    \mathcal{T}_{X}\{1\}:=\Spf(\mathcal{S}_p(\Omega^1_{R}\{-1\}))\quad \text{and}\quad \mathcal{T}^\sharp_{X}\{1\}:=\Spf(\Gamma_R(\Omega^1_{R}\{-1\})^\wedge_p)
    ,
  \]
  where  $\mathcal{S}_p(-)$ denotes the $p$-completed symmetric algebra and $\Gamma_R(-)$ denotes the PD-algebra.

  From the results of Bhatt--Lurie in \cite{bhatt2022absolute,bhatt2022prismatization}, we will deduce:
  
\begin{theorem}[\Cref{sec:smoothoid-case-1-complexes-on-ht-in-split-case}]
\label{sec:smoothoid-case-1-complexes-on-x-ht-smoothoid-case-introduction}
  Any section $X\to X^\HT$ of $X^\HT\to X$ induces an isomorphism
  \[
    X^\HT\cong B_X\mathcal{T}_{X}^\sharp\{1\}
  \]
  of $\mathcal{T}_{X}^\sharp\{1\}$-gerbes between the Hodge--Tate stack of $X$ and the classifying stack of $\mathcal{T}_{X}^\sharp\{1\}$ over $X$, hence a fully faithful functor
		\[
		\mathcal{D}(X^\HT)\hookrightarrow \mathcal{D}(\mathcal{T}_{X}^\vee\{-1\}).
		\]
		Its essential image is given by those  $\mathcal{M}\in \mathcal D(\mathcal{T}_{X}^\vee\{-1\})$ for which on any affine open $U:=\Spf(R)\subseteq X$, each $\delta \in \Omega^{1,\vee}_{R}\{1\}$ (seen as a section of $\mathcal{T}_{U}^\vee\{-1\}$) acts locally nilpotently on $H^\ast(U,\mathcal{M}\otimes_{R}^LR/p)$.
\end{theorem}

In particular, a splitting of $X^\HT$ implies that vector bundles on $X^\HT$ can be described as vector bundles $\mathcal{M}$ on $X$ together with a Higgs field, i.e., a morphism of vector bundles
\[
 \theta_{\mathcal{M}}\colon \mathcal{M}\to \mathcal{M} \otimes_{\mathcal{O}_X} \Omega^1_{X}\{-1\}
\]
with $ \theta_{\mathcal{M}}\wedge  \theta_{\mathcal{M}}=0$
whose components are topologically nilpotent.

\subsection{Globalization}

The assumption that $X^\HT$ is split is rather restrictive, cf.\ \cite[Remark 5.13]{bhatt2022prismatization}. It is satisfied if $X=\Spf(R)$ is affine and smooth over some perfectoid ring $R_0$, c.f.\ \cite[Construction 5.2]{bhatt2022prismatization}. For example, a splitting is induced by the datum of a toric chart for $X$.

To globalise the construction, we therefore show that the pushout of $X^\HT$ along a rescaling map is already split by the datum of a flat lift of $X$ to $A_0/I_0^2$. The existence of such a lift is a much weaker condition than that of a prismatic lift. We thus obtain a $p$-adic Simpson functor for a considerably weaker datum, at the expense of introducing a stronger convergence conditions on the Higgs field. Assume that the perfectoid base $R_0$ contains a primitive $p$-th root of unity $\zeta_p\in R_0$.

\begin{theorem}[\Cref{sec:appl-Hodge--Tate-functor-for-x-lift}, \Cref{sec:appl-Hodge--Tate-embedding-from-colimit}]
  \label{sec:smoothoid-case-1-complexes-on-the-Hodge--Tate-stack-for-choice-of-lift-introduction}
  Each lift $\tilde{X}$ of $X$ to $A_0/I_0^2$ induces a morphism
  $\Phi_{\tilde{X}}\colon X^\HT\to B_X\mathcal{T}^\sharp_{X}\{1\}$, linear over $\mathcal{T}^\sharp_{X}\{1\}\xrightarrow{\zeta_p-1} \mathcal{T}^\sharp_{X}\{1\}$, 
  that induces  a natural  equivalence
  \[ (\zeta_p-1)_\ast X^\HT \isomarrow B_X\mathcal{T}^\sharp_{X}\{1\}.\]
  Thus the pullback
  \[
    \Phi_{\tilde{X}}^\ast\colon \mathcal{P}erf(B_X\mathcal{T}^\sharp_{X}\{1\})\to \mathcal{P}erf(X^\HT)
  \]
  is fully faithful on isogeny categories.
\end{theorem}

This will lead to our global version of a derived $p$-adic Simpson functor for small objects:
The category $\mathcal{P}erf(B_X\mathcal{T}^\sharp_{X}\{1\})$ can again be described by Higgs bundles as in \Cref{sec:smoothoid-case-1-complexes-on-x-ht-smoothoid-case-introduction}. Roughly, the Higgs field gets multiplied by $(\zeta_p-1)$ under $\Phi^\ast_{\tilde{X}}$. Thus, if $X^\HT\to X$ is split, the essential image of $\Phi^\ast_{\tilde{X}}$ is given by Higgs perfect complexes $(\mathcal{M},\theta_{\mathcal{M}})$ with the stronger convergence condition that $\frac{1}{(\zeta_p-1)}\theta_{\mathcal{M}}$ is topologically nilpotent. In effect, this means that given an $A_0/I_0^2$-lift, the local descriptions of perfect complexes on $X^\HT$ can be glued  after introducing a convergence condition.

\subsection{A derived $p$-adic Simpson correspondence}
As a consequence of our analysis of complexes on the Hodge--Tate stack, we get a derived improvement of the previously known $p$-adic Simpson correspondence for ``small'' objects. The starting point of this is \Cref{sec:smoothoid-case-1-complexes-on-x-ht-smoothoid-case-introduction}.
Passing to the isogeny category of perfect complexes on both sides leads to the notion of a \textit{Higgs perfect complex} on the generic fibre $\X$.  Roughly, this is a perfect complex $\mathcal{M}$ on $\X$ with a Higgs field 
\[\theta_{\mathcal{M}}\colon \mathcal{M}\to \mathcal{M}\otimes_{\mathcal{O}_X}\Omega^1_{\X}\{-1\}.\]
The condition describing the essential image generalises to this context: we call a Higgs perfect complex $\omega$-\textit{Hitchin-small} if $\theta_{\mathcal{M}}$ is topologically nilpotent, see \Cref{sec:appl-Hodge--Tate-definition-higgs-perfect-complex}. Here: $\omega=(\zeta_p-1)^{-1}$.

Combining \Cref{sec:introduction-2-statement-main-theorem,sec:smoothoid-case-1-complexes-on-x-ht-smoothoid-case-introduction}, we obtain a
derived version of a local $p$-adic Simpson functor:
\begin{theorem}[{\Cref{t:local-p-adic-Simpson-functor-geometric}}]\label{t:intro-local-p-adic-Simpson-functor-geometric}
	Let $X$ be a smoothoid formal scheme with adic generic fibre $\X$. Any splitting $s:X\to X^\HT$ (for example induced by a toric chart) induces  a fully faithful functor
	\[ \LS_s:\left\{\begin{array}{@{}c@{}l}\text{$\omega$-Hitchin-small Higgs}\\\text{perfect complexes on } \X\end{array} \right\}\hookrightarrow \mathcal{P}erf(\X_v)
	.\]
\end{theorem}
 Using instead \Cref{sec:smoothoid-case-1-complexes-on-the-Hodge--Tate-stack-for-choice-of-lift-introduction}, we also get a global derived $p$-adic Simpson functor. For this we need to rescale the convergence condition on Higgs fields by a factor of $(\zeta_p-1)$ and arrive at the stronger notion of a  \textit{Hitchin-small} Higgs perfect complex, see \Cref{sec:appl-Hodge--Tate-definition-higgs-perfect-complex}.

\begin{theorem}[\Cref{sec:smoothoid-case-1-derived-local-p-adic-simpson-correspondence}]
  \label{sec:smoothoid-case-1-derived-local-p-adic-simpson-correspondence-introduction}
  Each lift $\tilde{X}$ of $X$ to $A_0/I_0^2$ induces a natural fully faithful functor
  \[ \mathrm S_{\tilde{X}}:\left\{\begin{array}{@{}c@{}l}\text{Hitchin-small Higgs}\\\text{perfect complexes on } \X\end{array} \right\}\hookrightarrow\mathcal{P}erf(\X_v).\]
\end{theorem}
Since \Cref{sec:smoothoid-case-1-derived-local-p-adic-simpson-correspondence-introduction,t:intro-local-p-adic-Simpson-functor-geometric} work on the derived level, they in particular include a comparison of cohomology and an extension of known functors to coherent Higgs modules on $\X$. They furthermore use a smallness condition which only involves the spectral properties of the Higgs field, contrary to other instances of the $p$-adic Simpson correspondence in the literature, e.g.\ in \cite{faltings2005p}. We note that for v-vector bundles, Faltings' notion of smallness  implies Hitchin smallness.

\subsection{The arithmetic case}
\label{sec:arithmetic-case}
Switching to an arithmetic setup, let us now assume that $X$ is a qcqs smooth $p$-adic formal scheme over the ring of integers $\O_K$ of a $p$-adic field, i.e.\, a complete discretely valued extension $K$ of $\Q_p$ with perfect residue field. Let $C$ be the completion of an algebraic closure of $K$. Let $\X$ be the rigid generic fibre of $X$.
Then v-vector bundles on $\X$ bear a relation to $p$-adic representations of $\Gal(C|K)$: For $X=\Spf(\O_K)$, v-vector bundles on $\X$ are equivalent to semi-linear representations of $\Gal(C|K)$ on finite dimensional $C$-vector spaces. For general $X$, v-vector bundles thus give rise to $p$-adic families of $\Gal(C|K)$-representations.

As in \S\ref{sec:smoothoid-case}, we start by analyzing complexes on $X^\HT$ when there exists  a global prismatic lift of $X$. The natural map
$X^\HT\to \Spf(\O_K)^\HT
$
makes the relative Hodge--Tate structure map
\[
  \pi_{X|\O_K}\colon X^\HT\to X\times_{\Spf(\O_K)} \Spf(\O_K)^\HT
\]
into a gerbe banded by the affine, faithfully flat group scheme $\mathcal{T}^\sharp_{X|\O_K}\{1\}$ over $X\times_{\Spf(\O_K)}\Spf(\O_K)^\HT$, cf.\ \cite[Proposition 5.12]{bhatt2022prismatization}. Here, $\{1\}$ refers to twisting by  $\mathcal{O}_{\Spf(\Z_p)^\HT}\{1\}$, cf.\ \cite[Example 3.5.2]{bhatt2022absolute}. 

Using this, we can describe complexes on $X^\HT$ by a derived version of Higgs--Sen bundles:

\begin{theorem}[\Cref{sec:autom-overl-1-description-of-g-a}, \Cref{sec:appl-Hodge--Tate-structure-of-x-ht-for-some-prismatic-lift}]
	\label{sec:introduction-2-description-of-ht-stack}
	Assume furthermore that $X=\Spf(R)$ is affine and that there exists a bounded prism $(A,I)$ such that $R=A/I$. Then the resulting morphism
	\[
	\overline{\rho_A}\colon X\to X^\HT
	\]
	is faithfully flat and exhibits $X^\HT\cong B_XG_A$ as the classifying stack of the relative group scheme $G_A$ over $X$ which embeds into the semi-direct product $(\mathcal{T}_{A|\Z_p}^\sharp\{1\}\times_{\Spf(A)}\Spf(R)) \rtimes \Gm^\sharp$ as the subgroup of pairs $(D\colon A\to \Ga^\sharp\{1\}=I/I^2\otimes_R\Ga^\sharp,x\in \Gm^\sharp)$ with $D$ a continuous derivation such that $D(a)=(1-x)(a\otimes 1)$ for $a\in I$. Consequently, there exists a natural fully faithful functor
        \[
          \mathcal{D}(X^\HT)\hookrightarrow \mathrm{Mod}_{\mathcal S}(\mathcal{D}(\O_K[\Theta_\pi])),
		\]
		where $\mathcal S$ is the $p$-adically completed symmetric algebra of the $\O_X$-module $\Omega_{X|\O_K}^{1,\vee}\{1\}$.
		The essential image is given by complexes $\mathcal{M}$ which are derived $p$-complete and such that each $\delta\in \Omega^{1,\vee}_{X|\O_K}\{1\}$ and  $\Theta^p_\pi-E^\prime(\pi)^{p-1}\Theta_\pi$ act locally nilpotently on $H^\ast(X, \mathcal{M}\otimes_{\Z_p}^L\F_p)$.
              \end{theorem}
				Here $\pi\in \O_K$ is a choice of uniformiser and $E(u)\in W(k)[[u]]$ is the minimal polynomial of $\pi$ over the maximal unramified subextension  $K_0$ of $K$.  \Cref{sec:introduction-2-description-of-ht-stack} then describes vector bundles on $X^\HT$ as triples $(\mathcal{M},\theta_{\mathcal{M}}, \Theta_\pi)$ where $\mathcal{M}$ is a vector bundle on $X$, where $\theta_{\mathcal{M}}\colon \mathcal{M}\to \mathcal{M}\otimes_{\O_X} \Omega^1_{X|\O_K}\{-1\}$ is a topologically nilpotent Higgs field, and $\Theta_\pi\colon \mathcal{M}\to \mathcal{M}$ is a \textit{Sen operator}, such that the diagram
\[\begin{tikzcd}
	\mathcal{M}  & {\mathcal{M} \otimes_{\O_X} \Omega^1_{X|\O_K}\{-1\}} \\
	\mathcal{M} & {\mathcal{M}\otimes_{\O_X} \Omega^1_{X|\O_K}\{-1\}}
	\arrow["{\theta_{\mathcal{M}}}", from=2-1, to=2-2]
	\arrow["{\theta_{\mathcal{M}}}", from=1-1, to=1-2]
	\arrow["{\Theta_\pi-E^\prime(\pi)}", from=1-2, to=2-2]
	\arrow["{\Theta_\pi}"', from=1-1, to=2-1]
      \end{tikzcd}\]
    commutes. We thus recover the notion of an ``enhanced Higgs bundle'' of Min--Wang, which they have used  to describe prismatic Hodge--Tate crystals \cite[Theorem 4.3]{MinWang22}. The latter are equivalent to vector bundles on $X^\HT$ via \cite[Remark 9.2]{bhatt2022prismatization}. From this perspective, \Cref{sec:introduction-2-description-of-ht-stack} is a generalisation of this description from vector bundles to the derived category  $\mathcal{D}(X^\HT)$.

    In \Cref{sec:appl-Hodge--Tate-definition-higgs-sen-complexes}, we define the notion of a Higgs--Sen perfect complex on the rigid generic fiber $\X$. We also define the subclass of Hitchin-small Higgs--Sen perfect complexes, for which there is a convergence condition only on the Sen operator. We arrive at an analogue of \Cref{sec:smoothoid-case-1-derived-local-p-adic-simpson-correspondence-introduction}:

    \begin{theorem}[\Cref{sec:arithmetic-case-1-derived-local-p-adic-simpson-in-arithmetic-case-plus-description-of-perf-on-x-ht}]
      \label{sec:arithmetic-case-1-derived-local-p-adic-simpson-in-arithmetic-case}
      Let $K$ be $p$-adic and choose a uniformizer $\pi \in \mathcal{O}_K$. Let $\X=X^\rig$ for $X$ a smooth $p$-adic formal scheme over $\O_K$. There exists a natural fully faithful functor
      \[
       \mathrm{S}_\pi \colon {\left\{ \begin{array}{@{}c@{}l}\text{Hitchin-small Higgs-Sen}\\\text{perfect complexes on $\X$}\end{array}\right\}}\hookrightarrow \mathcal{P}erf(\X_v).
      \]
    \end{theorem}
    This is more canonical than \Cref{sec:smoothoid-case-1-complexes-on-the-Hodge--Tate-stack-for-choice-of-lift-introduction} as there is a canonical choice of a lift:
    Arguing similarly to \cite{MinWang22}, we derive \Cref{sec:arithmetic-case-1-derived-local-p-adic-simpson-in-arithmetic-case} from \Cref{sec:smoothoid-case-1-derived-local-p-adic-simpson-correspondence-introduction} by using the canonical lift of the base change of $X$ to $\mathcal{O}_{\widehat{\overline{K}}}$. Here, the Higgs field on the base change is automatically nilpotent by Galois equivariance.
   
\subsection{Relation to previous work}
\label{sec:relat-prev-work}    
The field of $p$-adic non abelian Hodge theory has seen many developments in recent years. Let us explain how the previous theorems relate to the recent literature on the $p$-adic Simpson correspondence. As the field has become vast, we only give some pointers to the most directly related works rather than giving an exhaustive account. 

	\begin{enumerate}
		\item \cite{bhatt2022absolute}, \cite{bhatt2022prismatization}, \cite{drinfeld2020prismatization} have introduced the prismatization of a $p$-adic formal scheme, and in particular the theory of Hodge--Tate stacks on which this paper is based. If $X=\Spf(\Z_p)$, \Cref{sec:introduction-2-statement-main-theorem} can be deduced from \cite[Proposition 3.7.3]{bhatt2022absolute}. The description of vector bundles on $X^\HT$ via Higgs bundles (for $X$ affinoid smoothoid) appears in \cite[Corollary 6.6]{bhatt2022prismatization} and in this case the Cartier theory we use in \Cref{sec:introduction-2-description-of-ht-stack} yields a derived enhancement, which is probably well-known. If $K$ is $p$-adic, i.e., $k$ perfect, and unramified, Higgs-Sen bundles and their relation to $X^\HT$ are discussed in \cite[Remark 9.2]{bhatt2022prismatization}. 
		
		\item Similarly, there is a relation to the independent work of Tian on prismatic crystals \cite{tian2021finiteness}. The precise connection to our work is furnished by Bhatt--Lurie's equivalence of prismatic Hodge--Tate crystals and vector bundles on the Hodge--Tate stack, cf.\ \cite[Theorem 6.5]{bhatt2022prismatization}. \Cref{sec:introduction-2-description-of-ht-stack} is thus a derived version of Tian's description of prismatic crystals and their cohomology and as such generalises \cite[Thm~4.12, Thm~4.14]{tian2021finiteness}. 
		\item  When $X$ is a smooth formal scheme over $\O_C$, \Cref{t:intro-local-p-adic-Simpson-functor-geometric,sec:smoothoid-case-1-derived-local-p-adic-simpson-correspondence-introduction} are closely related to the local and global $p$-adic Simpson correspondence of Faltings \cite{faltings2005p}, hence also the related constructions of Abbes--Gros, Tsuji \cite{abbes2016p}\cite{Tsuji-localSimpson}, Liu--Zhu \cite{LiuZhu_RiemannHilbert} and Wang \cite{wang2021p}. Indeed, v-vector bundles are equivalent to Faltings' generalised representations (cf \cite[Prop.\ 2.3]{G-torsors-perfectoid-spaces}).

		Our $p$-adic Simpson functors are more general in three different ways: First, we work in a derived setting with perfect complexes instead of vector bundles. Second, our convergence condition of Hitchin-smallness encompasses both Faltings' notion of smallness as well as the local systems treated by Liu--Zhu. Third, we allow perfectoid families of smooth formal schemes. In work in progess, we use this to improve from an equivalence of categories to the more geometric statement of an isomorphism of moduli stacks of small objects.
		\item Our approach to non-abelian $p$-adic Hodge theory is in some sense a continuation of Abbes--Gros' idea of constructing a period sheaf geometrically (using what they call the \textit{torsor of deformations}). In our setting, period sheaves play a similar role, but a difference is arguably that in our setting, their definition immediately suggests itself from the geometry of $X^\HT$: the period sheaves arise as v-sheaves associated to the automorphism group of splittings of the Hodge--Tate stack and its variants. See \Cref{sec-comparison-with-prev-constructions} for more details on the relation of our approach to the ones of Faltings, Abbes--Gros, Liu--Zhu and Wang.
		\item Conceptually, our approach to the $p$-adic Simpson correspondence is perhaps most closely related to that of Tsuji \cite[IV]{abbes2016p}: The role that $X^\HT$ plays in our article in reorganising this correspondence in a geometric way is reminiscent of the role that the ``Higgs site'' plays in Tsuji's work. The latter is defined in terms of certain systems of $A_\inf/\xi^n$-lifts of $X$. Tsuji proves that ``Higgs crystals'', certain modules on the Higgs site, correspond to Higgs bundles with the same convergence condition as the one derived from our \Cref{sec:smoothoid-case-1-complexes-on-the-Hodge--Tate-stack-for-choice-of-lift-introduction}.
		
		\item For the arithmetic case of $p$-adic fields $K$, closely related results have been obtained by Tsuji \cite[\S15]{Tsuji-localSimpson} in the algebraic setup, and more recently also in a prismatic setting by Min--Wang \cite{MinWang22}, \cite{min2021hodge}, based on the earlier work of Tian \cite{tian2021finiteness}. They prove \Cref{sec:introduction-2-statement-main-theorem} for vector bundles, including the comparison of cohomology.  That being said, our proof of \Cref{sec:introduction-2-statement-main-theorem} is different and the extension to perfect complexes is a new contribution.
		
		\item More generally, our approach is also suitable for Sen theory over discretely valued fields with imperfect residue field (cf \Cref{sec:fully-faithf-arithm-remark-on-imperfect-residue-field-case}). This has previously been considered by Brinon \cite{BrinonSen}, Yamauchi \cite{YamauchiSen}, Ohkubo \cite{OhkuboSenTheory},  and more recently by Gao \cite{gao2021padic} and He \cite{he2022sen}.
		\item In \cite{analytic_HT}, the authors obtained \Cref{sec:introduction-2-statement-main-theorem} for $X=\mathrm{Spf}(\mathcal{O}_K)$ where $K$ is a $p$-adic field, and deduced a description of the whole of $\mathcal{P}erf(\mathrm{Spa}(K)_v)$.
		\item \cite{HWZ} has asked to what extent the $p$-adic Simpson correspondence generalises from vector bundles to principal bundles under rigid groups. Our approach  is well-suited to study such generalisations in the case of good reduction, by considering principal bundles on $X^\HT$.
	\end{enumerate}
	Further to the advantages mentioned above, one benefit of our approach to non-abelian Hodge theory via $X^\HT$ is that it is very general: it can in principle  be applied to many $p$-adic formal schemes without having to set up new machinery in each situation. On the one hand, this allows us to treat the various setups considered in the literature in a completely uniform way. On the other, it yields a strategy to construct Simpson/Sen-theoretic functors in new settings in the future.
	
\subsection{Outlook}
\label{sec:introduction-limitations}
The results of this paper only deal with the rigid generic fibre of smooth formal schemes, rather than with arbitrary smooth rigid analytic spaces. This good reduction restriction comes from the fact that our method relies crucially on Bhatt-Lurie's Hodge--Tate stacks, a theory developed for $p$-adic formal schemes. That being said, without further global assumptions like properness, the \textit{small} $p$-adic Simpson correspondence is inherently a statement about integral structures and will therefore always impose some conditions on integral models.

Towards a complete answer to \Cref{sec:introduction-2-question-introduction}, one therefore has to look beyond the small $p$-adic Simpson correspondence. Instead, motivated by our results in this article, we hope for the existence of an ``analytic Hodge--Tate stack'' $\X^\HT$ attached to any rigid space $\X$ that would give the answer to \Cref{sec:introduction-2-question-introduction}.  This will be the subject of future work of the first and third author.

Another question is how to describe the essential image of the functors \Cref{sec:smoothoid-case-1-derived-local-p-adic-simpson-correspondence-introduction,sec:arithmetic-case-1-derived-local-p-adic-simpson-in-arithmetic-case}.  Motivated by \cite{camargo2022locally}, we expect that if $\X$ is a smooth rigid space over a perfectoid field $C$, then any $\mathcal{F}\in \mathcal{P}erf(\X_v)$ admits a canonical Higgs field
$
  \theta_{\mathcal{F}}\colon \mathcal{F}\to \mathcal{F}\otimes_{\mathcal{O}_{\X}} \Omega^1_{\X}(-1).
$
The essential image of the functor $\mathrm{S}_{\tilde{X}}$  should then be given by those $\mathcal{F}$ such that $\theta_{\mathcal{F}}$ satisfies a suitable topological nilpotence condition. We will study the case of vector bundles in the companion paper \cite{AHLB-companion}.

\subsection{Plan of the paper}
\label{plan}
\S2-\S5 are devoted to the proof of \Cref{sec:introduction-2-statement-main-theorem}. The construction of $\alpha_X^\ast$ is not difficult and very general, but the proof of fully faithfulness is more involved. We ultimately reduce it to a computation in group cohomology in \S\ref{sec:introduction}. Most work is devoted to making a group action of perfectoid Galois covers explicit, a subtle issue that we address in \S\ref{sec:prismatic-homotopies}-\S\ref{sec:examples}.
   
   Using  the local geometry of the Hodge--Tate stack, we then prove  \Cref{sec:smoothoid-case-1-complexes-on-x-ht-smoothoid-case-introduction,sec:introduction-2-description-of-ht-stack} in  \S\ref{sec:complexes-Hodge--Tate-and-higgs-bundles}.  Putting everything up to this point together, these combine with \Cref{sec:introduction-2-statement-main-theorem} to give the local version of our derived $p$-adic Simpson correspondence at the end of  \S\ref{sec:complexes-Hodge--Tate-and-higgs-bundles}.
   
   In order to be able to globalise this, we then explain in \S\ref{sec:globalisation} how $X^\HT$ relates to various stacks of square-zero lifts of $X$. Based on this discussion,  we then prove
    \Cref{sec:smoothoid-case-1-derived-local-p-adic-simpson-correspondence-introduction,sec:arithmetic-case-1-derived-local-p-adic-simpson-in-arithmetic-case}.

\subsection*{Acknowledgements}

We would like to thank Bhargav Bhatt, Hui Gao, Tongmu He, Juan Esteban Rodr\'iguez Camargo, Peter Scholze, Yupeng Wang, Matti W\"urthen and Bogdan Zavyalov for helpful discussions. 
We would like to thank the referee for their close reading and very helpful and detailed comments. 

The second author was funded by the Deutsche Forschungsgemeinschaft (DFG, German Research Foundation) -- Project-ID 444845124 -- TRR 326.

\subsection{Notations}
\label{sec:notations}

We will use the following notations.
\begin{enumerate}
\item Let $p$ be a prime. Let $\Z_p^\cycl:=\Z_p[\zeta_{p^\infty}]^\wedge_p$ and $\Q_p^\cycl:=\Z_p^\cycl\tf$, the cyclotomic perfectoid field. We fix a primitive $p$-th root of unity $\zeta_p \in \Z_p^\cycl$ and write $\omega=(\zeta_p-1)^{-1} \in \Q_p^\cycl$.
  \item If $R_0 \to R$ is a morphism of $p$-complete rings, we denote by $\Omega_{R|R_0}^1$ the \textit{$p$-completion} of the module of differential forms of $R$ over $R_0$. This glues to define, for any morphism $X \to X_0$ of $p$-adic formal schemes, a sheaf $\Omega_{X|X_0}^1$ of $p$-completed differential forms. Its $\mathcal{O}_X$-linear dual, the tangent sheaf, is denoted by $T_{X|X_0}$. We use analogous notation for adic spaces. 
  \item If $X$ is a $p$-adic formal scheme, we denote by $X^{\prism}$ (resp.\ $X^{\rm HT}$) the Cartier--Witt stack of $X$ (resp. the Hodge--Tate locus in the Cartier--Witt stack of $X$, or Hodge--Tate stack of $X$), which is denoted $\WCart_X$ (resp. $\WCart_X^{\rm HT}$) in \cite[Definition 3.1, Construction 3.7]{bhatt2022prismatization}.
  \item For each $(A,I)\in (X)_\prism$ there exists the natural morphism
  \[
  \rho_A=\rho_{A,X}\colon \Spf(A)\to X^\prism,
  \]
  cf.\ \cite[Construction 3.10]{bhatt2022prismatization}, where $X^\prism$ is denoted by $\WCart_X$. We let
  \[
  \overline{\rho_A}\colon \Spf(A/I)\to X^\HT
  \]
  be the natural induced morphism to the Hodge--Tate locus.
  \item For any locally noetherian analytic adic space $\mathcal{X}$ we denote by $\mathcal{O}$ or $\mathcal{O}_{\mathcal{X}_{\mathrm{pro\acute{e}t}}}$ the completed structure sheaf on the pro-\'etale site of $\mathcal{X}$ (which was denoted by $\widehat{\mathcal{O}}$ in \cite{scholze2013p}).
\end{enumerate}

\section{A criterion for fully faithfulness}
\label{sec:form-reduct-prov}

Let $X$ be a bounded\footnote{That is locally $X\cong \Spf(R)$ with $R$ having bounded $p^\infty$-torsion.} $p$-adic formal scheme. In this section, we make the following assumptions:

\begin{enumerate}
	\item $X=\Spf(R)$ is affine and $R\cong A/I$ for some prism $(A,I)$. By \S\ref{sec:notations}.(4) this induces a map
	\[\eta:=\overline{\rho_A}\colon \Spf(R)\cong \Spf(A/I)\to X^\HT.\]
	\item The group scheme $G_A$ of isomorphisms of $\eta$
	is $p$-completely (faithfully) flat over $X$.
	\item There exists a perfect prism $(A_\infty, I_\infty)$ over $(A,I)$ such that $R=A/I\to R_\infty:=A_\infty/I_\infty$ is a quasi-syntomic cover (this condition actually implies that $X$ is a bounded).
\end{enumerate}

Furthermore, we set $X_\infty:=\Spf(R_\infty)$ and let $\Gamma$ be a profinite group with a continuous right-action on $X_\infty$, such that the morphism $f\colon X_\infty\to X$ is $\Gamma$-equivariant for the trivial action on $X$. In all examples we are interested in, the adic generic fibre $\X_\infty\to \X$ of $f$ will be a pro-\'etale $\Gamma$-torsor. 

\medskip

The goal of this section is first to construct a natural functor
\[\calPerf(X^\HT)\tf\to \calPerf([X_\infty/\Gamma])\tf,
\]
and then to derive a useful criterion, \Cref{sec:form-reduct-prov-1-crucial-assumption-on-galois-cohomology}, which guarantees that this is fully faithful. 

\medskip

As $X_\infty$ is perfectoid, we have $X_\infty \cong X_\infty^\HT$ and hence $f\colon X_\infty \to X$ lifts canonically to a morphism $\widetilde{f}\colon X_\infty \to X^\HT$. More precisely, we have $\widetilde{f}=\overline{\rho_{A_\infty}}$. Our assumptions now imply that there exists a $2$-commutative diagram
\[
\begin{tikzcd}[row sep =-0.05cm]
	& X \arrow[dd,"\eta"] \\
	X_\infty \arrow[ru,"f"] \arrow[rd,"\wt f"'] &              \\
	& X^\HT    
\end{tikzcd}\]
and thus the two morphisms $\widetilde{f}$ and $\eta\circ f$ are isomorphic in $X^\HT(X_{\infty})$. We deduce:
\begin{lemma}
	\label{sec:form-reduct-prov-1-consequences-of-axioms}
	\begin{enumerate}
		\item The morphism $\eta\colon X\to X^\HT$ is affine and $p$-completely faithfully flat.
		\item The morphism $\widetilde{f}$ is affine and $p$-completely faithfully flat.
	\end{enumerate}
\end{lemma}
\begin{proof}
	We first show that $\eta:X\to X^\HT$ is a surjection for the flat topology.
	This may be checked for $\widetilde{f}\colon X_{\infty}\to X^\HT$ instead. As $X_{\infty}\to X$ is quasi-syntomic by assumption,  we can deduce from \cite[Lemma 6.3]{bhatt2022prismatization} that $f^\HT:X_\infty^\HT\to X^\HT$ is surjective in the flat topology, and the statement follows because $X_{\infty}\cong X_{\infty}^\HT$. As $G_A$ is $p$-completely faithfully flat over $X$, we can conclude now that the morphism $X\to X^\HT$ is $p$-completely flat. This implies that $\widetilde{f}$ is $p$-completely flat because $\widetilde{f}\cong \eta\circ f$ and $f$ is quasi-syntomic. That $\eta$ is affine follows by base-change from the fact that $\rho_A:\Spf(B)\to \WCart_X$ is affine by \cite[Corollary~3.2.9]{bhatt2022absolute}. This finishes the proof.
\end{proof}

We can conclude that
$B_XG_A:=[X/G_A]\cong X^\HT$
via $\eta$. 
Consider the $n$-fold fibre product (of fpqc-sheaves)
\[
X^{(n)}:={X\times_{X^\HT}\ldots \times_{X^\HT} X}, \quad 
Z^{(n)}:=X_\infty \times_{X^\HT} X^{(n)}
\]
for $n\geq 1$ so that $X=X^{(1)}$. Set $Z:=Z^{(1)}\cong X_\infty\times_{X^\HT} X$.
Note that $Z$ is a $\Gamma$-equivariant $G_A$-torsor over $X_\infty$. Let $\eta_n\colon X^{(n)}\to X^\HT$ be the projection. Then for $n\geq 2$ we have
\[X^{(n)}\cong {G_A\times_X \ldots \times_X G_A},\quad 
Z^{(n)}\cong Z\times_X{G_A\times_X \ldots \times_X G_A},
\]
where the fibre products are $(n-1)$-fold. The second isomorphism is $\Gamma$-equivariant if $\Gamma$ acts  via the natural action on $Z$ and trivially on the each $G_A$.
For $n\geq 1$ we get a Cartesian diagram
\[\begin{tikzcd}
	{Z^{(n)}} & {X^{(n)}} \\
	X_\infty & {X^\HT.}
	\arrow["{h_n}", from=1-1, to=2-1]
	\arrow["{\eta_n}", from=1-2, to=2-2]
	\arrow["{\widetilde{f}}", from=2-1, to=2-2]
	\arrow["{\widetilde{f}_n}", from=1-1, to=1-2]
\end{tikzcd}\]
where $h_n$ is the projection to the first factor.
By \Cref{sec:form-reduct-prov-1-consequences-of-axioms}, $Z=\Spf(B_{A,R_\infty}^+)$ for some $p$-adically complete, $p$-completely faithfully flat $R_\infty$-algebra $B_{A,R_\infty}^+$. There are a right action by $\Gamma$ and a left action by $G_A$ which makes $Z$ into a torsor over $X_\infty=\Spf(R_\infty)$. These actions induce on $B_{A,R_\infty}^+$ a continuous, $R_\infty$-semilinear left action by $\Gamma$, and a right action by $G_A$, and these actions commute.

We also set
\[
B_{A,R_\infty}:= B_{A,R_\infty}^+\tf.
\]
We now make the following crucial assumption in which $R\Gamma(\Gamma,-)$ refers to continuous $\Gamma$-cohomology, defined as the cohomology of the formal stack $[\Spf(\Z_p)/\underline{\Gamma}]$. Equivalently, this can be computed via the bar complex of $\Gamma$ with continuous cochains. In our applications, $\Gamma$ will always be $\cong \Z_p^d$ or a semi-direct thereof, hence for any bounded complex $K$, the cohomology $R\Gamma(\Gamma,K)$ will be bounded.

\begin{assumption}
	\label{sec:form-reduct-prov-1-crucial-assumption-on-galois-cohomology}
	The cofiber of $R\to R\Gamma(\Gamma, B_{A,R_\infty}^+)$ is bounded and killed by $p^i$ for some $i\geq 1$.
\end{assumption}

As the morphism $\widetilde{f}\colon X_\infty \to X^\HT$ is $\Gamma$-equivariant (by naturality of the Hodge--Tate stack), it factors over a map
\[
\beta:=[\widetilde{f}/\Gamma]\colon [X_\infty/\underline{\Gamma}]\to X^\HT
\]
over the stack quotient $[X_\infty/\underline{\Gamma}]$ (for the fpqc-topology). Here, $\underline{\Gamma}$ refers to the affine formal group scheme $\Spf(\mathcal C(\Gamma,R_\infty))$ over $X_\infty$ where $\mathcal C(\Gamma,R_\infty)$ is the $R_\infty$-algebra of continuous maps $\Gamma\to R_\infty$.

\begin{proposition}
	\label{sec:form-reduct-prov-1-pushforward-beta-o}
	Under \Cref{sec:form-reduct-prov-1-crucial-assumption-on-galois-cohomology}, the cofibre of the natural map
	\[
	\mathcal{O}_{X^\HT}\to R\beta_\ast \mathcal{O}_{[X_\infty/\underline{\Gamma}]}
	\]
	is killed by $p^i$ for some $i\geq 1$.
\end{proposition}
\begin{proof}
	Using the \v{C}ech nerve for the affine covering $[Z/\underline{\Gamma}]\to [X_\infty/\underline{\Gamma}]$ we see that
	\[\textstyle
	\mathcal{O}_{[X_\infty/\underline{\Gamma}]}\cong \varprojlim_\Delta [h_{n}/\underline{\Gamma}]_{\ast}\mathcal{O}_{[Z^{(n)}/\underline{\Gamma}]}.
	\]
	\[\Rightarrow \textstyle
	R\beta_\ast(\mathcal{O}_{[X_\infty/\underline{\Gamma}]})=\varprojlim_\Delta R\beta_{\ast}[h_{n}/\underline{\Gamma}]_\ast(\mathcal{O}_{[Z^{(n)}/\underline{\Gamma}]}).
	\]
	We have a commutative diagram
	\[\begin{tikzcd}
		{[Z^{(n)}/\underline{\Gamma}]} & {X^{(n)}} \\
		{[X_\infty/\underline{\Gamma}]} & {X^\HT}
		\arrow["{[h_n/\underline{\Gamma}]}"', from=1-1, to=2-1]
		\arrow["\beta", from=2-1, to=2-2]
		\arrow["{\beta_n}", from=1-1, to=1-2]
		\arrow["{\eta_n}", from=1-2, to=2-2]
	\end{tikzcd}\]
	and thus
	\[
	R\beta_\ast[h_{n}/\underline{\Gamma}]_\ast(\mathcal{O}_{[Z^{(n)}/\underline{\Gamma}]})\cong \eta_{n,\ast}R\beta_{n,\ast}(\mathcal{O}_{[Z^{(n)}/\underline{\Gamma}]}),
	\]
	where we used that $\eta_n$ is affine.
	Now we use that
	\[
	Z^{(n)}\cong Z\times_X X^{(n)}\cong \Spf(B_{A,R_\infty}^+\widehat{\otimes}_{R} \mathcal{O}(X^{(n)})).
	\]
	Since $X\to X^\HT$ is $p$-completely faithfully flat by \Cref{sec:form-reduct-prov-1-consequences-of-axioms}, the same is true for $X^{(n)}\to X$. We may thus apply \Cref{sec:crit-fully-faithf-projection-formula-for-continuous-group-cohomology} below to deduce:
	\[
	R\beta_{n,\ast}(\mathcal{O}_{[Z^{(n)}/\underline{\Gamma}]})= R\Gamma(\Gamma, \mathcal{O}_{Z^{(n)}})\cong R\Gamma(\Gamma, B_{A,R_\infty}^+)\widehat{\otimes}_R^L \mathcal{O}_{X^{(n)}}.
	\]

	It follows from this that if we denote by $K$ the object of $D^b(X^{\HT})=D^b([X/G_A])=D^b_{G_A}(R)$ associated to $R\Gamma(\Gamma, B_{A,R_\infty}^+)$ with its $G_A$-action induced from the $G_A$-action on  $B_{A,R_\infty}^+$, then
	\[\textstyle R\beta_{\ast}\O_{[X_\infty/\underline\Gamma]}\cong \varprojlim_{\Delta} \eta_{n,\ast}R\beta_{n,\ast}(\mathcal{O}_{[Z^{(n)}/\underline{\Gamma}]})=R\Gamma(G_A,K)\]
	where the latter is the ``group cohomology of $G_A$'' computed by the bar complex of $G_A$. In order to show that the cofiber of the natural map from
	\[
\textstyle	\mathcal{O}_{X^\HT}\cong \varprojlim_{\Delta} \eta_{n,\ast}\mathcal{O}_{X^{(n)}}\cong R\Gamma(G_A,\O_X)\]
	is killed by $p^j$ for some $j\geq 0$, it therefore suffices to show that the natural map $\O_{X}\to K$ has $p^j$-torsion cofiber $K'$ inside $D^b_{G_A}(R)$. To see this, we invoke \Cref{sec:form-reduct-prov-1-crucial-assumption-on-galois-cohomology} to conclude that $p^i=0$ on the cohomology groups of $K'$. Since $K'$ is bounded, it follows that there is some $j\geq 0$ such that $p^j=0$ on $K'$: For example, this can be seen from $\mathrm{RHom}(K',K')\otimes_{\Z}{}\Z[\tfrac{1}{p}]= \mathrm{RHom}(K',K'\otimes_{\Z}{}\Z[\tfrac{1}{p}])=0$.
\end{proof}

In the proof, we have used the following lemma, which is essentially a projection formula.

\begin{lemma}
	\label{sec:crit-fully-faithf-projection-formula-for-continuous-group-cohomology}
	For any $M\in \mathcal{D}^b([\Spf(R)/\underline{\Gamma}])$ and any $p$-complete and $p$-completely flat $R$-module $N$, the following natural map is an isomorphism:
	\[
	N\widehat{\otimes}_R^L R\Gamma(\Gamma,M)\to R\Gamma(\Gamma, N\widehat{\otimes}_R^L M).
	\]
\end{lemma}
Here we see $\underline{\Gamma}$ as the pro-finite group scheme $\Spf(\mathcal C(\Gamma,R))$ relatively over $R$, and we equip $N\widehat{\otimes}_R^L M$ with its $\Gamma$-action through the second factor.
\begin{proof}
	The continuous group cohomology $R\Gamma(\Gamma,M)$ can be calculated via the double complex
	\[
	M\to M\widehat{\otimes}_R^L \mathcal C(\Gamma,R)\to M\widehat{\otimes}_R^L \mathcal C(\Gamma,R)\widehat{\otimes}_R^L \mathcal C(\Gamma,R)\to \ldots
	\]
	with uniformly bounded columns. Similarly, $R\Gamma(\Gamma, N\widehat{\otimes}_R^L M)$ is calculated by the double complex
	\[
	N\widehat{\otimes}^L_RM\to N\widehat{\otimes}_R^L M\widehat{\otimes}_R^L \mathcal C(\Gamma,R)\to N\widehat{\otimes}_R^LM\widehat{\otimes}_R^L \mathcal C(\Gamma,R)\widehat{\otimes}_R^L \mathcal C(\Gamma,R)\to \ldots.
	\]
	which has again uniformly bounded columns by $p$-completeness and $p$-complete flatness of $N$. Thus, for a fixed cohomological degree, only finitely many columns contribute, and thus the statement follows from exactness of the functor $N\widehat{\otimes}_R(-)$ on $p$-complete complexes of $R$-modules.
\end{proof}

\begin{corollary}
	\label{sec:form-reduct-prov-1-cohomology-agrees-up-to-p-torsion}
	Let $\mathcal{E}\in \calPerf(X^\HT)$ be a perfect complex.
	Then under \Cref{sec:form-reduct-prov-1-crucial-assumption-on-galois-cohomology}, the map
	\[
	R\Gamma(X^\HT,\mathcal{E})\to R\Gamma([X_\infty/\underline{\Gamma}],\beta^\ast \mathcal{E})
	\]
	has cofiber killed by $p^i$ for some $i\geq 1$.
\end{corollary}
\begin{proof}
	By the projection formula, we have
	$
	R\beta_\ast(\beta^\ast\mathcal{E})\cong \mathcal{E}\widehat{\otimes}_{\mathcal{O}_{X^\HT}}^LR\beta_\ast(\mathcal{O}_{[X_\infty/\underline{\Gamma}]}).
	$
	Let $K$ be the cofiber of $\mathcal{O}_{X^\HT}\to R\beta_{\ast}(\mathcal{O}_{[X_\infty/\underline{\Gamma}]})$, then it suffices to see that
	$
	R\Gamma(X^\HT,\mathcal{E}\widehat{\otimes}_{\mathcal{O}_{X^\HT}}^L K)$
	is killed by $p^i$.
	But this follows from our assumption by \Cref{sec:form-reduct-prov-1-pushforward-beta-o}.
\end{proof}

\begin{proposition}
	\label{sec:form-reduct-prov-1-fully-faithfulness-for-beta}
	Under \Cref{sec:form-reduct-prov-1-crucial-assumption-on-galois-cohomology}, the following functor 
	is fully faithful:
	\[
	\beta^\ast\tf\colon \mathcal{P}erf(X^\HT)\tf\to \mathcal{P}erf([X_\infty/\underline{\Gamma}])\tf
	\]
\end{proposition}
Here, the $(-)\tf$ refers to the isogeny category, i.e., the $\infty$-category with the same objects but $R\Hom$-complexes tensored with $\Q_p$. 
\begin{proof}
	This follows from \Cref{sec:form-reduct-prov-1-cohomology-agrees-up-to-p-torsion} because perfect complexes are dualizable and hence morphisms between them can be calculated via cohomology of the (again dualizable) derived internal Hom.
\end{proof}

Next, we reinterpret \Cref{sec:form-reduct-prov-1-fully-faithfulness-for-beta} via the v-site $\X_v$ of the adic generic fibre $\X$ of $X$. Recall that $\X_v$ is the category of analytic perfectoid spaces $T$ over $\X$ endowed with v-topology, i.e.\ the Grothendieck topology generated by surjections between affinoid perfectoid spaces.
Let $\X_{v,\mathrm{affd}}$ be the full subcategory with $T$ affinoid perfectoid. Note that this does not change the v-topos of $\X$.

If $T=\Spa(B,B^+)\in \X_{v}$ is affinoid perfectoid, then $B^+$ is a perfectoid ring and the natural morphism $\Spf(B^+)\to \Spf(R)$ lifts canonically to a morphism $\Spf(B^+)\to X^\HT$. Pulling back a perfect complex along this morphism and then tensoring with $B$ (i.e.\ inverting $p$) yields a functor
\[
\alpha^{+,\ast}_X\colon \mathcal{P}erf(X^\HT)\to \varprojlim\limits_{\Spa(B,B^+)\in \X_{v,\mathrm{affd}}} \mathcal{P}erf(B). 
\]
By v-descent for perfect complexes on perfectoid spaces, cf.\ \cite[Theorem 2.1]{anschutz2021fourier}, the inverse limit identifies with the category of perfect complexes on the ringed site $(\X_v,\O)$ for the structure sheaf $\O$ defined by $T\mapsto \O(T)$. Formally inverting $p$, i.e.\ passing to isogeny categories, yields a functor
\[
\alpha^\ast_X\colon \mathcal{P}erf(X^\HT)\tf\to \mathcal{P}erf(\X_v).
\]

\begin{proposition}
	\label{sec:form-reduct-prov-1-fully-faithfulness-for-alpha}
	Assume that the $\Gamma$-action on $X_\infty$ makes the adic generic fibre $\X_\infty\to \X$ into a pro-\'etale $\Gamma$-torsor and that \Cref{sec:form-reduct-prov-1-crucial-assumption-on-galois-cohomology} holds. Then 
	$\alpha_X^\ast$
	is fully faithful.
\end{proposition}
\begin{proof}
	Let $S_\infty:=\O(\X_\infty)$ and $S_\infty^+=\O^+(\X_\infty)$.
	By v-descent of perfect complexes on perfectoid spaces it suffices to show that the functor from perfect complexes on $X^\HT$ to the category $\mathcal{P}erf^\Gamma(S_\infty)$ of ``perfect complexes over $S_\infty$ with a continuous $\Gamma$-action'' 
        is fully faithful. Here we define a perfect complex over $S_\infty$ with continuous $\Gamma$-action  to be a cartesian perfect complex on the \v{C}ech nerve $\X_\infty\times \underline{\Gamma}^\bullet$.
        
	The functor
	\[
	\mathcal{P}erf^\Gamma(S^+_\infty)\tf\to \mathcal{P}erf^\Gamma(S_\infty)
	\]
	is fully faithful by an argument similar to the proof of \Cref{sec:form-reduct-prov-1-fully-faithfulness-for-beta}.	Now $R_\infty\to S_\infty^+$ has cofiber killed by some $p^n$ as the $p^\infty$-torsion in $R_\infty$ is bounded and both $R_\infty/\mathrm{torsion}$ and $S_\infty^+$ define lattices in the Banach space $S_\infty$. This implies that the functor
	\[
	\mathcal{P}erf^\Gamma(R_\infty)\tf\to \mathcal{P}erf^\Gamma(S_\infty^+)\tf
	\]
	is fully faithful. Now, clearly the functor $\mathcal{P}erf(X^\HT)\tf\to \mathcal{P}erf^\Gamma(S_\infty)$ factors over the functor $\beta^\ast\tf\colon \mathcal{P}erf(X^\HT)\tf \to \mathcal{P}erf^\Gamma(R_\infty)\tf$.
	By \Cref{sec:form-reduct-prov-1-fully-faithfulness-for-beta} we can conclude.
\end{proof}

\begin{remark}
\label{rk:no-obvious-proof-by-descent}
Since quasi-syntomic covers $X \to Y$ of bounded $p$-adic formal schemes give rise to surjections $X^\HT \to Y^\HT$ for the flat topology (\cite[Lemma 6.3]{bhatt2022prismatization}), one may hope to prove \Cref{sec:introduction-2-statement-main-theorem} by quasi-syntomic descent. But if $X=\mathrm{Spf}(R)$ is quasi-regular semiperfectoid, $X^\HT=\mathrm{Spf}(\overline{\prism_R})$, while the diamond of $\X$ is $\mathrm{Spd}(R_{\mathrm{perfd}}\tf,R_{\mathrm{perfd}})$. Hence the functor $\alpha_X^\ast$ will not be fully faithful for such an $X$ in general.\footnote{If $R$ is quasi-regular semi-perfectoid, $p$-torsionfree and non-reduced, e.g.\ $R=\mathcal{O}_{\mathbb{C}_p}\langle T^{1/p^\infty}\rangle/(T)$, then $\overline{\prism_R}\tf$ is non-reduced as well. On the other hand, $R_{\mathrm{perfd}}\tf$ is reduced, which makes it impossible for $\alpha_X^\ast$ to be fully faithful.} This is why we will rather prove fully faithfulness in the smoothoid and arithmetic cases via \Cref{sec:form-reduct-prov-1-fully-faithfulness-for-alpha}, by verifying the assumption that the (higher) Galois cohomology of a certain geometrically defined period ring, namely $B_{A,R_\infty}=B_{A,R_\infty}^+\tf$, is trivial.  
\end{remark}

In order to verify that \Cref{sec:form-reduct-prov-1-crucial-assumption-on-galois-cohomology} holds in all situations of interest to us, we will use stability properties of \Cref{sec:form-reduct-prov-1-fully-faithfulness-for-alpha} resp.\ \Cref{sec:form-reduct-prov-1-crucial-assumption-on-galois-cohomology} under ind-$p$-completely \'etale extensions $R\to R^\prime$ which we now discuss. Since $\delta$-structures lift uniquely to ind-\'etale extensions by \cite[Lemma 2.18]{Bhatta}, $R\to R^\prime$ lifts uniquely to a morphism $(A,I)\to (A^\prime,I^\prime)$ of prisms. We can set $A^\prime_\infty:= A_\infty\widehat{\otimes}_AA^\prime$ and $I^\prime_\infty:=I_\infty A_\infty^\prime$ with its induced continuous $\Gamma$-action. Note that $R^\prime_\infty$ is again perfectoid.

\begin{lemma}
	\label{sec:crit-fully-faithf-assumption-passes-to-ind-etale-extensions}
	If \Cref{sec:form-reduct-prov-1-crucial-assumption-on-galois-cohomology} holds for $(R,A,R_\infty, A_\infty, \Gamma)$, then it holds for $(R^\prime, A^\prime, R^\prime_\infty, A_\infty^\prime, \Gamma)$.
\end{lemma}
\begin{proof}
	As $X^\prime:=\Spf(R^\prime)\to X$ is ind-$p$-complete \'etale, the natural map $X^{\prime,\HT}\to X^\HT\times_X X^\prime$ is an isomorphism, cf.\ \cite[Remark 3.9]{bhatt2022prismatization} or \Cref{sec:crit-fully-faithf-relative-ht-map} below. This implies the claim by \Cref{sec:crit-fully-faithf-projection-formula-for-continuous-group-cohomology}.
\end{proof}

\begin{lemma}
  \label{sec:crit-fully-faithf-relative-ht-map}
  Let $Z\to Y$ be a smooth map of bounded $p$-adic formal schemes. Then
  \[
    \pi_{Z|Y}\colon Z^\HT\to Y^\HT\times_Y Z 
  \]
  is a gerbe banded by the flat affine formal group scheme $\mathcal{T}_{Z|Y}^\sharp\{1\}$ given by the twist of $\mathcal{T}_{Z|Y}^\sharp:=\Spf(\Gamma_{\mathcal{O}_Z}^\bullet(\Omega^1_{Z|Y})^{\wedge}_p)$ by the pullback of the line bundle $\mathcal{O}_{\Spf(\Z_p)^\HT}\{1\}$ from \cite[Example 3.5.2]{bhatt2022absolute}.
  \end{lemma}
\begin{proof}
  The argument of \cite[Proposition 5.12]{bhatt2022prismatization} applies here as well.
\end{proof}

There is another stability property of \Cref{sec:form-reduct-prov-1-crucial-assumption-on-galois-cohomology}.
Namely, assume that there exists a perfect prism $(A_0,I_0)$ with a morphism $(A_0,I_0)\to (A,I)$ such that the $\Gamma$-action on $A_\infty$ is $A_0$-linear. Now let $(A_0,I_0)\to (B_0,J_0)$ be a morphism with $(B_0,J_0)$ another perfect prism.

\begin{lemma}
	\label{sec:crit-fully-faithf-base-change-along-morph-of-perfect-prisms}
	Assume \Cref{sec:form-reduct-prov-1-crucial-assumption-on-galois-cohomology} holds for $(R,A,R_\infty, A_\infty, \Gamma)$. Assume that the natural map
	\[
	(X\times_{\Spf(A_0/I_0)} \Spf(B_0/J_0))^\HT\to X^\HT\times_{\Spf(A_0/I_0)} \Spf(B_0/J_0) 
	\]
	is an isomorphism, for example this happens if $X\to \Spf(A_0/I_0)$ is $p$-completely smooth. Then \Cref{sec:form-reduct-prov-1-crucial-assumption-on-galois-cohomology} holds for the data $(R\widehat{\otimes}_{A_0/I_0} B_0/J_0,A\widehat{\otimes}_{A_0} B_0,R_\infty\widehat{\otimes}_{A_0/I_0} B_0/J_0, A_\infty\widehat{\otimes}_{A_0}B_0, \Gamma)$. If $(A_0,I_0)\to (B_0,J_0)$ is $p$-completely faithfully flat, then the converse holds.
\end{lemma}
\begin{proof}
	This follows again from \Cref{sec:crit-fully-faithf-projection-formula-for-continuous-group-cohomology}.  
\end{proof}

Lastly, we explain how \Cref{sec:introduction-2-statement-main-theorem} can be localized on $X$.

\begin{lemma}
	\label{sec:crit-fully-faithf-main-theorem-localizes-on-x}
	In the setup of \Cref{sec:introduction-2-statement-main-theorem} assume that $X=\cup_{i=1}^nX_i$ with $X_i\subseteq X$ affine open. Assume that \Cref{sec:introduction-2-statement-main-theorem} holds for $X_i$ and all their intersections. Then \Cref{sec:introduction-2-statement-main-theorem} holds for $X$.
\end{lemma}
\begin{proof}
	Let $\mathcal{E},\mathcal{F}\in \mathcal{P}erf(X^\HT)$. We have to check that the morphism
	\[
	R\Hom_{\mathcal{P}erf(X^\HT)}(\mathcal{E},\mathcal{F})\to R\Hom_{\mathcal{P}erf(\X_v)}(\alpha^{+,\ast}_X\mathcal{E}, \alpha^{+,\ast}_X\mathcal{F}) 
	\]
	has cofiber $K$ killed by $p^m$ for some $m\geq 1$. As $X$ is qcqs this can be checked after pullback along $\coprod_{i=1}^n X_i\to X$. Namely, $K$ can be calculated as a finite limit of the cofibers $K_Y$ of
	\[
	R\Hom_{\mathcal{P}erf(Y^\HT)}(\mathcal{E},\mathcal{F})\to R\Hom_{\mathcal{P}erf(Y^\rig_v)}(\alpha^{+,\ast}_{Y}\mathcal{E}, \alpha^{+,\ast}_{Y}\mathcal{F}) ,
	\]
	where $Y$ runs through the finite intersections of the $X_i$'s. We may choose $m\geq 1$ large enough such that $p^mK_Y=0$ for all $Y$. Now a finite limit in $\mathcal{D}(\Z)$ of complexes of $\Z/p^m$-modules is a complex of $\Z/{p^m}^\prime$-modules for some $m^\prime\geq m$ as follows by considering finite products and fiber sequences.
\end{proof}

\section{Explicit naturality for the prismatization}
\label{sec:prismatic-homotopies}
The goal of this section is to make explicit some group actions on the Hodge--Tate stack coming from its functorial nature, which will be important for checking \Cref{sec:form-reduct-prov-1-crucial-assumption-on-galois-cohomology} in practice. This section expands results from \cite[Section 3.1]{analytic_HT} and \cite[Section 9]{bhatt2022prismatization}. 

\subsection{Explicit naturality}
\label{sec:explicit-naturality}

Let $R$ be a $p$-complete ring with bounded $p^\infty$-torsion and $X:=\Spf(R)$. Let $(A,I)\in (X)_\prism$ and consider the morphism 
from \S\ref{sec:notations}.(4)
  \[
\overline{\rho_A}\colon \Spf(A/I)\to X^\HT.
\]
Our aim in this section is to make explicit how this behaves with respect to morphisms in $(X)_\prism$.   

In order to make certain morphisms of animated rings concrete, we fix a surjection $F\to R$ by a free polynomial algebra. Let $\mathfrak{a}:=\ker(F\to R)$.
For $(A,I)\in (X)_\prism$ we assume that there exists a lift of the structure morphism $\iota_A\colon R\to A/I$ to a morphism $\widetilde{\iota_A}\colon F\to A$ of rings. We can then make $\rho_A$ explicit: let $f\colon A\to S$ be a morphism with $S$ a $p$-complete ring and $g\colon A\to W(S)$ the lift induced by the $\delta$-structure. The base change of $I\to A$ along $g$ defines the Cartier-Witt divisor
\[
(I\otimes_{A,g}W(S)\xrightarrow{\alpha} W(S))
\]
for $S$, and together with the natural composition $R\xrightarrow{\iota_A}A/I\xrightarrow{\overline{g}} \cone(\alpha)$ this yields the point $\rho_A(f)\in X^\prism(S)$. More explicitly, the composition $\overline{g}\circ \iota_A$ can be represented by the diagram

\[\begin{tikzcd}
	{\mathfrak{a}} & I & {I\otimes_{A,g} W(S)} \\
	F & A & {W(S),}
	\arrow["{\widetilde{\iota_A}}", from=1-1, to=1-2]
	\arrow["x\mapsto x\otimes 1", from=1-2, to=1-3]
	\arrow["\alpha", from=1-3, to=2-3]
	\arrow["g", from=2-2, to=2-3]
	\arrow["{\widetilde{\iota_A}}", from=2-1, to=2-2]
	\arrow[from=1-1, to=2-1]
	\arrow[from=1-2, to=2-2]
\end{tikzcd}\]
where each column represents the respective animated ring.

\begin{construction}\label{constr:explicit-2-arrow-in-X^prism}
	Let $\Varphi\colon (A,I)\to (B,J)$ be a morphism in $(X)_\prism$, then the diagram
\[
\begin{tikzcd}[row sep = 0cm,column sep = 0.5cm]
	\Spf(B) \arrow[rr,"\Varphi^\ast"] \arrow[rd,"{\rho_B}"'] &          & \Spf(A) \arrow[ld,"{\rho_A}"]\\
	& X^\prism &                   
\end{tikzcd}\]
	commutes naturally, i.e., in the groupoid $X^\prism(B)$ there exists a natural isomorphism
	\[
	\can_\Varphi\colon \rho_A\circ \Varphi^\ast\to \rho_B.
	\]
	We now make $\can_\Varphi$ explicit:
	We may assume that $R\to B/J$ has a lift $\widetilde{\iota_B}:F\to B$. Let $\overline{\Varphi}\colon A/I\to B/J$ be the reduction of $\Varphi$, then the requirement that $\overline{\Varphi}$ is a morphism of $R$-algebras yields the homotopy
\[
h_\Varphi\colon F\to J,\ s\mapsto \Varphi\circ \widetilde{\iota_A}(s)-\widetilde{\iota_B}(s).
\]
If $f\colon B\to S$ is a morphism and $g\colon B\to W(S)$ its natural lift, then the point $\rho_A(\Varphi^\ast(f))=\rho_A(f\circ \Varphi)\in X^\prism(S)$ can be represented by the commutative diagram
\[\begin{tikzcd}
	{\mathfrak{a}} & I & {I\otimes_{A,g\circ \Varphi} W(S)} \\
	F & A & {W(S),}
	\arrow["{\widetilde{\iota_A}}", from=1-1, to=1-2]
	\arrow["x\mapsto x\otimes 1", from=1-2, to=1-3]
	\arrow["\alpha", from=1-3, to=2-3]
	\arrow["g\circ \Varphi", from=2-2, to=2-3]
	\arrow["{\widetilde{\iota_A}}", from=2-1, to=2-2]
	\arrow[from=1-1, to=2-1]
	\arrow[from=1-2, to=2-2]
\end{tikzcd}\]
where the implicit Cartier-Witt divisor is the rightmost column.
On the other hand, the element $\rho_B(f)\in X^\prism(S)$ can be represented by the diagram
\[\begin{tikzcd}
	{\mathfrak{a}} & J & {J\otimes_{B,g} W(S)} \\
	F & B & {W(S).}
	\arrow["{\widetilde{\iota_B}}", from=1-1, to=1-2]
	\arrow["x\mapsto x\otimes 1", from=1-2, to=1-3]
	\arrow["\beta", from=1-3, to=2-3]
	\arrow["g", from=2-2, to=2-3]
	\arrow["{\widetilde{\iota_B}}", from=2-1, to=2-2]
	\arrow[from=1-1, to=2-1]
	\arrow[from=1-2, to=2-2]
\end{tikzcd}\]
Define the isomorphism
\[
\Phi_\Varphi\colon I\otimes_{A,g\circ \Varphi} W(S)\isomarrow J\otimes_{B,g}W(S),\ i\otimes w\mapsto \Varphi(i)\otimes w,
\]
which yields the isomorphism 
of Cartier-Witt divisors
\[\begin{tikzcd}
	{I\otimes_{A,g\circ \Varphi}W(S)} & {J\otimes_{B,g} W(S)} \\
	{W(S)} & {W(S).}
	\arrow["{\Phi_\Varphi}", from=1-1, to=1-2]
	\arrow["\beta", from=1-2, to=2-2]
	\arrow["\alpha", from=1-1, to=2-1]
	\arrow["{\Id}", from=2-1, to=2-2]
\end{tikzcd}\]
This defines the first part of data for $\can_\Varphi$. Let
$
\Phi^\prime_\Varphi\colon \cone(\alpha)\to \cone(\beta)
$
be the induced isomorphism of animated rings. The second datum is the isomorphism of the two morphisms of animated rings
\[
R\xrightarrow{\iota_A} A/I\xrightarrow{\overline{g\circ \Varphi}} \cone(\alpha)\xrightarrow{\Phi^\prime_\Varphi} \cone(\beta)\quad
\text{and}
\quad R\xrightarrow{\iota_B} B/J\xrightarrow{\overline{g}} \cone(\beta)
\]
constructed as follows: We combine the morphism $(R\xrightarrow{\iota_A} A/I) \to (R\xrightarrow{\iota_B} B/J)$
 coming from the equality $\iota_B=\Varphi\circ \iota_A$ (of morphisms of usual rings) with the natural isomorphism between $A/I\xrightarrow{\overline{g\circ\Varphi}} \cone(\alpha)\xrightarrow{\Phi^\prime_\Varphi} \cone(\beta)$ and $\ A/I\xrightarrow{\Varphi} B/J\xrightarrow{\overline{g}} \cone(\beta)$ witnessed by the following cube:
\begin{equation}
\begin{tikzcd}[column sep ={1cm,between origins},row sep ={0.60cm,between origins}]
	& {I\otimes_{A,g\circ \Varphi} W(S)} \arrow[ddd,"\alpha"'{yshift=-5pt}] \arrow[rrr,"{\Phi_\Varphi}"] &  &                           & {J\otimes_{B,g} W(S)} \arrow[ddd,"\beta"{yshift=-5pt}] \\
	{I} \arrow[rrr,crossing over] \arrow[ru]             &                            &  & {J} \arrow[ru]             &                \\
	&                            &  &                           &                \\
	& {W(S)} \arrow[rrr,"\Id"]             &  &                           & {W(S)}             \\
	{A} \arrow[from=uuu] \arrow[ru] \arrow[rrr,"\Varphi"'] &                            &  & {B} \arrow[ru,"g"'] \arrow[from=uuu,crossing over] &               
\end{tikzcd}
\end{equation}
Explicitly, we seek a homotopy for the morphisms of complexes given by the outer columns of
\[\begin{tikzcd}
	{\mathfrak{a}} & I & {I\otimes_{A,g\circ \Varphi} W(S)} & {J\otimes_{B,g}W(S)} \\
	F & A & {W(S)} & {W(S)}
	\arrow["{\widetilde{\iota_A}}", from=1-1, to=1-2]
	\arrow["{g\circ \Varphi}", from=1-2, to=1-3]
	\arrow["{\Phi_\Varphi}", from=1-3, to=1-4]
	\arrow[from=1-1, to=2-1]
	\arrow[from=1-2, to=2-2]
	\arrow["{\widetilde{\iota_A}}", from=2-1, to=2-2]
	\arrow["\alpha", from=1-3, to=2-3]
	\arrow["{g\circ \Varphi}", from=2-2, to=2-3]
	\arrow["{\Id}", from=2-3, to=2-4]
	\arrow["\beta", from=1-4, to=2-4]
\end{tikzcd}
\quad \text{and} \quad\begin{tikzcd}
	{\mathfrak{a}} & J & {J\otimes_{B,g}W(S)} \\
	F & B & {W(S).}
	\arrow["\widetilde{\iota_B}", from=1-1, to=1-2]
	\arrow["g", from=1-2, to=1-3]
	\arrow[from=1-1, to=2-1]
	\arrow[from=1-2, to=2-2]
	\arrow["\beta", from=1-3, to=2-3]
	\arrow["{\widetilde{\iota_B}}", from=2-1, to=2-2]
	\arrow["g", from=2-2, to=2-3]
\end{tikzcd}\]
Now the homotopy
\begin{equation}
	\label{eq:1-homotopy-h-varphi}
	h_\Varphi^\prime\colon F\to J\otimes_{B,g}W(S),\ x\mapsto h_\Varphi(x)\otimes 1
\end{equation}    
does what we want as $h_\Varphi$ witnesses the equality $\iota_B=\Varphi\circ \iota_A$. We can now set $\can_{\Varphi}=(\Phi_\Varphi,h_\Varphi^\prime)$. \qed
\end{construction}
\begin{definition}
If $\Varphi,\Varpsi\colon (A,I)\to (B,J)$ are two morphisms in $(X)_\prism$, then from the diagram
\[
\begin{tikzcd}[column sep =0.8cm, row sep =0.1cm]
		& \Spf(B) \arrow[ld, "\Varpsi^\ast"'] \arrow[dd, "\rho_B" description] \arrow[rd, "\Varphi^{\ast}"] &                                                    \\
		\Spf(A) \arrow[rd, "\rho_A"'] \arrow[r, Rightarrow,"\can_{\Varpsi}",shorten >=1.5ex,shorten <=1.0ex] & {}                                                                                             & \Spf(A) \arrow[ld, "\rho_A"] \arrow[l, Rightarrow,"\can_{\Varphi}"',shorten >=1.5ex,,shorten <=1.0ex] \\
		& {X^\prism}                                                                                             &      
\end{tikzcd}\]
we obtain a natural isomorphism
\[
\gamma_{\Varphi,\Varpsi}:=(\Phi_{\Varphi,\Varpsi},h_{\Varphi,\Varpsi}^\prime):=\can_{\Varpsi}^{-1}\circ \can_\Varphi\colon \rho_A\circ \Varphi^*\to \rho_A\circ \Varpsi^*.
\]
\end{definition}
Using \Cref{constr:explicit-2-arrow-in-X^prism}, we can make $\gamma_{\Varphi,\Varpsi}$ explicit. For this let $f\colon B\to S$ be a morphism to a $p$-complete ring $S$, and $g\colon B\to W(S)$ its natural lift. 
The part of $\gamma_{\Varphi,\Varpsi}$ (more precisely, its pullback along $\Spf(S)\to \Spf(B/J)$) acting on the Cartier-Witt divisor is induced by the composition
\[
\Phi_{\Varphi,\Varpsi}=\Phi_{\Varpsi}^{-1}\circ \Phi_\Varphi\colon I\otimes_{A,g\circ \Varphi} W(S)\xrightarrow{\Phi_\Varphi} J\otimes_{B,g}W(S)\xrightarrow{\Phi^{-1}_{\Varpsi}} I\otimes_{A,g\circ \Varpsi} W(S),
\]
yielding the isomorphism
\[\begin{tikzcd}
	{I\otimes_{A,g\circ \Varphi}W(S)} & {I\otimes_{A,g\circ \Varpsi} W(S)} \\
	{W(S)} & {W(S)}
	\arrow["{\Phi_{\Varphi,\Varpsi}}", from=1-1, to=1-2]
	\arrow["\alpha"', from=1-1, to=2-1]
	\arrow["{\Id}", from=2-1, to=2-2]
	\arrow["{\alpha^\prime}", from=1-2, to=2-2]
\end{tikzcd}\]
of Cartier-Witt divisors. We get the induced isomorphism
\[
\Phi^\prime_{\Varphi,\Varpsi}\colon \cone(\alpha)\isomarrow \cone(\alpha^\prime)
\]
of animated rings. The following is the main computation of this subsection:

\begin{lemma}\label{sec:prismatic-homotopies-1-action-on-cw-divisor-more-explicit-in-principal-case}
	Assume that $I$ is generated by some distinguished element $\xi\in A$, that  $f\colon B\to S$ factors over $B/J$, and that $\overline{\Varphi}=\overline{\Varpsi}$ agree as maps $A/I\to B/J$. Set  \[\textstyle u_{\Varphi,\Varpsi,\xi}:=\frac{\Varphi(\xi)}{\Varpsi(\xi)}\in B^\times.\]
	Then
	$\gamma_{\Varphi,\Varpsi}=(\Phi_{\Varphi,\Varpsi},h_{\Varphi,\Varpsi}^\prime):\rho_A\circ \Varphi^*\to \rho_A\circ \Varpsi^\ast$ is given explicitly by the isomorphism
	\[ \Phi_{\Varphi,\Varpsi}: {I\otimes_{A,g}W(S)}\to {I\otimes_{A,g\circ \Varpsi} W(S)},\quad a \xi\otimes w \mapsto a\xi\otimes g(u_{\Varphi,\Varpsi,\xi})w \]
	and the homotopy
	\[\textstyle h_{\Varphi,\Varpsi}^\prime\colon F\to J\otimes_{B,g}W(S),\ x\mapsto \xi\otimes g\big( \frac{\Varphi(\widetilde{\iota_A}(x))-\Varpsi(\widetilde{\iota_A}(x))}{\Varpsi(\xi)}\big).\]
\end{lemma}
\begin{proof}
	Since $I=\xi A$, we know that $J=\Varphi(\xi)B=\Varpsi(\xi)B$ and $\Varphi(\xi)$, $\Varpsi(\xi)$ are distinguished elements as $\Varphi,\Varpsi$ are morphisms of prisms. The assumptions that $f$ factors over $B/J$ and $\overline{\Varphi}=\overline{\Varpsi}$ imply that $g\circ \Varphi=g\circ \Varpsi$. We can therefore compute $\Phi_{\Varphi,\Varpsi}=\Phi_{\Varpsi^{-1}}\circ \Phi_\Varphi$ as follows: For $a\in A, w\in W(S)$,
	\begin{alignat*}{2}
		\Phi^{-1}_{\Varpsi}\circ \Phi_\Varphi(a \xi\otimes w) =&  \Phi^{-1}_{\Varpsi}(\Varphi(a\xi)\otimes w)=  \Phi^{-1}_{\Varpsi}(\Varphi(\xi)\otimes g(\Varphi(a))w) =  \Phi^{-1}_{\Varpsi}(\Varpsi(\xi)\otimes g(\Varpsi(a))g(u_{\Varphi,\Varpsi,\xi})w)\\
		=&  \Phi^{-1}_{\Varpsi}(\Varpsi(a)\Varpsi(\xi)\otimes g(u_{\Varphi,\Varpsi,\xi})w)
		=  a\xi\otimes g(u_{\Varphi,\Varpsi,\xi})w
	\end{alignat*}
	as desired. 
The second part of data of $\gamma_{\Varphi,\Varpsi}$ is an isomorphism of the two compositions
\begin{equation}
	\label{eq:2-first-composition-alpha}
R\xrightarrow{\iota_A} A/I \xrightarrow{g\circ \Varphi} \cone(\alpha)\xrightarrow{\Phi^\prime_{\Varphi,\Varpsi}}\cone(\alpha^\prime)\quad \text{and} \quad
R\xrightarrow{\iota_A} A/I \xrightarrow{g\circ \Varpsi} \cone(\alpha^\prime).
\end{equation}
In order to compute this, we first consider the isomorphism between the compositions
\begin{equation}
	\label{eq:1-first-composition}
	R\xrightarrow{\iota_A} A/I \xrightarrow{g\circ \Varphi} \cone(\alpha)\xrightarrow{\Phi^\prime_{\Varphi}}\cone(\beta) \quad \text{ and }\quad 
	R\xrightarrow{\iota_A} A/I \xrightarrow{g\circ \Varpsi} \cone(\alpha^\prime)\xrightarrow{\Phi^\prime_{\Varpsi}}\cone(\beta)
\end{equation}
defined by identifying both with
$
R\xrightarrow{\iota_B}B/J\xrightarrow{g} \cone(\beta)
$
via the previously described second parts of $\can_\Varphi, \can_{\Varpsi}$.
Using \eqref{eq:1-homotopy-h-varphi}, the isomorphism between the morphisms in \eqref{eq:1-first-composition} is given by the homotopy
\[
h^\prime_\Varphi-h^\prime_{\Varpsi}\colon F\to J\otimes_{B,g}W(S),\ x\mapsto (h_\Varphi(x)-h_{\Varpsi}(x))\otimes 1,
\]
and we compute that for any $x\in F$, we have
\[
\begin{array}{rll}
	&h_\Varphi(x)-h_{\Varpsi}(x)&= \big(\Varphi\circ \widetilde{\iota_A}(x)-\widetilde{\iota_B}(x)\big)-\big(\Varpsi\circ \widetilde{\iota_A}(x)-\widetilde{\iota_B}(x)\big)=\Varphi\circ \widetilde{\iota_A}(x)-\Varpsi\circ \widetilde{\iota_A}(x)\\
\Rightarrow &h'_\Varphi(x)-h'_{\Varpsi}(x)&=(\Varphi\circ \widetilde{\iota_A}(x)-\Varpsi\circ \widetilde{\iota_A}(x))\otimes 1=\Varphi(\xi)\otimes g\big(\frac{\Varphi\circ \widetilde{\iota_A}(x)-\Varpsi\circ \widetilde{\iota_A}(x)}{\Varphi(\xi)}\big)
\end{array}
\]
To get the isomorphism between the maps in \eqref{eq:2-first-composition-alpha}, we  apply $\Phi_{\Varpsi}^{-1}\colon J\otimes_{B,g}W(S)\to I\otimes_{A,g\circ \Varpsi} W(S)$:
	\[\textstyle
	 \Phi^{-1}_{\Varpsi}\circ (h^\prime_\Varphi(x)- h^\prime_{\Varpsi}(x))	=  \xi\otimes g(u_{\Varphi,\Varpsi,\xi})g\big (\frac{\Varphi(\widetilde{\iota_A}(x))-\Varpsi(\widetilde{\iota_A}(x))}{\Varphi(\xi)}\big) 
	= \xi\otimes g\big(\frac{\Varphi(\widetilde{\iota_A}(x))-\Varpsi(\widetilde{\iota_A}(x))}{\Varpsi(\xi)}\big).\qedhere
	\]
\end{proof}
\begin{remark}\label{r:explicit-naturality-HT-version}
Base changing $\gamma_{\Varphi,\Varpsi}\colon \rho_A\circ \Varphi^*\to \rho_A\circ \Varpsi^*$ from \Cref{sec:prismatic-homotopies-1-action-on-cw-divisor-more-explicit-in-principal-case} to $X^\HT\to X^\prism$ yields a natural isomorphism $\overline{\gamma_{\Varphi,\Varpsi}}\colon \overline{\rho_A}\circ \overline{\Varphi}\isomarrow \overline{\rho_A}\circ \overline{\Varpsi}$
 of morphisms $\Spf(B/J)\to X^\HT$.
If $\overline{\Varphi}=\overline{\Varpsi}$, this is an automorphism. If $f\colon B\to S$ is a morphism factoring over $B/J$, then \Cref{sec:prismatic-homotopies-1-action-on-cw-divisor-more-explicit-in-principal-case} makes the pullback of $\overline{\gamma_{\Varphi,\Varpsi}}$ along $\Spf(S)\to \Spf(B/J)$ explicit in the case that $I=\xi A$.
\end{remark}
\begin{remark}
	\label{sec:expl-natur-prism-dependence-of-choice-of-f}
	The formulas in \Cref{sec:prismatic-homotopies-1-action-on-cw-divisor-more-explicit-in-principal-case} are independent of the auxiliary choice of $F$ in the following sense:
	If $\kappa \colon F^\prime\to F$ is some morphism such that $F^\prime\to R$ is still surjective, then one checks that the formulas agree after replacing $\widetilde{\iota_A}$ by $\widetilde{\iota_A}\circ \kappa$.  Note that if $A/I=R$, then we can in all calculations actually take $F=A$ with its natural surjection to $R$ and $\widetilde{\iota_A}=\Id_A$ (we used that $F$ is a free polynomial algebra only to ensure the existence of $\widetilde{\iota_B}$, but if $R=A/I$ we can take $\widetilde{\iota_B}:=\Varphi$). This will be the only case we are interested in.
	We chose the presentation involving a general $F$ in order to simplify the exposition of the subtle isomorphisms between morphisms of animated rings. Namely, if $R=A/I$ and $A=F$ then we could take $\widetilde{\iota_A}=\Id_{A}$ and $\widetilde{\iota_B}=\Varphi$. However, when discussing $\can_{\Varpsi}$ we could not change from $\widetilde{\iota_B}=\Varphi$ to $\widetilde{\iota_B}=\Varpsi$, but we have to take the same $\widetilde{\iota_B}$ for $\can_\Varphi$ and $\can_{\Varpsi}$.
\end{remark}

\subsection{Automorphisms of $\overline{\rho_A}$}
\label{sec:autom-overl}

Let $(A,I)$ be a prism. Set $R:=A/I$ and $X:=\Spf(R)$. In this section we want to understand the group sheaf of automorphisms
\[G_A:=\mathrm{Aut}(\overline{\rho_A})\]
 of
$
\overline{\rho_A}\colon \Spf(A/I)\to X^\HT
$
more explicity. We follow  \cite[Construction 9.4]{bhatt2022prismatization}, and generalize it slightly to allow non-noetherian rings like $\O_C\langle T^{\pm 1}\rangle$ for $C$ a complete algebraically closed extension of $\Q_p$.

Let $S$ be a $p$-complete $R$-algebra. We recall that for any two objects
\[
(J\xrightarrow{\alpha} W(S), R\xrightarrow{\iota} \cone(\alpha)) \quad\text{and}\quad (J^\prime \xrightarrow{\beta} W(S), R\xrightarrow{\iota^\prime} \cone(\beta))\in X^\prism(S),
\]
an isomorphism between these is given by a pair $(\gamma_1,\gamma_2)$ of an isomorphism $\gamma_1\colon J\isomarrow J^\prime$ of $W(S)$-modules such that $\alpha=\beta\circ \gamma_1$, and an isomorphism $\gamma_2$ (from left to right) of the two morphisms
\[
R\xrightarrow{\iota} \cone(\alpha)\xrightarrow{\overline{\gamma_1}} \cone(\beta),\ R\xrightarrow{\iota^\prime} \cone(\beta)
\]
of animated rings. Here, $\overline{\gamma_1}$ is the isomorphism of animated rings induced by $\gamma_1$.

We now describe $\gamma_1,\gamma_2$ more explicitly. We first recall from \cite[Notation 3.4.9]{bhatt2022absolute} the formal group scheme
\[ \Gm^\sharp:=\Spf(\Z_p\langle (\tfrac{(t-1)^n}{n!})_{n\in\N}\rangle)\]
given by the $p$-completed PD-hull of $\Gm$ with respect to the ideal defined by the identity section. Similarly, we recall from \cite[Variant 3.4.12]{bhatt2022absolute} the formal group scheme 
defined analogously for $\Ga$:
\[  \Ga^\sharp:=\Spf(\Z_p\langle (\tfrac{x^n}{n!})_{n\in\N}\rangle)\]
The relevance of $\Gm^\sharp$ and $\Ga^\sharp$  is that there are natural isomorphisms of formal group schemes
\begin{equation}\label{eq:W[F]-W*[F]}
	W[F]^\times\isomarrow\Gm^\sharp \quad \text{and} \quad
	W[F]\isomarrow\Ga^\sharp
\end{equation}
defined by the projection to the first component by \cite[Lemma 3.4.11, Variant 3.4.12]{bhatt2022absolute}. Moreover:
\begin{lemma}
	\label{sec:autom-overl-1-automorphisms-of-cartier-witt-divisors-in-ht-locus}
	Let $(J\xrightarrow{\beta} W(S))\in \Z_p^\HT(S)$ be any Cartier--Witt divisor.
	
	\begin{enumerate}
		\item The multiplication action of $W^\times[F]$ on $J$ yields via \eqref{eq:W[F]-W*[F]} an isomorphism between $\Gm^\sharp$ and
		\[
		T \mapsto \mathrm{Aut}(J\otimes_{W(S)} W(T)\xrightarrow{\beta\otimes 1} W(T))\in \Z_p^\HT(T). 
		\]
		considered as a group sheaf
		on the category of $p$-complete $S$-algebras.
		\item We have $\ker(\beta)=J\otimes_{W(S)} W[F](S)$. For $\beta\in \Spf(R)^\HT(S)$, this identifies via \eqref{eq:W[F]-W*[F]} with
		\[ \ker(\beta)=I/I^2\otimes_A\Ga^\sharp(S).\]
	\end{enumerate}
	
\end{lemma}
\begin{proof}
	For part (1), given $x\in W^\times[F](T)$, we claim that we have $x\cdot \beta=\beta$. Indeed, this equality can be checked fpqc-locally on $T$, and Zariski-locally on $T$ we know that $\beta$ identifies with the multiplication by $V(u)$ for some $u\in W(T)^\times$. Then
	$
	xV(u)=V(F(x)u)=V(u)
	$
	as desired. Hence there is a natural map from $\Gm^\sharp$ to the group sheaf in question. This is an isomorphism: As $J$ is an invertible $W(S)$-module, any automorphism is given  by multiplication by some $x\in W(T)$ such that $x\cdot \beta=\beta$, and by the above computation we deduce $F(x)=1$ from injectivity of $V$.
	
	For part (2), the first statement can again be checked fpqc-locally on $S$, so we may assume that $\beta$ identifies with multiplication by $V(u)$ for some $u\in W(S)^\times$. This morphism has kernel $W[F](S)$ because $u$ is a unit and $V$ injective. For the second statement, we note that the $W(S)$-action on $W[F](S)$ factors over $S=W(S)/V(W(S))$ and $A\to S$ factors over $R$. This shows that $J\otimes_{W(S)}W[F](S)=I\otimes_A W[F](S)\cong I/I^2\otimes_R W[F](S)$. Now use \eqref{eq:W[F]-W*[F]}.
\end{proof}

Let $f\colon A\to A/I\cong R \to S$ be the composition and $g\colon A\to W(S)$ its natural lift.
The action of $G_A$ on the Cartier-Witt divisor $(I\otimes_AW(S)\xrightarrow{\alpha} W(S))$ yields by \Cref{sec:autom-overl-1-automorphisms-of-cartier-witt-divisors-in-ht-locus}.1 a homomorphism
\[
\pi\colon G_A\to \Gm^\sharp=W^\times[F].
\]
For $\gamma=(\gamma_1,\gamma_2)\in G_A$ this yields a more explicit understanding of $\gamma_1$.

We now turn to a description of $\gamma_2$. Set $x:=\pi(\gamma)\in W^\times[F](S)$. By construction, the maps
\[
R\to \cone(\alpha)\xrightarrow{\overline{\gamma_1}} \cone(\alpha)\quad \text{and} \quad R\to \cone(\alpha)
\]
are induced by the commutative diagrams
\[\begin{tikzcd}
	I & {I\otimes_{A,g} W(S)} \\
	A & {W(S)}
	\arrow["{i\mapsto i\otimes x}", from=1-1, to=1-2]
	\arrow["g", from=2-1, to=2-2]
	\arrow[from=1-1, to=2-1]
	\arrow["\alpha", from=1-2, to=2-2]
\end{tikzcd}\quad \textrm{ and }\quad \begin{tikzcd}
	I & {I\otimes_{A,g} W(S)} \\
	A & {W(S).}
	\arrow["{ i\mapsto i\otimes 1}", from=1-1, to=1-2]
	\arrow["g", from=2-1, to=2-2]
	\arrow[from=1-1, to=2-1]
	\arrow["\alpha", from=1-2, to=2-2]
\end{tikzcd}
\]
An isomorphism $\gamma_2$ from left to right is then a $p$-adically continuous homotopy $D\colon A\to I\otimes_AW(S)$. We can now describe precisely which such pairs $(\gamma_1,\gamma_2)$ define elements of $G_A$. 
\begin{lemma}
	\label{sec:autom-overl-1-description-of-g-a}
	The group sheaf $G_A$ identifies naturally with the sheaf
	\[S\mapsto 
	\Big\{(x,D)\in \Gm^\sharp(S)\times \mathrm{Der}_{\Z_p}(A,\Ga^\sharp\{1\}(S))  \,\Big|\ D(a)=(x-1)(a\otimes 1) \textrm{ for } a\in I\Big\}
	\]
	where $\mathrm{Der}_{\Z_p}$ denotes the $(p,I)$-adically continuous derivations and we use $\Ga^\sharp\{1\}\cong I/I^2\otimes_R\Ga^\sharp$ to formulate the equality on the right. The group structure is given by
	\[
	(x',D')\cdot (x,D):=(x'\cdot x, D'\cdot x+D).
	\]
	The action of $(x,D)$ on $(I\otimes_A W(S)\xrightarrow{\alpha} W(S), R\to \cone(\alpha))$ is via multiplication with $x$ on $I\otimes_A W(S)$ and by the homotopy $D\colon A\to I\otimes_A W[F](S)$.
\end{lemma}
\begin{proof}
	Write $F$ for the displayed sheaf, then the above discussion yields a natural map $G_A\to F$: Indeed,
 as $\alpha\circ D=g-g=0$, the map $D$ factors through $\ker(\alpha)=I\otimes_A W[F](S)\cong I/I^2\otimes_R W[F](S)$ by \Cref{sec:autom-overl-1-automorphisms-of-cartier-witt-divisors-in-ht-locus}.2. Thus, $D$ is even $(p,I)$-adically continuous.
The fact that $\gamma_2$ is an isomorphism of two morphisms of animated rings implies that $D\colon A\to I\otimes_AW[F](S)$ is actually a ($\Z_p$-linear) derivation. The requirement that $D$ is a homotopy from left to right means that for all $a\in I$,
\[
D(a)=a\otimes x-a\otimes 1=(x-1)(a\otimes 1).
\]
We thus get the desired map $G_A\to F$.
Conversely, any $(p,I)$-adically continuous derivation $D\colon A\to I\otimes_A W[F](S)$ satisfying the above equation defines a homotopy $D$ which naturally yields  an isomorphism between the morphisms $R\to \cone(\alpha)\xrightarrow{\overline{\gamma_1}} \cone(\alpha)$ and  $R\to \cone(\alpha)$ of animated rings, cf.\ \cite[Section 5.1.8]{cesnavicius2019purity}, resp.\ \cite[Section 25.3]{lurie_spectral_algebraic_geometry}. Hence the map $G_A\to F$ is an isomorphism. 

For the group structure, it is clear from \Cref{sec:autom-overl-1-automorphisms-of-cartier-witt-divisors-in-ht-locus}.1 that the projection to the first factor is a group homomorphism.
To compute the effect of the group structure on the homotopy, note that the isomorphism $R\to \cone(\alpha)$ underlying $(x',D')\cdot (x,D)\cdot \alpha$ is given by
\[ R\to \cone(\alpha)\xrightarrow{\overline{x}}\cone(\alpha)\xrightarrow{\overline{x}'}\cone(\alpha).\]
The isomorphism of this to $R\to \cone(\alpha)$
defined by $(x',D')\cdot (x,D)$ is the composition
\[
\left[\!\!\begin{tikzcd}[column sep=0.2cm]
	I \arrow[d] \arrow[r] & I\otimes W(S) \arrow[d] \arrow[r, "x"] & I\otimes W(S) \arrow[d] \arrow[r, "x'"] & I\otimes W(S) \arrow[d] \\
	A \arrow[r]           & W(S) \arrow[r,"\Id"]                                    & W(S) \arrow[r,"\Id"]                                     & W(S)                   
\end{tikzcd}\!\!\right]\!\!\to\!\! 
\left[\!\!\begin{tikzcd}[column sep=0.2cm]
	I \arrow[d] \arrow[r] & I\otimes W(S) \arrow[d] \arrow[r, " x"] & I\otimes W(S) \arrow[d]  \\
	A \arrow[r]           & W(S) \arrow[r,"\Id"]                                    & W(S)    
\end{tikzcd}\!\!\right]\!\!\to\!\!
\left[\!\!\begin{tikzcd}[column sep=0.2cm]
	I \arrow[d] \arrow[r] & I\otimes W(S) \arrow[d] \\
	A \arrow[r]           & W(S)
\end{tikzcd}\!\!\right]
\]
where the first morphism is induced by $D'$ and the second by $D$. By commuting the order of $x$ and $x'$ in the first diagram, we see that the first arrow is the composition of $\overline{(D',x')}$ with  $\overline{x}$, which corresponds to the homotopy $A\xrightarrow{D'}I\otimes W(S)\xrightarrow{x}I\otimes W(S)$. Thus the composition is $D'x+D$.
\end{proof}

\begin{definition}\label{def:Gamma^bullet(-)}
	For any $R$-module $M$, we denote by $\Gamma_R(M)$ the $p$-completed PD-hull of $\mathrm{Sym}^\bullet_R(M)$ with respect to the ideal generated by $M$. 
	For example, we could write $\Ga^\sharp$ as $\Spf(\Gamma_R(R\cdot x))$.
\end{definition}
Consider
$
\Gamma_R(\widehat{\Omega}^1_{A|\Z_p}\otimes_A R)
$
where $\widehat{\Omega}^1_{A|\Z_p}$ are the $(p,I)$-completed differentials of $A$ over $\Z_p$. Then the functor $S\mapsto \mathrm{Der}_{\Z_p,\mathrm{cont}}(A,\Ga^\sharp\{1\}(S))=\Hom_{R}(\widehat{\Omega}^1_{A|\Z_p}\otimes_A R, \Ga^\sharp\{1\}(S))$ is represented by
\[
\mathcal{T}^\sharp_{A|\Z_p}\{1\}\times_{\Spf(A)}\Spf(R):=\Spf(\Gamma_R(\widehat{\Omega}^1_{A|\Z_p}\otimes_AR\{-1\})).
\]
From \Cref{sec:autom-overl-1-description-of-g-a}, we deduce as in \cite[Construction 9.4]{bhatt2022prismatization} that
\[
G_A\subseteq  (\mathcal{T}^\sharp_{A|\Z_p}\{1\}\times_{\Spf(A)}\Spf(R))\rtimes \Gm^\sharp,
\]
where the semidirect product is formed with respect to the natural rescaling action of $\Gm^\sharp$ on $\mathcal{T}^\sharp_{A|\Z_p}\{1\}\times_{\Spf(A)}\Spf(R)$. 
Here and in the following,  we now swap the order of $\gamma_1$, $\gamma_2$  to align with the convention of writing semi-direct products in such a way that the normal subgroup comes first.

\begin{proposition}
	\label{sec:autom-overl-1-explicit-group-structure-on-g-a}
	The projection $G_A\to \mathcal{T}^\sharp_{A|\Z_p}\{1\}\times_{\Spf(A)}\Spf(R), (D,x)\mapsto D$ is an isomorphism of sheaves on $p$-complete $R$-algebras. In particular, $G_A$ is representable by a formal group scheme over $R$. The group structure on $G_A$ transfers to the operation for $D, D^\prime\in \mathcal{T}^\sharp_{A/\Z_p}\{1\}\times_{\Spf(A)}\Spf(R)$
	\[\textstyle
	D'\ast D:=D'+D+\frac{D(I)}{I\otimes 1}D^\prime,
	\]
	where $\frac{D(I)}{I\otimes 1}\in \Ga^\sharp$ is the unique element which Zariski-locally where $I=\xi A$ is given by $\frac{D(\xi)}{\xi\otimes 1}$.
\end{proposition}
\begin{proof} Let $S$ be a $p$-complete $R$-algebra.
	For $(D,x)\in G_A(S)$ we have $x(a\otimes 1)=a\otimes 1+D(a)$ for all $a\in I$ by 
	\Cref{sec:autom-overl-1-description-of-g-a}. As $I$ is invertible, this determines $x$ uniquely by $D$. More precisely, assume that $I=(\xi)$ and let $\lambda\in A^\times$. Then $D(\lambda\xi)=\lambda D(\xi)+\xi D(\lambda)=\lambda D(\xi)$ because the $A$-module structure on $\Ga^\sharp\{1\}(S)$ factors over $R=A/\xi$. This implies that $\frac{D(I)}{I\otimes 1}:=\frac{D(\xi)}{\xi\otimes 1}=\frac{D(\lambda \xi)}{\lambda \xi \otimes 1}$ is independent of the choice of $\xi$ and hence glues to an element in $\Ga^\sharp(S)$. Now, $x=1+\frac{D(I)}{I\otimes 1}\in \Gm^\sharp(S)$ is uniquely determined by $D$, namely $x(\xi\otimes 1)=D(\xi)+\xi\otimes 1$ implies $x=\frac{D(\xi)}{\xi\otimes 1}+1$. The group structure 	$(D',x')\cdot (D, x)=(D+x\cdot D', x'x)$ from 
	\Cref{sec:autom-overl-1-description-of-g-a}
	then gives the desired description
	$
	D'\ast D=D+D'+\frac{D(\xi)}{\xi\otimes 1}D^\prime$.
\end{proof}

In the setup of \Cref{sec:prismatic-homotopies-1-action-on-cw-divisor-more-explicit-in-principal-case}, this has the following consequence:
\begin{proposition}\label{sec:autom-overl-1-elements-in-g-a-coming-from-automorphism}
	Let $(A,I)$ be a prism with $A/I=R$ and assume that $I=(\xi)$ is principal.  Let $\Varphi,\Varpsi:(A,I)\to (B,J)$ be two morphisms of prisms such that $\overline{\Varphi}=\overline{\Varpsi}:A/I\to B/J$ agree. Then the automorphism $\gamma_{\Varphi,\Varpsi}\colon \overline{\rho_A}\circ \overline{\Varphi}\to \overline{\rho_A}\circ \overline{\Varpsi}$ from \Cref{r:explicit-naturality-HT-version}
	 is given for any $p$-complete, $p$-torsion free $B/J$-algebra $S$ with its natural lift $g:B\to W(S)$ by the action of the element $
	(D_{\Varphi,\Varpsi,\xi},x_{\Varphi,\Varpsi,\xi})\in G_A(S)$
	defined by
	\begin{alignat*}{3}
	D_{\Varphi,\Varpsi,\xi}&\colon&& A\to I\otimes_A  \Ga^\sharp(S),\, a\mapsto \xi\otimes g(\tfrac{\Varphi(a)-\Varpsi(a)}{\Varpsi(\xi)}),\\
	x_{\Varphi,\Varpsi,\xi}&=&&\,g(u_{\Varphi,\Varpsi,\xi}), \text{ where } u_{\Varphi,\Varpsi,\xi}:=\tfrac{\Varphi(\xi)}{\Varpsi(\xi)}\in B^\times.
	\end{alignat*}
\end{proposition}
\begin{proof}
	We wish to apply \Cref{sec:prismatic-homotopies-1-action-on-cw-divisor-more-explicit-in-principal-case}. For this we can set $F=A$, $\widetilde{\iota_A}=\Id_A$ due to \Cref{sec:expl-natur-prism-dependence-of-choice-of-f}. Note that if $S$ is $p$-torsion free, then $\Ga^\sharp(S)\to \Ga(S)$ is injective, so also the following is injective:
	\[
	G_A(S)\to ((\mathcal{T}_{A|\Z_p}\{1\}\times_{\Spf(A)}\Spf(R))\rtimes \Gm)(S).
	\]
 Hence $D_{\Varphi,\Varpsi,\xi},x_{\Varphi,\Varpsi,\xi}$ are determined by their composition with $I\otimes_AW[F](S)\to I\otimes_A S$ and $W^\times[F](S)\to S^\times$, i.e., by their first Witt component. The result now follows from \Cref{sec:prismatic-homotopies-1-action-on-cw-divisor-more-explicit-in-principal-case}.
\end{proof}
\begin{corollary}\label{c:explicit-descr-GA-if-Omega_A-finite-free}
	\label{sec:autom-overl-2-explicit-action-on-g-a}
	If 
	$
	\widehat{\Omega}^1_{A|\Z_p}\cong \textstyle\oplus_{i=0}^n A \cdot du_i
	$ for some $u_i\in A$,
 \Cref{sec:autom-overl-1-explicit-group-structure-on-g-a} yields an isomorphism
	\[
	G_A\isomarrow \mathcal{T}_{A|\Z_p}^\sharp\{1\}\times_{\Spf(A)}\Spf(R)\cong \textstyle\prod\limits_{i=0}^n \Ga^\sharp\{1\},\ D\mapsto (D(u_0),\ldots, D(u_n)).
	\]
	If there is a generator $\xi$ of $I$ that we can use to trivialise the Breuil-Kisin twist $\Ga\{1\}\cong\Ga$, then the group structure on $G_A$ transfers through this to the map
	\begin{alignat*}{3}
	\textstyle\prod\limits_{i=0}^n \Ga^\sharp\times \prod\limits_{i=0}^n \Ga^\sharp \to \textstyle\prod\limits_{i=0}^n \Ga^\sharp,\quad
	((a_i)_{i=0,\dots,n},(b_j)_{j=0,\dots,n})\mapsto (a_k+b_k+a_k\textstyle\sum\limits_{l=0}^n\frac{b_l\partial \xi}{\partial u_l})_{k=0,\dots,n}
	\end{alignat*}
	where the $\frac{\partial}{\partial u_l}\colon A\to R$ are the derivations forming the dual basis for $du_0, \ldots, du_n$.
	If $\xi=E(u_0)$ is a polynomial in $u_0$ with coefficients killed by all $\frac{\partial}{\partial u_l}$, this  simplifies to
	\[
	((a_i)_{i=0,\dots,n},(b_j)_{j=0,\dots,n})\mapsto (a_k(1+E^\prime(u_0)b_0)+b_k)_{k=0,\dots,n}.
	\]
\end{corollary}
\begin{proof}
	This is immediate from
	\Cref{sec:autom-overl-1-explicit-group-structure-on-g-a} by unravelling the formulas there.
\end{proof}
For $n=0$, \Cref{c:explicit-descr-GA-if-Omega_A-finite-free} recovers the formula of \cite[Example 9.6]{bhatt2022prismatization}. 

\subsection{Application to group actions}
\label{sec:appl-group-acti}

We now apply
\S\ref{sec:explicit-naturality} and \S\ref{sec:autom-overl} in a more specific situation. 
\begin{Setup}\label{setup:Galois-action}
Fix a prism $(A,I)$ over $\Z_p$ with $I=\xi A$ principal. Set $R:=A/I$ and $X:=\Spf(R)$. Let $\Varphi\colon (A,I)\to (A_\infty,J)$ be a morphism of prisms such that $R_\infty:=A_\infty/J$ is $p$-torsionfree. Let 
\[\tau\colon (A_\infty, J)\to (A_\infty,J)\] be an automorphism of prisms. Set $\Varpsi:=\tau\circ \Varphi$.
We assume that 
$
\overline{\Varphi}\colon R\to R_\infty
$
is invariant under $\sigma:=\overline{\tau}\colon R_\infty\to R_\infty$, i.e., $\overline{\Varpsi}=\sigma\circ \overline{\Varphi}=\overline{\Varphi}$, so we are in the setting of \Cref{sec:prismatic-homotopies-1-action-on-cw-divisor-more-explicit-in-principal-case} and \Cref{sec:autom-overl-1-elements-in-g-a-coming-from-automorphism}. 
\end{Setup}
 In our applications, $R\tf\to R_\infty\tf$ will be pro-\'etale Galois, and $\sigma$ will be in the Galois group.

The naturality of $A\mapsto \rho_A$ yields $2$-commutative diagrams
\begin{equation}\label{eq:Galois-action-diagram}
\begin{tikzcd}[row sep =0.25cm,column sep = 0.5cm]
	{\Spf(A_\infty)} & {\Spf(A)} \\
	&& {X^\prism} \\
	{\Spf(A_\infty)} & {\Spf(A)}
	\arrow["{\rho_A}", from=1-2, to=2-3]
	\arrow["{\rho_A}"', from=3-2, to=2-3]
	\arrow["\Varphi"', from=3-1, to=3-2]
	\arrow["\Varphi", from=1-1, to=1-2]
	\arrow["{\rho_{A_\infty}}"', from=1-1, to=2-3]
	\arrow["{\rho_{A_\infty}}", from=3-1, to=2-3]
	\arrow["\tau", from=1-1, to=3-1]
\end{tikzcd}\quad
\text{ and  }\quad
\begin{tikzcd}[row sep =0.25cm,column sep = 0.5cm]
{\Spf(R_\infty)} & {\Spf(R)} \\
&& {X^\HT} \\
{\Spf(R_\infty)} & {\Spf(R)}
\arrow["{\overline{\rho_A}}", from=1-2, to=2-3]
\arrow["{\overline{\rho_A}}"', from=3-2, to=2-3]
\arrow["\overline{\Varphi}"', from=3-1, to=3-2]
\arrow["\overline{\Varphi}", from=1-1, to=1-2]
\arrow["{\overline{\rho_{A_\infty}}}"', from=1-1, to=2-3]
\arrow["{\overline{\rho_{A_\infty}}}", from=3-1, to=2-3]
\arrow["\sigma", from=1-1, to=3-1]
\end{tikzcd}
\end{equation}
where the diagram on the right is the base change to the Hodge--Tate locus of the one on the left. 

The aim of this subsection is to study the fibre product $Z_A$ making the following square Cartesian
\begin{equation}\label{d:def-ZA}
\begin{tikzcd}
	Z_A \arrow[d] \arrow[r] & \Spf(R_\infty)\arrow[dl,dotted,"\overline{\Varphi}"'] \arrow[d,"\overline{\rho_{A_\infty}}"] \\
	\Spf(R) \arrow[r,"\overline{\rho_{A}}"]       & X^\HT.
\end{tikzcd}
\end{equation}
As $\overline{\rho_{A}}$ is affine, $Z_A$ is represented by a formal scheme. 
The $S$-points of $Z_A$ on a $p$-complete algebra $S$ identify with pairs  $(x,\gamma)$ of a morphism $x:\Spf(S)\to \Spf(R_\infty)$ and an isomorphism $\gamma:\overline{\rho_{A_\infty}}\circ x\to \overline{\rho_{A}}\circ \overline{\Varphi}\circ x$ in $X^\HT(S)$ (by composition with the natural morphism $X^\HT\to X$ one sees that necessarily the morphism $\Spf(S)\to Z_A\to \Spf(R)$ is given by $\overline\Varphi\circ x$). Hence $Z_A$ inherits from $\Spf(R_\infty)$ a natural action of $\sigma$: The automorphism $\sigma\colon R_\infty\to R_\infty$ induces first a natural \mbox{2-arrow} $\can_{\sigma}:\overline{\rho_{A_\infty}}\circ \sigma\to \overline{\rho_{A_\infty}}$ obtained from \eqref{eq:Galois-action-diagram} by \Cref{constr:explicit-2-arrow-in-X^prism}. This defines  an automorphism $Z_A\to Z_A$ over $\sigma$: Explicitly, in terms of functors of points, this is defined by sending $(x,\gamma)$ to $(\sigma\circ x, \gamma\circ \can_{\sigma})$ where $ \gamma\circ\can_{\sigma}$ is defined as the composition
\[\overline{\rho_{A_\infty}}\circ \sigma\circ x\xrightarrow{\can_\sigma} \overline{\rho_{A_\infty}}\circ x\xrightarrow{\gamma} \overline{\rho_{A}}\circ\overline{\Varphi}\circ x.\]
We can equivalently regard $Z_A\to Z_A$ as an $R_\infty$-linear isomorphism
\[
\sigma_{Z_A}\colon Z_A\to \sigma^\ast Z_A:=Z_A\times_{\Spf(R_\infty),\sigma}\Spf(R_\infty).
\]

Our aim is now to make this natural isomorphism explicit.
 To do so, we first observe that since $\overline{\rho_A}$ is a torsor under the formal group scheme $G_A$ from \Cref{sec:autom-overl-1-explicit-group-structure-on-g-a}, it follows by base-change that the map $Z_A\to \Spf(R_\infty)$ is a torsor under the formal group scheme over $\Spf(R_\infty)$ defined by
\[ G_{A,R_\infty}:=G_A\times_{\Spf(R)}\Spf(R_\infty).\]
Thus $Z_A$ receives a natural left $G_{A,R_\infty}$-action by letting $g\in G_A(S)$ send $(x,\gamma)\in Z_A(S)$ to $(x,g\circ \gamma)$.  

The map $\overline{\Varphi}$ indicated as the dotted arrow in the diagram defines a splitting of this torsor. Indeed, the natural \mbox{$2$-isomorphism} $\can_\Varphi\colon \overline{\rho_A}\circ \overline{\Varphi}\to \overline{\rho_{A_\infty}}$ from \Cref{constr:explicit-2-arrow-in-X^prism} yields an identification
\[
\vartheta_\Varphi \colon G_{A,R_\infty}\isomarrow Z_A.
\]

Observe that $G_{A,R_\infty}$ also receives a natural $\sigma$-action via the action on the second factor. Here we note that we have a natural $R_\infty$-linear isomorphism $\sigma^{\ast}G_{A,R_\infty}\isomarrow G_{A,R_\infty}$ since $\sigma$ leaves $R$ invariant.

The map $\vartheta_\Varphi$ need not commute with the \mbox{$\sigma$-actions} on both sides, as $\Varphi$ need not be $\tau$-invariant.  Instead, to make $\sigma_{Z_A}$ explicit in terms of $G_A$, we can transport it via $\vartheta_\Varphi$ to the automorphism
\[
\sigma_{G_{A,R_\infty}}:=\sigma^\ast\vartheta^{-1}_\Varphi \circ \sigma_{Z_A} \circ \vartheta_{\Varphi}\colon G_{A,R_\infty}\to \sigma^\ast G_{A,R_\infty}\cong G_{A,R_\infty}.
\]
Now $\vartheta_\Varphi$ identifies $Z_A$ and its $\sigma_{Z_A}$-action with $G_{A,R_\infty}$ and its induced $\sigma_{G_{A,R_\infty}}$-action. 
\begin{theorem}\label{t:expl-descr-of-action-on-GA}
	\label{sec:appl-group-acti-1-identification-of-sigma-via-group-action}
	The isomorphism
	$
	\sigma_{G_{A,R_\infty}}$ can be described in terms if the right multiplication by $G_{A,R_\infty}$ as the translation by the element $(D,x)\in G_{A,R_\infty}(R_\infty)$  where
	\[\textstyle 	D\colon A\to I\otimes_A\Ga^\sharp(R_\infty),\quad a\mapsto \xi\otimes g(\frac{\tau(\Varphi(a))-\Varphi(a)}{\Varphi(\xi)}),\quad \text{ and } x=g(\frac{\tau(\Varphi(\xi))}{\Varphi(\xi)})\in \Gm^\sharp(R_\infty).\] 
Here  $g\colon A_\infty\to R_\infty$ is the canonical reduction, and we use that $R_\infty$ is $p$-torsionfree to describe $D$ via the inclusion $\Ga^\sharp(R_\infty)\subseteq R_\infty$. Similarly, we implicitly use $\Gm^\sharp(R_\infty)\subseteq \Gm(R_\infty)$ to describe $x$.
\end{theorem}
This the final outcome of our discussion in \Cref{sec:prismatic-homotopies}.
\begin{proof}
	We consider the fibre product of \eqref{eq:Galois-action-diagram} with the morphism $\overline{\rho}_A:\Spf(R)\to X^\HT$. Note that the fibre product of $\overline{\rho}_A$ with itself is $G_A$. That the homotopy $\can_{\Varphi}$ induces the isomorphism $\vartheta_\Varphi$ now means that it makes the following diagram 2-commutative:
	\[\begin{tikzcd}[column sep =0.5cm,row sep = 0cm]
		& G_{A,R_\infty} \arrow[rrd,""] \arrow[dd] \arrow[ld,"\vartheta_{\Varphi}"'] &  &              \\
		Z_A \arrow[rrr,crossing over] \arrow[dd] &                                                  &  & \Spf(R) \arrow[dd] \\
		& \Spf(R_\infty) \arrow[rrd,"\overline{\rho_A}\circ \overline{\Varphi}"]\arrow[ld,"\Id"']          &  &              \\
		\Spf(R_\infty) \arrow[rrr,"\overline{\rho_{A_\infty}}"]            &                                                  &  & X^\HT.
	\end{tikzcd}\]
	In terms of functors of points, $G_{A,R_\infty}$ sends any $R$-algebra $S$ to the pairs $(x,h)$ of a morphism $x:\Spf(S)\to \Spf(R_\infty)$ over $R$ and an isomorphism $h:\overline{\rho_{A}}\circ \overline{\Varphi}\circ x\to  \overline{\rho_{A}}\circ \overline{\Varphi}\circ x$ in $X^\HT(S)$. Then $\vartheta_\Varphi$
	\[ \vartheta_\Varphi:G_{A,R_\infty}\isomarrow Z_A \quad\text{is given by} \quad (x,h)\mapsto (x,\overline{\rho_{A_\infty}}\circ x\xrightarrow{\can_{\Varphi}^{-1}}\overline{\rho_{A}}\circ \overline{\Varphi}\circ x\xrightarrow{h}\overline{\rho_{A}}\circ \overline{\Varphi}\circ x),\]
	hence its inverse $\vartheta_\Varphi^{-1}$ is given similarly by precomposing with $\can_{\Varphi}$. Combined with the explicit description of $\sigma_{Z_A}$ given above, it follows that $\sigma_{G_{A,R_\infty}}=\sigma^{\ast}\vartheta_{\Varphi}^{-1}\circ\sigma_{Z_A}\circ\vartheta_\Varphi$ is given by the composition
	\[G_{A,R_\infty}\isomarrow G_{A,R_\infty},\quad (x,h)\mapsto (x,\overline{\rho_{A}}\circ \overline{\Varphi}\circ x\xrightarrow{\can_{\Varphi}}\overline{\rho_{A_\infty}}\circ\sigma\circ x\xrightarrow{\can_{\sigma}}\overline{\rho_{A_\infty}}\circ  x\xrightarrow{\can_{\Varphi}^{-1}}\overline{\rho_{A}}\circ \overline{\Varphi}\circ x\xrightarrow{h}\overline{\rho_{A}}\circ \overline{\Varphi}\circ x)\]
	where we use that $\overline{\Varphi}\circ \sigma=\overline{\Varphi}$ to identify the first term of the homotopy.
	Using that $\Varpsi=\tau\circ \Varphi$, this homotopy is the translation 
	\[ h\mapsto h\circ \can_{\Varphi}^{-1}\circ \can_{\Varpsi}.\]
	But by definition, $\can_{\Varphi}^{-1}\circ \can_{\Varpsi}=\gamma_{\Varphi,\Varpsi}^{-1}$.
	We now use that by \Cref{sec:autom-overl-1-elements-in-g-a-coming-from-automorphism}, the automorphism $\gamma_{\Varphi,\Varpsi}$ is given on $S$-points for any $p$-complete, $p$-torsion free $R_\infty$-algebra $S$ by the action of the pair $(D_{\Varphi,\Varpsi,\xi},x_{\Varphi,\Varpsi,\xi})\in G_A(S)$ defined by
	\[
	D_{\Varphi,\Varpsi,\xi}\colon A\to I\otimes_A\Ga^\sharp(S),\ a\mapsto \xi\otimes g(\tfrac{\Varphi(a)-\Varpsi(a)}{\Varpsi(\xi)})
	\quad \text{and}\quad
	x_{\Varphi,\Varpsi,\xi}=g(\tfrac{\Varphi(\xi)}{\Varpsi(\xi)})=g(\tfrac{\Varphi(\xi)}{\tau(\Varphi(\xi))}),
	\]
	where $g:A_\infty\to W(S)$ is the canonical lift of $A_\infty\to R_\infty\to S$. By functoriality, this is determined by its value on $S=R_\infty$. Since $R_\infty$ is $p$-torsionfree by assumption, we can use $\theta\colon W(R_\infty)\to R_\infty$ to identify any element in $W[F](R_\infty)=\Ga^\sharp(R_\infty)$ with its image in $\Ga(R_\infty)=R_\infty$.
	
	Finally, one calculates using the semi-direct product structure explained in \Cref{sec:autom-overl-1-explicit-group-structure-on-g-a} that the inverse is given by $(D,x)=(-D_{\Varphi,\Varpsi,\xi}x_{\Varphi,\Varpsi,\xi}^{-1},x_{\Varphi,\Varpsi,\xi}^{-1})$, as described.
\end{proof}

\begin{remark}
	We implicitly made a sign convention, related to the question whether $Z_A$ is a torsor for a left or right action of the (possibly non-commutative) group $G_A$. In the above, the left action on $G_{A,R_\infty}$ identifies via $\vartheta_\Varphi$ with the $G_A$-action on $Z_A$, making the latter a torsor under a left $G_A$-action.
One could instead define the homotopy in the definition of $Z_A$ to go into the other direction, so the $G_A$-action on $Z_A$ is from the right. Equivalently, we can decide to let $G_A$ act via its inverse. Either way, turning around all arrows in the above discussion, the statement of \Cref{t:expl-descr-of-action-on-GA} would become that the morphism is given by left-multiplication with $(D_{\Varphi,\Varpsi,\xi},x_{\Varphi,\Varpsi,\xi})$.
\end{remark}

\begin{remark}
  \label{sec:appl-group-acti-1-dependence-modulo-xi-squared}
  By definition, the pair $(D,x)$ in \Cref{sec:appl-group-acti-1-identification-of-sigma-via-group-action} depends on $\Varphi, \Varpsi=\tau\circ \Varphi$ and $\xi$. Let $(A_0,I_0)\to (A,I)$ be a morphism of prisms such that $\Varphi$ and $\tau$ are $A_0$-linear. Furthermore, assume that $\xi\in I_0$, then $x=1$ and $D\colon A\to \Ga^\sharp(R_\infty),\ a\mapsto -\xi\otimes g(\frac{\Varphi(a)-\Varpsi(a)}{\xi})$ with $g\colon A_\infty\to W(R_\infty)$ the $\delta$-lift of $A_\infty\to R_\infty$. If $a\in \xi\cdot A$, then $D(a)=0$. Indeed, write $a=\xi b$ for some $b\in A$. Then we have $
    \frac{\Varphi(a)-\Varpsi(a)}{\xi}=\Varphi(b)-\Varpsi(b)$,
  and as $R_\infty$ is $p$-torsion free it suffices to check that the image of $\Varphi(b)-\Varpsi(b)$ vanishes under $A_\infty\to R_\infty$. But this follows because $\sigma$ is $R$-linear. 
\end{remark}

\section{Examples}
\label{sec:examples}
We now discuss various settings of smooth formal schemes over different base rings in which we apply the discussion from \S\ref{sec:appl-group-acti} to make the action on $G_{A,R_\infty}$ described in \Cref{t:expl-descr-of-action-on-GA} explicit.

\subsection{$p$-adic fields}
\label{sec:p-adic-fields-new-identification-of-galois-action-for-p-adic-fields}

Let $K$ be a $p$-adic field, $R:=\mathcal{O}_K$ its ring of integers and $k$ its residue field.
Set $C:=\widehat{\overline{K}}$ and $R_\infty:=\mathcal{O}_C$. Then $\Gal(\overline{K}/K)$ acts on $R_\infty$ and this action extends to an action on
\[
(A_\infty:=A_\inf(\mathcal{O}_C),J:=\ker(\theta\colon A_\infty\to R_\infty)),
\]
where as usual, we have Fontaine's map 
\[\theta:A_\infty\to R_\infty=\O_C.\]
Let $\pi\in R$ be a uniformizer. Then we get the associated Breuil-Kisin prism
\[
(A,I):=(\mathfrak{S}:=W(k)[[u]],(E(u))).
\]
Any choice of a system $\pi^\flat=(\pi,\pi^{1/p},\ldots)\in \O_C^\flat$ of $p$-power roots of $\pi$ yields a morphism
\[
\Varphi\colon A\to A_\infty,\ u\mapsto [\pi^\flat]
\]
of prisms. As $R\cong A/I$ and each $\sigma\in \Gal(\overline{K}/K)$ fixes $R$, we are then in the situation of \Cref{setup:Galois-action}.
Fix a compatible system $\varepsilon=(1,\zeta_p,\ldots)\in \O_C^\flat$ of primitive $p$-power roots of unity and as usual set $\mu:=[\varepsilon]-1$, $\xi:=\frac{\mu}{\varphi^{-1}(\mu)}$.
Let us use $E(u)\in I$ as a generator (this plays the role of $\xi$ in \Cref{sec:appl-group-acti-1-identification-of-sigma-via-group-action}, and the $\xi$ we just defined has a different role). 
\begin{lemma}\label{p:descr-D-x-p-adic-field}
	Let $\sigma\in \Gal(\overline{K}|K)$ and let $\tau:A_\infty\to A_\infty$ be the induced automorphism. Then the associated element $(D,x)\in G_A(R_\infty)$ of \Cref{sec:appl-group-acti-1-identification-of-sigma-via-group-action} can  be described as follows: We have $\Omega^1_{A|\Z_p}=A\cdot du$ and $D$ is the unique derivation $A\to I\otimes_A\Ga^\sharp(\O_C)$ such that
	\[ D(u)=E(u)\otimes c(\sigma)\pi z\]
	where $c(\sigma)\in \Z_p$ is the unique element such that $
	\sigma([\pi^\flat])=[\varepsilon^{c(\sigma)}][\pi^\flat]$, and where we define
	\[\textstyle
	z:=\theta\big(\frac{\mu}{E([\pi^\flat])}\big)=(\zeta_p-1)\theta\big(\frac{\xi}{E([\pi^\flat])}\big).
	\]
	Second, we have
	$x=\chi_{\pi^\flat}(\sigma):=1+c(\sigma) E^\prime(\pi)\pi z\in \Gm^\sharp(\O_C)$.
\end{lemma}
\begin{remark}
\label{sec:p-adic-fields-remark-chi-pi-flat-similar-to-the-cyclotomic-character}
The map
\[
\chi_{\pi^\flat}\colon \Gal(\overline{K}/K)\to 1+\pi(\zeta_p-1)\mathcal{O}_C,\ \sigma\mapsto \chi_{\pi^\flat}(\sigma)=1+c(\sigma)E^\prime(\pi)\pi z
\]
is a cocycle and plays a similar role as the cyclotomic character, cf.\ \cite[Lemma 3.6]{analytic_HT}.
\end{remark}
\begin{proof}[Proof of \Cref{p:descr-D-x-p-adic-field}]
	The equality in the definition of $z$ follows from $\mu=\varphi^{-1}(\mu)\xi$ and $\theta(\varphi^{-1}(\mu))=\zeta_p-1$.
	
As explained in \Cref{sec:appl-group-acti-1-identification-of-sigma-via-group-action}, since $R_\infty$ is $p$-torsionfree, it suffices to identify $(D,x)$ after composition with
$
G_A\to \mathcal{T}_A\rtimes \Gm$, i.e.\ we may apply $\theta^\prime\colon W(R_\infty)\to R_\infty$. 
The formula for $x$ in \Cref{t:expl-descr-of-action-on-GA} then boils down to the following computation, which we use again later:
\begin{lemma}\label{l:std-form-for-computing-Dx}
	Let $B$ be a perfectoid $\Z_p^\cycl$-algebra. Let $\sigma$ be a $\Z_p$-linear automorphism of $B$, we also denote by $\sigma$ the induced automorphism of $A_\inf(B)$. Let $t^\flat\in B^\flat$ be an element such that $\sigma$ fixes $t:=t^{\flat\sharp}\in B$. Let $E(u)$ be a polynomial in $A_\inf(B)[u]$ with coefficients that are fixed by $\sigma$. Then
	\[\sigma(E([t^\flat]))\equiv E([t^\flat])+ c(\sigma)\cdot E'([t^\flat])\cdot [t^\flat]\cdot \mu\bmod \mu^2\]
	where $c(\sigma)\in \Z_p$ is any element such that $\sigma([t^\flat])=[\varepsilon]^{c(\sigma)}[t^\flat]$.\footnote{If $B$ is $p$-torsion free, this element is unique.} In particular, this shows
	\[\textstyle\theta\big(\frac{E(\sigma([t^\flat]))-E([t^\flat])}{\mu}\big)=c(\sigma)t\frac{\partial E}{\partial u}(t)\in B.\]
\end{lemma}
\begin{proof}
	We calculate
	\[E(\sigma([t^\flat]))=E(\varepsilon^{c(\sigma)}[t^\flat])=E((1+\mu)^{c(\sigma)}[t^\flat]).\]
	Calculating modulo $\mu^2$ gives $(1+\mu)^{c(\sigma)}\equiv 1+c(\sigma)\mu\bmod \mu^2$ and therefore
	\[E(\sigma([t^\flat]))\equiv E([t^\flat]+c(\sigma)\mu[t^\flat]))\equiv E([t^\flat])+c(\sigma)\mu E'([t^\flat])[t^\flat]\bmod \mu^2.\]
	It follows that
	\[\textstyle\frac{E(\sigma([t^\flat]))-E([t^\flat])}{\mu} \equiv c(\sigma)\cdot E'([t^\flat])\cdot [t^\flat] \bmod \mu.\]
	Since $\mu\in \ker(\theta)$, applying $\theta$ gives the desired description as $\theta([t^\flat])=t$.
\end{proof}
By \cite[Lemma 3.23]{Bhatt2018}, $\ker(A_{\infty}\to W(R_\infty))=\mu\cdot A_\infty$. By \Cref{l:std-form-for-computing-Dx} with $t^\flat=\pi^\flat$, this shows
\[\textstyle x=\theta\big(\frac{E(\sigma([\pi^\flat]))}{E([\pi^\flat])}\big)=\theta(1+c(\sigma)E'([\pi^\flat])[\pi^\flat]\frac{\mu}{E([\pi^\flat]])})=1+c(\sigma) E^\prime(\pi)\pi z=\chi_{\pi^\flat}(\sigma)\]
where $c(\sigma)\in \Z_p$ and $z$ are as defined in the statement of the Lemma, and we use that $\mu/\xi = \varphi^{-1}(\mu)$.

Second, we now calculate the derivation
$D$.
Since $\widehat{\Omega}^1_{A/\Z_p}=A\cdot du$, it suffices to compute $D(u)$. According to \Cref{t:expl-descr-of-action-on-GA}, and using again \Cref{l:std-form-for-computing-Dx} (for $E(u)=u$), we have
\[ \textstyle D(u)=E(u)\otimes \theta\big(\frac{\sigma([\pi^\flat])-[\pi^\flat]}{E([\pi^\flat])}\big)=E(u)\otimes \theta\big(\frac{\sigma([\pi^\flat])-[\pi^\flat]}{\mu}\frac{\mu}{E([\pi^\flat])}\big)=E(u)\otimes c(\sigma)\pi z.\qedhere\]
\end{proof}
\begin{remark}
	Alternatively, one could deduce the description of $D$ from the one of $x$: As $D$ is a derivation, we have $
	D(E(u))=E^\prime(\pi)D(u)$.
	On the other hand,
	$
	D(E(u))=(x-1)(E(u)\otimes 1)
	$
	by 
	\Cref{sec:autom-overl-1-description-of-g-a}. Hence,
	\[
	E^\prime(\pi)D(u)=(x-1)(E(u)\otimes 1)\in I/I^2\otimes_R R_\infty,
	\]
	\[
	\Rightarrow D(u)=(x-1)\tfrac{E(u)\otimes 1}{E^\prime(\pi)}=E(u)\otimes c(\sigma)\pi z
	\]
	since $x-1=E'(\pi)c(\sigma)\pi z$.
	This determines $D(a)$ for $a\in A$ as $D(a\cdot E(u))=(x-1)(aE(u)\otimes 1)$ while on the other hand $D(a\cdot E(u))=aE^\prime(\pi)D(u)+E(\pi)D(a)$.
\end{remark}
Via
 \Cref{sec:appl-group-acti-1-identification-of-sigma-via-group-action}, \Cref{p:descr-D-x-p-adic-field} now describes the action of $\Gal(\overline{K}/K)$ on $Z_A$ from  \eqref{d:def-ZA}:
\begin{proposition}\label{c:Galois-action-on-ZA-p-adic-fields}
	\label{sec:p-adic-fields-description-of-action-on-z-a-for-p-adic-field}
	The choice of $\pi^\flat$ yields an isomorphism
	\[
	Z_A\cong G_{A,R_\infty}\cong \Ga^\sharp=\Spf(\widehat{\textstyle\bigoplus\limits_{n\geq 0}} R_\infty\cdot \tfrac{a^n}{n!})
	\]
	with respect to which the $R_\infty$-semilinear action of $\mathrm{Gal}(\overline{K}/K)$ is given for $\sigma\in \mathrm{Gal}(\overline{K}/K)$ by
	\[
	\sigma(a)=\chi_{\pi^\flat}(\sigma)a+c(\sigma)\pi z.
	\]
\end{proposition}
\begin{proof}
	Since $\Omega^1_{A|\Z_p}=A\cdot du$, the displayed isomorphism is given by the projection to the first factor via \Cref{sec:autom-overl-1-explicit-group-structure-on-g-a}. Here we use the generator $E(u)$ to trivialize $I$ and hence to identify $\mathcal{T}^\sharp_{A|\Z_p}\{1\}\cong \mathcal{T}^\sharp_{A|\Z_p}$.
	By \Cref{sec:appl-group-acti-1-identification-of-sigma-via-group-action}, $\sigma$ acts as the right multiplication by $D(u)$. By \Cref{p:descr-D-x-p-adic-field}, this evaluates to $D(u)=E(u)\otimes c(\sigma)\pi z$. By the last part of \Cref{sec:autom-overl-2-explicit-action-on-g-a}, the action by $D(u)$ is thus given by
	\[
	a\mapsto a(1+E^\prime(\pi)c(\sigma)\pi z )+c(\sigma)\pi z=\chi_{\pi^\flat}(\sigma)a+c(\sigma)\pi z
      \]
    where we again used the generator $E(u)$ of $I$.
\end{proof}
We previously obtained \Cref{c:Galois-action-on-ZA-p-adic-fields} in \cite{analytic_HT}, where we have studied the example given in this subsection in detail. In particular, we can deduce \Cref{sec:form-reduct-prov-1-crucial-assumption-on-galois-cohomology} from \cite[Theorem 3.12]{analytic_HT}.\footnote{The reference proves that $K\cong R\Gamma(G_K,B_{A,R_\infty})$, but the proof shows the stonger statement that $\O_K\to R\Gamma(G_K,B_{A,R_\infty}^+)$ has cofiber killed by some $p^i, i\geq 1$. Alternatively, the cohomology is calculated by a two term complex of Banach spaces and the open mapping theorem implies the existence of a suitable $p^i$ killing the cofiber.}

\begin{remark}
If $t:=1+E^\prime(\pi)a$, then $\sigma(t)=\chi_{\pi^\flat}(\sigma)t$, which might be easier to remember. Indeed,
\[\sigma(t)  = 1+E^\prime(\pi)\sigma(a) 	=  1+E^\prime(\pi)c(\sigma)\pi z +E^\prime(\pi)\chi_{\pi^\flat}(\sigma)a 	= \chi_{\pi^\flat}(\sigma)(1+E^\prime(\pi)a)= \chi_{\pi^\flat}(\sigma)t.\]
\end{remark}

\subsection{Tori over perfectoid rings}
\label{sec:tori-over-geometric}

Let $R_0$ be any $p$-torsion free perfectoid ring containing $\Z_p^\cycl$. Set $A_\inf:=A_\inf(R_0)$, and define elements $[\varepsilon], \xi,\mu \in A_\inf$ as in \Cref{sec:p-adic-fields-new-identification-of-galois-action-for-p-adic-fields}.
Set
\[
R:=R_0\langle T_1^{\pm 1},\ldots, T_n^{\pm 1}\rangle
\quad \text{and}\quad
A:=A_\inf\langle u_1^{\pm 1},\ldots, u_n^{\pm 1}\rangle
\]
where the implicit completion is the $(p,\xi)$-adic completion. There is a natural $\delta$-structure on $A$ such that $\delta(u_i)=0$ for $i=1,\ldots, n$.
We identify $A/(\xi)\cong R,\ u_i\mapsto T_i$.
We set
\[
R_\infty:=R_0\langle T_1^{\pm 1/p^\infty},\ldots, T_n^{\pm 1/p^\infty} \rangle
\quad \text{and} \quad
A_\infty:=A_{\inf}(R_\infty),\]
then $R\to R_\infty$ is the usual perfectoid cover.
Choose 
$
T_i^\flat=(T_i,T^{1/p}_i,\ldots )\in R_{\infty}^\flat
$ for $i=1,\ldots, n$, then sending $u_i\mapsto [T_i^\flat]$ defines a map
\[
\Varphi\colon A\to A_\infty.
\]
The group $\Gamma:=\Z_p(1)^n$ acts on $R_\infty$ fixing $R$. Moreover, the action of $\Gamma$ lifts to $A_\infty$ (not fixing $\Varphi$!). 

\begin{lemma}\label{p:descr-D-x-tori-over-perfectoid}
	Let $\sigma\in \Z_p(1)^n$ and let $\tau:A_\infty\to A_\infty$ be the induced lift. Then the associated element $(D,x)\in G_A(R_\infty)$ of \Cref{sec:appl-group-acti-1-identification-of-sigma-via-group-action} can be described as follows: We have $x=1$ and $D$ is the derivation $A\to I\otimes_A\Ga^\sharp(R_\infty)$ given on any $f=f(u_1,\dots,u_n)\in A_\inf\langle u_1^{\pm 1},\ldots, u_n^{\pm 1}\rangle$ by
	\[\textstyle D(f)=
	\xi\otimes \sum\limits_{i=1}^n c_i(\sigma)(\zeta_p-1)T_i\frac{\partial f}{\partial u_i}(T_1,\ldots, T_n)\]
	where $c_i(\sigma)\in \Z_p$ is the unique element such that 
$\sigma([T_i^\flat])=[\varepsilon]^{c_i(\sigma)}[T_i^\flat]$.
\end{lemma}
\begin{proof}
	The action of $\sigma$ leaves $\Varphi(\xi)$ invariant, hence
	$x=\theta\big(\frac{\tau(\Varphi(\xi))}{\Varphi(\xi)}\big)=1$.
	Using that the map $A_\infty\to R_\infty$ is given by Fontaine's $\theta$, we see that the derivation $D$ is given by
\[\textstyle
D\colon A\to \xi A_\inf\otimes_{A_\inf}R_\infty,\ f(u_1,\ldots, u_n)\mapsto \xi\otimes  \theta\big(\frac{f(\sigma([T_1^\flat]),\ldots, \sigma([T_n^\flat]))-f([T_1^\flat],\ldots, [T_n^\flat])}{\xi}\big)
\]
Using \Cref{l:std-form-for-computing-Dx} and the fact that $\mu/\xi= \varphi^{-1}(\mu)$ for which $\theta(\varphi^{-1}(\mu))=(\zeta_p-1)$, we see that
\[\textstyle D(u_i)= \xi\otimes \theta\big(\frac{\sigma([T_i^\flat])-[T_i^\flat]}{\xi}\big)=\xi\otimes \theta\big(\frac{\sigma([T_i^\flat])-[T_i^\flat]}{\mu}\big)\theta\big(\frac{\mu}{\xi}\big)=\xi\otimes c_i(\sigma)\cdot T_i\cdot (\zeta_p-1).
\]
We now use that $\Omega^1_{A|\Z_p}=\oplus_{i=1}^nAdu_i$ for the $(p,\xi)$-adically completed K\"ahler differentials $\Omega^1_{A|\Z_p}$ of $A$ over $\Z_p$:
By continuity and linearity, the value of $D(u_i)$  for $i=1,\dots,n$ now determines $D(f)$ via the usual formula for derivations for any $f=f(u_1,\dots,u_n)\in A$.
\end{proof}
By \Cref{sec:appl-group-acti-1-identification-of-sigma-via-group-action} and \Cref{sec:autom-overl-2-explicit-action-on-g-a} this describes the action of $\Gamma$ on $Z_A$ from \eqref{d:def-ZA}:
\begin{proposition}\label{c:descr-Galois-action-tori-over-perfectoid}
	We have
	\[Z_A\cong (\Ga^{\sharp})^n= \Spf(\textstyle\widehat{\bigoplus\limits_{m_1,\ldots, m_n\geq 0}} R_\infty \frac{a_1^{m_1}}{m_1!}\ldots \frac{a_n^{m_n}}{m_n!})\]
	and the $R_\infty$-semilinear action of $\Gamma$ on $Z_A$ 
	is given for $\sigma \in \Gamma$ by sending 
	\[\textstyle
	\sigma(a_i)=a_i+\sum\limits_{j=1}^n c_j(\sigma)(\zeta_p-1)T_j.
	\]
\end{proposition}
\begin{proof}
	Since we have $\Omega^1_{A|\Z_p}=\oplus_{i=1}^nAdu_i$ for the $(p,\xi)$-adically completed K\"ahler differentials, the first isomorphism follows from $Z_A=G_{A,R_\infty}$ and \Cref{sec:autom-overl-2-explicit-action-on-g-a}. By \Cref{sec:appl-group-acti-1-identification-of-sigma-via-group-action}, the action of $\sigma$ is computed on $a_i$ by the right multiplication of the element $(D(u_j))_{j=1,\dots,n}$ of $(\Ga^{\sharp})^n$. By   \Cref{p:descr-D-x-tori-over-perfectoid}, we have $D(u_j)=(c_j(\sigma)(\zeta_p-1)T_j)_{j=1,\dots,n}$. We deduce from the last part of \Cref{sec:autom-overl-2-explicit-action-on-g-a} that multiplication by $(D(u_j))_{j=1,\dots,n}$ has the described effect. Here we use that in the notation of \Cref{sec:autom-overl-2-explicit-action-on-g-a}, $E(u_0)=\xi$ is a constant polynomial and hence $E'(u_0)=0$.
\end{proof}
\subsection{Tori over $p$-adic fields}
\label{sec:tori-over-p-tori-over-p-adic-fields}
Let $K$ be a $p$-adic field and $C=\widehat{\overline{K}}$. We let
\[
R=\mathcal{O}_K \langle T_1^{\pm 1},\ldots, T_n^{\pm 1}\rangle\quad \text{and}\quad
A= \mathfrak{S} \langle u_1^{\pm 1},\ldots, u_n^{\pm 1}\rangle,\]
where $\mathfrak{S}$ is the Breuil-Kisin prism of \S\ref{sec:p-adic-fields-new-identification-of-galois-action-for-p-adic-fields} and the completion is $(p, E(u))$-adic. Set $u_0:=u$ and $T_0:=\pi$. We define a $\delta$-structure on $A$ extending that of $\mathfrak{S}$ by setting $\delta(u_i)=0$, $i=1,\dots, n$.  Then $(A,I=(E(u)))$ is a prism. We have a natural isomorphism $A/I \cong R$ given by $u_i\mapsto T_i$. We also set
\[
R_\infty:=\mathcal{O}_C \langle T_1^{\pm 1/p^\infty},\ldots, T_n^{\pm 1/p^\infty} \rangle
\quad \text{and}\quad
A_\infty:=A_{\inf}(R_\infty),
\]
so $R\to R_\infty$ is the usual perfectoid cover.
Choose $\pi^\flat=(\pi,\pi^{1/p},\dots) \in \mathcal{O}_C^\flat$ and 
$
T_i^\flat=(T_i,T^{1/p}_i,\ldots )\in R_{\infty}^\flat
$ for $i=1,\ldots, n$.
Sending $u \mapsto [\pi^\flat]$, $u_i\mapsto [T_i^\flat]$ then defines a map $\Varphi\colon A\to A_\infty$.
Consider now
\[
\Gamma := (\Z_p\gamma_1 \oplus\dots\oplus \Z_p\gamma_n) \rtimes \mathrm{Gal}(\overline{K}/K),
\]
the semi-direct product for the action of $\mathrm{Gal}(\overline{K}/K)$ on $\Z_p\gamma_1 \oplus\dots\oplus \Z_p\gamma_n$ given by $g\gamma_i g^{-1} = \gamma_i^{\chi(g)}$ (with $\chi$ the cyclotomic character) for $g\in \mathrm{Gal}(\overline{K}/K)$ and $i=1,\dots,n$. Then $\Gamma$ acts on $R_\infty$ by
$\gamma_i\cdot T_j^{1/p^k} = \zeta_{p^k}^{\delta_{i,j}} T_j^{1/p^k}
$
and the natural action of $\mathrm{Gal}(\overline{K}/K)$ on $\mathcal{O}_C$. This action leaves $R$ fixed and admits a natural lift to $A_\infty$ via $\gamma_i\cdot [T_j^\flat] = [\varepsilon^{\delta_{i,j}}] [T_j^\flat]$. We are thus once again in \Cref{setup:Galois-action}.
\begin{lemma}\label{p:descr-D-x-tori-over-p-adic-field}
	Let $\sigma = \gamma_1^{m_1}\dots \gamma_n^{m_n} g \in \Gamma$ and let $\tau:A_\infty\to A_\infty$ be the induced lift. Then the associated element $(D,x)\in G_A(R_\infty)$ of \Cref{sec:appl-group-acti-1-identification-of-sigma-via-group-action} can be described as follows: We have
	\[ x= 1+E^\prime(\pi)\cdot \pi \cdot z\cdot c(g)=\chi_{\pi^\flat}(g)\]
	as in \Cref{p:descr-D-x-p-adic-field}. The derivation $D$ sends $f=f(u_0,\dots,u_n)\in A=W(k)[[u]]\langle u_1^{\pm 1},\ldots, u_n^{\pm 1}\rangle$ to
	\[ \textstyle D(f)=E(u) \otimes  z\cdot \sum_{i=0}^n c_i(\sigma) T_i  \theta(\frac{\partial f}{\partial u_i})(T_0,T_1,\dots T_n).\]
	where $c_i(\sigma)$ is defined by $\sigma([T_i^\flat])=[\varepsilon]^{c_i(\sigma)}[T_i^\flat]$ for $i=1,\ldots, n$, and $z$ was defined in \Cref{p:descr-D-x-p-adic-field}.
\end{lemma}
\begin{proof}
 The computation of $x$ works exactly as in \Cref{p:descr-D-x-p-adic-field}. In order to compute the derivation $D$, we use that $\widehat{\Omega}^1_{A/\Z_p}=\oplus_{i=0}^nAdu_i$, so it suffices to describe the action on the generators $u_i\in A$.

For this we evaluate the formula from \Cref{t:expl-descr-of-action-on-GA} by using \Cref{l:std-form-for-computing-Dx} applied with $t=T_i$:
\[\textstyle D(u_i):=E(u) \otimes  \theta \big(\frac{\sigma([T_i^{\flat}])-[T_i^\flat]}{E([\pi^\flat])}\big)=E(u) \otimes  \theta \big(\frac{\sigma([T_i^{\flat}])-[T_i^{\flat}]}{\mu}\big)z=E(u)\otimes c_i(\sigma)\cdot T_i \cdot z.
\]
The description of $D(f)$ follows by linearity and continuity by the usual formula for derivatives.
\end{proof}
As in the previous subsections, we derive from \Cref{p:descr-D-x-tori-over-p-adic-field}, \Cref{sec:appl-group-acti-1-identification-of-sigma-via-group-action} and \Cref{sec:autom-overl-2-explicit-action-on-g-a} the following description of the $\Gamma$-action on the formal scheme $Z_A$ of \eqref{d:def-ZA}:
 
 \begin{proposition}\label{c:descr-tori-over-p-adic-field}
  We have \[Z_A\cong \Spf(\textstyle\widehat{\bigoplus\limits_{m_0,\ldots, m_n\geq 0}} R_\infty \frac{a_0^{m_0}}{m_0!}\ldots \frac{a_n^{m_n}}{m_n!}),\]
  where $a_i$ corresponds to $du_i\otimes 1\in \widehat{\Omega}^1_{A/\Z_p}\widehat{\otimes}_AR$. The $R_\infty$-semilinear action of $\Gamma$ on $Z_A$ is given by
\[\textstyle
\sigma(a_i)=\chi_{\pi^\flat}(\sigma)a_i+\sum_{j=0}^n c_j(\sigma)zT_j.
\]
for $\sigma\in \Gamma$  and $i=0,\ldots, n$. Here, $\chi_{\pi^\flat}(\sigma):=1+c_0(\sigma)E^\prime(\pi)\pi z$ can be regarded as the composition of the cocycle $\chi_{\pi^\flat}$ from \Cref{sec:p-adic-fields-remark-chi-pi-flat-similar-to-the-cyclotomic-character} with the projection $\Gamma\to \Gal(\overline{K}/K)$.
\end{proposition}
\begin{proof}
	The factor $(1+E^\prime(u_0)b_0)$ in \Cref{sec:autom-overl-2-explicit-action-on-g-a}  evaluates to $\chi_{\pi^\flat}(\sigma)$ by definition.
\end{proof}
\section{Galois cohomology and the proof of fully faithfulness}
\label{sec:introduction}

In this section, we verify \Cref{sec:form-reduct-prov-1-crucial-assumption-on-galois-cohomology} in the examples of \S\ref{sec:examples}. We start with the key calculation.

\subsection{The key calculation}
\label{sec:crucial-calculation}
Set $S:=\Z_p^\cycl\langle \kappa,T^{\pm 1}\rangle$ and $S_\infty:=\Z_p^\cycl\langle \kappa,T^{\pm 1/p^\infty}\rangle$ where $\kappa$ is a formal variable. Then $\Gamma:=\Z_p(1)$ acts in the usual way continuously and $S$-linearly on $S_\infty$. Fix a topological generator $\gamma=(1,\zeta_p,\zeta_{p^2},\ldots)\in\Gamma$ of $\Z_p(1)$. We define the $p$-complete, $p$-completely faithfully flat $S_\infty$-algebra
\[\textstyle
B^+:=\widehat{\bigoplus\limits_{n\geq 0}} ~ S_\infty\frac{x^n}{n!}
  \]
and the $\Q_p$-Banach algebra $B:=B^+\tf$.  For the fixed primitive $p$-th root of unity $\zeta_p$, set $c:=(\zeta_p-1)\cdot \kappa$. We now extend the natural $\Gamma$-action on $S_\infty$ to a continuous action on $B^+$ by setting
\[
\gamma(x):=x+c.
\]
Hence, the role of $\kappa$ is that it will later in \Cref{sec:tori-over-perfectoid-1-main-theorem-with-perfectoid-base} act as a placeholder recording a $\Gamma$-action on $B$. For example, sending $\kappa$ to $T$ will recover on $B$ the action from \Cref{c:descr-Galois-action-tori-over-perfectoid}.

We note that the $\Gamma$-action on $B^+$ is well-defined because for any $\sum\limits_{n\geq 0} a_n\tfrac{x^n}{n!}\in B^+$, we have
\begin{equation}\label{eq:expl-action-on-B}
	\textstyle
	\gamma(\sum\limits_{n=0}^\infty a_n\frac{x^n}{n!}) =  \sum\limits_{n=0}^\infty \gamma(a_n)\frac{(x+c)^n}{n!}=\sum\limits_{n=0}^\infty\gamma(a_n)\sum\limits_{k=0}^n\frac{x^k}{k!}\frac{c^{n-k}}{(n-k)!}=\sum\limits_{n=0}^\infty \left(\sum\limits_{m=0}^\infty \gamma(a_{n+m})\frac{c^{m}}{m!}\right)\frac{x^n}{n!}
\end{equation}
and the last sum converges in $B^+$ as $a_{m}\to 0$ for $m\to \infty$ and since the $p$-adic valuation
\[\textstyle v_p((\zeta_p-1)^n/n!)= \frac{n}{p-1}-\frac{n-s_p(n)}{p-1}=\frac{s_p(n)}{p-1}\geq \frac{1}{p-1}\]
is bounded below by Legendre's formula for $n\geq 1$ (here $s_p(n)$ is the sum of the digits of $n$ in base $p$).

\begin{theorem}\label{t:Galois-cohom-for-B}
	There is $n\in \N$ such that for any $p$-complete $S$-algebra $R$, the cofiber of the map
	\[
	R\to R\Gamma(\Gamma, B^+\widehat{\otimes}_{S} R)
	\] is killed by $(cp)^n$. Thus if $c$ maps to a unit in $R\tf$, this map is an isomorphism after inverting $p$.
\end{theorem}
\begin{proof}
	As $B^+$ is topologically free over $S$, we have $R\widehat{\otimes}_S B^+\cong R\widehat{\otimes}_S^{L}B^+$. Therefore the complex
	\[
	R\widehat{\otimes}_S B^+\xrightarrow{\gamma-1} R\widehat{\otimes}_S B^+
	\]computes 
	continuous $\Gamma$-cohomology and we can reduce to $R=S$.
	By \eqref{eq:expl-action-on-B} we  have
	\[ \textstyle
	(\gamma-1)\sum\limits_{n= 0}^\infty a_n\frac{x^n}{n!}=\sum\limits_{n=0}^\infty \left((\gamma-1)a_n+\sum\limits_{m=1}^\infty \gamma(a_{n+m})\frac{c^{m}}{m!}\right)\frac{x^n}{n!}\]
	for any $\sum\limits_{n=0}^\infty a_n\frac{x^n}{n!}\in B^+$. Consider now the subring
	$
	\textstyle B_0^+:=\widehat{\bigoplus\limits_{n\geq 0}} S \frac{x^n}{n!}\subseteq B^+
	$
	as well as the quotient
	$Q:=B^+/B_0^+\cong \widehat{\bigoplus\limits_{n\geq 0}} S_\infty/S \frac{x^n}{n!}$.
	As  $S\to R\Gamma(\Gamma, B^+)$ factors through $R\Gamma(\Gamma,B_0^+)$, it suffices to show:
	\begin{enumerate}
		\item the map $S\to R\Gamma(\Gamma, B_0^+)$ has cofiber killed by some $(cp)^n, n\geq 1$, that
		\item the complex $R\Gamma(\Gamma,Q)$ is killed by $(cp)^n$ for some $n\geq 1$. 
	\end{enumerate}
	We begin with proving (2). We may replace $\Gamma$ by $\Gamma^\prime=\Z_p\gamma^{p^m}$ as $R\Gamma(\Gamma,Q)=R\Gamma(\Gamma/\Gamma^\prime, R\Gamma(\Gamma^\prime,Q))$. This changes the effect of $\gamma$ on $x$ to $x\mapsto x+p^mc$. We may therefore without loss of generality make $|c|$ as small as we like. In terms of the orthonormal basis $(e_n:=\frac{x^n}{n!})_{n=0}^\infty$ relatively over the $\Z_p^\cycl$-Banach module $S_\infty/S$, we can represent $\gamma-1:Q\to Q$ as the infinite upper triangular matrix
	\[M=\left(
	\begin{matrix}
		\gamma-1 & \gamma\cdot c & \gamma\cdot \frac{c^2}{2!} & \gamma\cdot \frac{c^3}{3!} & \cdots & \gamma\cdot \frac{c^n}{n!} & \cdots         \\
		& \gamma-1      & \gamma \cdot c             & \gamma\cdot \frac{c^2}{2!} &        & \gamma\cdot \frac{c^{n-1}}{(n-1)!} & \cdots  \\
		&               & \ddots                  &               &        &                                    &
	\end{matrix}\right)\]
	\begin{lemma}
		\label{sec:key-calculation-bounded-p-torsion-in-kernel-and-cokernel}
		The kernel and cokernel of the map $\gamma-1\colon S_\infty/S\to S_\infty/S$ are bounded $p^\infty$-torsion.
	\end{lemma}
	\begin{proof}
		This follows  from \cite[Lemma 5.5]{scholze2013p} by base change along $\Z^\cycl_p\langle T^{\pm 1}\rangle \to S$.
	\end{proof}
The map $\gamma-1\colon S_\infty/S\tf\to S_\infty/S\tf$ of $\Q_p$-Banach spaces thus admits a continuous $S$-linear inverse $\rho$. This induces an $S$-linear continuous operator on $Q\tf$ with matrix representation given by the diagonal matrix of $\rho$. Then  $M\cdot \rho=1+U$ where $U$ is an upper triangular matrix. Making $c$ smaller if necessary, we can assume that $U$ has entries of absolute value $<1$, converging to $0$ in each row. Then $\sum_{m=0}^\infty (-U)^m$ is an inverse to $1+U$. Thus $M\cdot \rho\colon Q\tf\to Q\tf$ is invertible, and therefore so is $M\colon Q\tf\to Q\tf$. By the Banach open mapping theorem, statement (2) follows.
	
	For (1), we argue similarly: $\gamma$ is an $S$-linear endomorphisms of the Banach $S$-module $B_0^+$, and the matrix representing $\gamma-1$ on $B_0^+$ with respect to the orthogonal basis $(e_n:=\frac{x^n}{n!})_{n=0}^\infty$ over $S$ has zeros on and below the diagonal. In particular, $S\subseteq B_0^+$ being  fixed by $\Gamma$, there is a copy of $S$ in the kernel. 
	
	We now consider the induced map $\phi\colon B_0^+/S\to B_0^+$ and claim that this is an isomorphism up to $(cp)^n$-torsion for some $n\geq 1$. Take as a basis in $B_0^+/S$ the elements $e_1, e_2,\ldots$ and in $B_0^+$ the elements $e_0,e_1,\ldots$, then the resulting matrix $M$ representing $\phi$ is now of the form:
	
	\[\left(
	\begin{matrix}
		c & \frac{c^2}{2!} & \frac{c^3}{3!} & \cdots & \frac{c^n}{n!}         \\
		& c              & \frac{c^2}{2!} &        & \frac{c^{n-1}}{(n-1)!} \\
		&                &                & \ddots &                       
	\end{matrix}\right)\]
	This is $c$ times  a unipotent matrix, and after shrinking $c$ by the same trick as before, we can assume that the strictly upper diagonal entries are in $pS$. Then $M=c(1+U)$ for a matrix $U$ such that $\sum\limits_{m=0}^\infty (-U)^m$ converges. Then $M\colon B^+_0/S\to B^+_0$ is an isomorphism up to $c$-torsion as claimed.	
\end{proof}

\subsection{Fully faithfulness in the smoothoid case}
\label{sec:tori-over-perfectoid}
We recall the following definition from \S\ref{sec:introduction-1}:
\begin{definition}
  \label{sec:fully-faithf-smooth-definition-smoothoid}
  A $p$-adic formal scheme $X$ is \textit{smoothoid} if $X$ is locally smooth over some perfectoid $p$-adic formal scheme $X_0$. 
\end{definition}

This is a variant for formal schemes of the ``smoothoid adic spaces'' introduced in \cite[\S2]{heuer-sheafified-paCS}. More precisely, the adic generic fibre $\X$ of any smoothoid formal scheme is such a smoothoid adic space, in particular $\X$ is sousperfectoid and hence sheafy. Like in \cite[Definition~2.10]{heuer-sheafified-paCS}, there is a good notion of a global sheaf of differentials on any smoothoid $p$-adic formal scheme:

\begin{lemma}
  \label{sec:fully-faithf-smooth-differentials-are-smoothoid}
  \begin{enumerate}
  \item Let $Z\to Y$ be a morphism of perfectoid $p$-adic formal schemes, then the (automatically $p$-completed) cotangent complex $L_{Z|Y}$ vanishes.
    \item If $X$ is any smoothoid formal scheme, then $H^0(L_{X|\Z_p})=\Omega^1_{X|\Z_p}$ is a finite, locally free sheaf. 
    \end{enumerate}
    \end{lemma}
    \begin{proof}
      We first prove (1). We may assume that $Z=\Spf(T), Y=\Spf(S)$ are affine. Let $(A,I)$ resp.\ $(B,J)$ be the perfect prisms associated with $S$ resp.\ $T$. Then
      $
        L_{Z/Y}=L_{T|S}^\wedge\cong L_{B|A}^\wedge\widehat{\otimes}_{B}^L T$
      for the $p$-completed cotangent complexes and $L_{B|A}^\wedge\otimes_{B}^L B/p\cong L_{(B/p)|(A/p)}$ vanishes as $B/p, A/p$ are perfect.
      
      For (2), we may work locally and assume that there is a smooth morphism $f:X\to S$ where $S$ is a perfectoid formal scheme. Consider the transitivity triangle
      \[
        \mathcal{O}_X\widehat{\otimes}^L_{\mathcal{O}_{S}}{L_{S|\Z_p}} \to L_{X|\Z_p}\to L_{X|S}.
      \]
     The first term has vanishing $H^0$ since $S$ is perfectoid. The last term is $\Omega^1_{X|S}[0]$ since $f$ is smooth.
   \end{proof}
   
\begin{definition}\label{def:absolute-diff-of-smoothoid}
	Let $X$ be a smoothoid formal scheme.
	Motivated by \Cref{sec:fully-faithf-smooth-differentials-are-smoothoid}, we simply write $\Omega^1_{X}:=\Omega^1_{X|\Z_p}$.
If there is a smooth morphism $X\to S$ to a perfectoid space, then $\Omega^1_X=\Omega^1_{X|S}$, but since we usually do not make the local perfectoid base $S$ explicit, we drop this from the notation. 
\end{definition}
    We now establish \Cref{sec:introduction-2-statement-main-theorem} in the smoothoid case, i.e., prove the following theorem.

\begin{theorem}
	\label{sec:tori-over-perfectoid-1-main-theorem-with-perfectoid-base}
	Assume that $X$ is a qcqs smoothoid formal scheme. Then the functor
	\[
	\alpha_X^\ast\colon \mathcal{P}erf(X^\HT)\tf\to \mathcal{P}erf(\X_v)
	\]
	is fully faithful, where $\X$ is the adic generic fibre of $X$.
\end{theorem}
\begin{proof}
  By \Cref{sec:crit-fully-faithf-main-theorem-localizes-on-x} we may assume that $X$ is affine and even \'etale over some torus $\mathbb{T}_{R_0}^n:=\Spf(R_0\langle T_1^{\pm 1},\ldots, T_n^{\pm 1}\rangle)$ over a perfectoid base $\Spf(R_0)$. By Andr\'e's lemma (\cite[Th\'eor\'eme~2.5.1]{andre2018conjecture}, \cite[Theorem 7.14]{Bhatta} and \cite[Remark 7.15]{Bhatta}) there exists a quasi-syntomic cover $R_0\to S_0$ with $S_0$ a perfectoid $\Z_p^\cycl$-algebra. Let $S_\bullet$ be the Cech nerve of $S_0$ over $R_0$. Then $X^\HT\times_{\Spf(R_0)}\Spf(S_n)\cong (X\times_{\Spf(R_0)}\Spf(S_n))^\HT$ for any $n\geq 0$, each $S_n$ is a $\Z_p^\cycl$-algebra and each $X\times_{\Spf(R_0)}\Spf(S_n)$ is smoothoid. Contemplating this Cech nerve further shows that it is sufficient to prove \Cref{sec:tori-over-perfectoid-1-main-theorem-with-perfectoid-base} for all $X\times_{\Spf(R_0)}\Spf(S_n), n\geq 0$. Indeed, we can commute inverting $p$ with the inverse limit calculating $R\Hom$ because the terms of the inverse limit are uniformly bounded. 
  
  Hence, we may assume that $R_0$ is a $\Z_p^\cycl$-algebra. By \Cref{sec:crit-fully-faithf-assumption-passes-to-ind-etale-extensions} we may reduce to $X=\mathbb{T}_{R_0}^n$ by choosing $(A,I), (A_\infty, I_\infty), \Gamma$ as in \S\ref{sec:tori-over-geometric}. We are then in the setup of \S\ref{sec:tori-over-geometric} (by \cite[Lecture 4, Proposition 3.2]{bhatt_lectures_on_prismatic_cohomology} we may assume that $R_0$ is $p$-torsionfree because $R_0/\sqrt{pR_0}$ is perfectoid and both sides of the statement only depend on this $p$-torsionfree quotient)  and use the notations introduced there. We want to verify that \Cref{sec:form-reduct-prov-1-crucial-assumption-on-galois-cohomology} holds in this setup, namely we claim that the map
\[
R=R_0\langle T^{\pm 1}_1,\ldots, T_n^{\pm 1}\rangle \to R\Gamma(\Gamma, B^+_{A,R_\infty})
\]
has cofiber killed by $p^i$ for some $i\geq 1$, where by \S\ref{sec:tori-over-geometric},
\[
\textstyle B^+_{A,R_\infty}\cong \O(G_{A,R_\infty})=\widehat{\bigoplus\limits_{m_1,\ldots, m_n\geq 0}}R_\infty \frac{a_1^{m_1}}{m_1!}\cdots \frac{a_n^{m_n}}{m_n!}
\]
with $\Gamma=\Z_p(1)^n$ acting $R_\infty$-semilinearly as described in \Cref{c:descr-Galois-action-tori-over-perfectoid}. If $n=1$, the claim follows from applying \Cref{t:Galois-cohom-for-B} to the map $S \to R$, sending $\kappa$ to $T_1$. 

More generally, we can argue by induction: Let $\Z_p(1)\cong H_0\subseteq \Z_p(1)^n$ be the subgroup given by the last coordinate. Then by the Hochschild--Serre spectral sequence, we have 
\[R\Gamma(\Gamma, B^+_{A,R_\infty})=R\Gamma(\Gamma/H_0,R\Gamma(H_0, B^+_{A,R_\infty})).\]
Let us temporarily denote $B^+_{A,R_\infty}$ by $B^{+,(n)}_{A,R_\infty}$ to indicate the dependence on $n$. Then
applying \Cref{t:Galois-cohom-for-B} to the map
\[ \textstyle S\to B^{+,(n-1)}_{A,R_\infty}\langle T_n^{\pm 1}\rangle, \quad \kappa\mapsto  T_n,
\]
we see that the natural map
\[B^{+,(n-1)}_{A,R_\infty}\to R\Gamma(H_0, B^{+,(n)}_{A,R_\infty})\]
has cofiber killed by $p^k$ for some $k$. Inductively, this shows the claim.

 Hence we can conclude by \Cref{sec:form-reduct-prov-1-fully-faithfulness-for-alpha}. 
\end{proof}

\subsection{Fully faithfulness in the arithmetic case}
\label{sec:smooth-form-schem-galois-cohomo-in-arithmetic-case}

Let $K$ be a $p$-adic field. We retain the setup and notation of \S\ref{sec:tori-over-p-tori-over-p-adic-fields}. In particular, we have $R_\infty:=\mathcal{O}_C \langle T_1^{\pm 1/p^\infty},\ldots, T_n^{\pm 1/p^\infty}\rangle$.
\begin{proposition}\label{p:assumption-2.2-satisfied-arithm-case}
	Consider $\O(G_{A,R_\infty})$ with its natural $\Gamma$-action as described in  \S\ref{sec:tori-over-p-tori-over-p-adic-fields}.
Then the cofiber of the natural map
\[
R\to R\Gamma(\Gamma, \O(G_{A,R_\infty}))
\]
is bounded and killed by $p^i$ for some $i\geq 1$.
\end{proposition}
\begin{proof}
 Using \Cref{t:Galois-cohom-for-B} we can as a first step argue exactly as in the proof of  \Cref{sec:tori-over-perfectoid-1-main-theorem-with-perfectoid-base} to take care of the Tate variables $T_1,\ldots, T_n$ and thus reduce to the case that $R=\mathcal{O}_K$. Then
\[
\textstyle \O(G_{A,R_\infty})\cong\widehat{\textstyle\bigoplus\limits_{n\geq 0}} \O_C\cdot \tfrac{a^n}{n!}
\]
by \Cref{c:Galois-action-on-ZA-p-adic-fields}.
The calculation is now reduced to showing that for $H:=\Gal(C|K)$, the cofiber of the natural map
$
\O_K\to R\Gamma(H,\O(G_{A,R_\infty}))$ is killed by $p^i$ for the $H$-action as in \S\ref{sec:tori-over-p-tori-over-p-adic-fields}. But by \cite[Theorem~3.12]{analytic_HT}, this map is an isomorphism after inverting $p$, so the desired statement follows by the Open Mapping Theorem. 
\end{proof}
We have thus established \Cref{sec:form-reduct-prov-1-crucial-assumption-on-galois-cohomology} for $R=\mathcal{O}_K\langle T_{1}^{\pm 1},\ldots, T_{n}^{\pm 1}\rangle$. We can deduce:

\begin{theorem}
	\label{sec:smooth-form-schem-fully-faithfulness-for-arithmetic-base}
	Let $K$ be a $p$-adic field. For any qcqs smooth morphism $X\to \Spf(\O_K)$, the functor
	\[
	\alpha_X^\ast\colon \mathcal{P}erf(X^\HT)\tf\to \mathcal{P}erf(\X_v)
	\]
	from \S\ref{sec:form-reduct-prov} is fully faithful, where $\X$ is the adic generic fibre of $X$.
\end{theorem}
\begin{proof}
	As in \Cref{sec:tori-over-perfectoid-1-main-theorem-with-perfectoid-base} we can reduce to the case that $X=\Spf(\O_K\langle T_1^{\pm 1},\ldots, T_{n}^{\pm n}\rangle)$, where we have established \Cref{sec:form-reduct-prov-1-crucial-assumption-on-galois-cohomology} in \Cref{p:assumption-2.2-satisfied-arithm-case}. Thus, we can conclude by \Cref{sec:form-reduct-prov-1-fully-faithfulness-for-alpha}. 
      \end{proof}

      \begin{remark}
        \label{sec:fully-faithf-arithm-remark-on-imperfect-residue-field-case}
        \Cref{sec:smooth-form-schem-fully-faithfulness-for-arithmetic-base} can be extended to qcqs smooth formal schemes over a complete $p$-adic discrete valuation ring $\mathcal{O}_K$ with $p$-finite residue field, cf.\ \cite[\S.4]{Hyodo-HT-imperfect-res}, using very similar calculations. More precisely, analogs to the Breuil-Kisin prisms have been constructed for $\mathcal{O}_K$ in \cite{gao2020integral} and following \cite{Hyodo-HT-imperfect-res} one can construct the necessary concrete perfectoid extension $K_\infty$ of $K:=\mathrm{Frac}(\mathcal{O}_K)$, analogous to the cyclotomic extension. There again exists Tate's normalized traces (\cite[\S 4.3, Ingredient 2]{gao2020integral}, \cite[Lemma 6.5]{OhkuboSenTheory}, \cite[\S 3]{Hyodo-HT-imperfect-res}) and thus the proof \cite[Theorem 3.12]{analytic_HT} goes through for $K$. 
      \end{remark}

\section{Complexes on the Hodge--Tate stack and Higgs modules}
\label{sec:complexes-Hodge--Tate-and-higgs-bundles}

In this section we want to prove \ref{sec:smoothoid-case-1-complexes-on-x-ht-smoothoid-case-introduction}, \ref{t:intro-local-p-adic-Simpson-functor-geometric} and \ref{sec:introduction-2-description-of-ht-stack} from the introduction. Since fully faithfulness of the functor $\alpha_X^\ast$ is now proved, this is mainly a question of describing complexes on the Hodge--Tate stack (and some variants of it) for suitable formal schemes $X$ explicitly in terms of (derived) Higgs or Higgs--Sen bundles. As the structure of the stacks differs in the smoothoid and in the arithmetic case, our discussion will naturally take a different form in these two cases.

However, the general principle is the same in all settings. Indeed, in all settings of \S\ref{sec:examples}, the Hodge--Tate stack $X^\HT$ is (after some choices isomorphic to) the classifying stack of some formal group scheme $G:=G_A$ over $X$, which has the explicit shape of being an extension 
\[
1\to V^\sharp\to G\to H\to 1
\]
(usually split)
with $V^\sharp$ the PD-hull of the zero section of a (geometric) vector bundle $V$ over $X$ and $H=\{1\}$ or  $H\cong \Gm^\sharp$ or $H\cong \Ga^\sharp$, which acts by multiplication on $V^\sharp$ via the natural morphism
\[
X^\HT\to \Spf(\Z_p)^\HT\cong B\Gm^\sharp.
\]
We will therefore now study the representation theory of $G$. While wo do not seek maximal generality, our arguments extend to more general situations ($BH$ could be some $p$-adic formal stack $\mathcal{Z}$, and $V^\sharp$ the PD-envelope of the zero-section of some (geometric) vector bundle over $\mathcal{Z}$).

\subsection{Representations of $G$}
\label{sec:representations-g_a}
Much of the material of this section is well-known,  though maybe not exactly in the form presented here. In particular, the following discussion is closely related to \cite[Section 3.5]{bhatt2022absolute}, \cite[Lemma 6.7]{bhatt2022prismatization} and \cite[Section 2.4]{bhatt2022F-gauges}.

We begin by fixing the general setup and introducing some notation.
Let $R$ be a $p$-complete ring with bounded $p^\infty$-torsion.
\begin{definition}\label{def:Sym-bullet-notation}
	 For any finite projective $R$-module $W$, we write
	 \[\Symv(W):=\mathrm{Sym}^\bullet_R(W), \quad \Symvw(W)=\Symv(W)^\wedge_p, \quad \mathbb V(W):=\mathrm{Spf}(\Symvw(W)).\]
	 Then $\mathbb V(W)$ is a geometric vector bundle over $\Spf(R)$.
\end{definition}

 Let $E$ be a finite projective $R$-module.
 Let \[V:=\mathbb V(E),\quad A:=\Gamma_R^\bullet(E)^\wedge_p,\quad V^\sharp:=\Spf(A),\] 
 then $A$ is the $p$-completed PD-hull of the zero section of $V$.

We also give ourselves a $p$-completely flat affine group scheme $G$ over $R$ which is an extension
\[
1\to V^\sharp\to G\to H\to 1,
\]
where $H$ is either $\{1\}, \Gm^\sharp$ or $\Ga^\sharp$, or an extension of such by $U^\sharp$, for another geometric vector bundle $U$. We want to describe the representations of $G$. By \cite[Theorem 2.5]{analytic_HT},
the representation theory of $H$  is understood. We now wish to describe $\mathcal{D}(BG)$ in terms of $V^\sharp$ and $\mathcal{D}(BH)$.

The idea for describing the category $\mathcal{D}(BG)$ of the formal stack $BG$ over $\Spf(R)$ is quite simple.
By $p$-completely faithfully flat descent, $\mathcal{D}(BG)$ is equivalent to the category of $\O(G)$-comodules in the category $\mathcal{D}(\Spf(R))=\widehat{\mathcal{D}(R)}$ of derived $p$-complete complexes of $R$-modules.

\begin{remark}
  \label{sec:representations-g-1-underived-case}
Let us first explain the strategy in a simple, underived special case: assume for the moment that $H=\{1\}$, so $G=V^\sharp$, and let $M$ be a finite projective module over $R$ endowed with the structure of an $\O(G)$-comodule. Explicitly, this means that there is a coaction 
$M\to M\otimes_R \O(G)$.
Then this can be dualized to an action on $M$ of the $R$-algebra \[\O(G)^\vee:=\Hom_{R}(\O(G),R)=\Hom_{R}(\Gamma^\bullet_R(E),R),\] 
which is naturally a Hopf algebra over $R$ by dualizing the multiplication and comultiplication on $\O(G)$. Now by \cite[Proposition A10]{berthelot2015notes}, the dual $\O(G)^\vee$ identifies with the (already $p$-complete) $(E^\vee)$-adic completion $\widehat{\mathrm{Sym}^\bullet_R}(E^\vee)$ of the symmetric algebra $\mathrm{Sym}^\bullet_R(E^\vee)$, i.e., the power series algebra on $E^\vee:=\Hom_R(E,R)$. Hence the $\O(G)$-coaction on $M$ is equivalently given by some morphism
\[
E^\vee\otimes_R M\to M.
\]
The existence of an extension to $\widehat{\mathrm{Sym}^\bullet_R}(E^\vee)$ now enforces the condition that for each $\delta\in E^\vee$ the action $M\xrightarrow{\delta} M$ is topologically nilpotent for the natural topology on $M$ as a finite projective $R$-module. 
\end{remark}

If $H\neq \{1\}$, we aim for a  similar argument, but we need to work in the category $\mathcal{D}(BH)$ of representations of $H$. We therefore aim to prove the following. Recall that
$\Symvw(E^\vee)=\mathrm{Sym}^\bullet_R(E^\vee)^\wedge_p$.
We can regard this as an $E_\infty$-algebra in $\mathcal{D}(BH)$ via the natural $H$-action on $E^\vee$.
\begin{theorem}
	\label{sec:representations-g-representations-of-bg-over-bh}
	There exists a natural fully faithful functor
	\[
	\Phi_{BH}\colon \mathcal{D}(BG)\to \mathrm{Mod}_{\Symvw(E^\vee)}(\mathcal{D}(BH)),
	\]
	Its essential image is given by $M$ for which each $\delta\in E^\vee$ acts locally nilpotently on $H^\ast(M\otimes_{\Z_p}^L\F_p)$.
\end{theorem}

Geometrically, let $V^\vee:=\mathbb{V}(E^\vee)=\Spf(\mathcal{S}_p(E^\vee))$ be the dual (geometric) vector bundle of $V$ over $\Spf(R)$. Through the $H$-action on $E^\vee$, this naturally lives over the classifying stack $BH$, and we denote this geometric vector bundle over $BH$ by $\mathcal{V}^\vee\to BH$. The category $\mathrm{Mod}_{\Symvw(E^\vee)}(\mathcal{D}(BH))$ in \Cref{sec:representations-g-representations-of-bg-over-bh} is then equivalent to the category $\mathcal{D}(\mathcal{V}^\vee)$ because the morphism $\mathcal{V}^\vee\to BH$ is affine.

For the proof of \Cref{sec:representations-g-representations-of-bg-over-bh}, we require some preliminaries to implement the strategy of \cite[Theorem 3.5.15]{bhatt2022absolute}. 
The following calculation of the Cartier dual of $V^\sharp=\Spf(A)$ is well-known.

\begin{lemma}
	\label{sec:representations-g-cartier-dual-of-v-sharp}
	The functor $S\mapsto \Hom_{\mathrm{Grp}}(V^\sharp\times_{\Spf(R)}\Spf(S),\Gm)$ on $p$-complete $R$-algebras is represented by the formal group scheme
	\[
	\widehat{V^\vee}:=\Spf(B),
	\]
	where $B:=\Hom_{R}(A,R)$. Here 	$\widehat{V^\vee}$ is the formal completion of $V^\vee=\mathbb{V}(E^\vee)$ along its zero section.
\end{lemma}
\begin{proof}
	The functor $V^\sharp$ sends a discrete $R$-algebra $S$ to the $S$-module of $R$-linear maps $x\colon E\to S$ together with divided powers $\frac{x(e)^n}{n!}$ for each $e\in E$ and $n\geq 0$.
	The functor $\widehat{V^{\vee}}$ sends a $p$-nilpotent $R$-algebra $S$ to the $S$-module of $R$-linear morphisms $a\colon E^\vee\to S$ such that $a(\delta)\in S$ is nilpotent for each $\delta\in E^\vee$.
	Now the duality is induced by the pairing
	\begin{equation}\label{eq:explicit-duality}
	\textstyle\widehat{
	V^{\vee}}\times_{\Spf(R)} V^\sharp\to \Gm,\ (a,x)\mapsto \mathrm{exp}(\langle a, x\rangle):=\sum\limits_{n=0}^\infty \frac{(x(a))^n}{n!},
	\end{equation}
	where $a\colon E^\vee\to S$ is seen as an element in $E\otimes_RS$ and $\frac{x(a)^n}{n!}$ is extended in its unique fashion.
\end{proof}

Note that the topology on $B$ is not $p$-adic: In fact
$
\textstyle B\cong \prod_{n\geq 0} \mathrm{Sym}^n_R(E^\vee)$,
and the topology is the product of the $p$-adic topologies on $\mathrm{Sym}^n_R(E^\vee)$, or in other words the $(p,\mathrm{Sym}^{>0}_R(E^\vee))$-adic topology.
\\

We will define $\Phi_{BH}$ as a sort of Fourier transform (in the spirit of \cite{laumon1996transformation}). For any $a\in \widehat{V^{\vee}}(S)$, let 
\[\chi_a\colon V^\sharp\times_{\Spf(R)}\Spf(S)\to \Gm,\ x\mapsto \mathrm{exp}(\langle a,x\rangle)\] be the associated character. Then we obtain the natural pairing
\[
\widehat{V^{\vee}}\times_{\Spf(R)} BV^\sharp\to B\Gm,\ (a,\mathcal{T})\mapsto \chi_{a,\ast}(\mathcal{T}):=\Gm\times^{V^\sharp}\mathcal{T}
\]
via pushing forward $V^\sharp$-torsors along $\chi_a$. Let $\mathcal{L}$ be the pullback of the tautological line bundle on $B\Gm$ along this pairing.
We can view $\mathcal{L}$ as an $A$-comodule in $\mathcal{D}(\Spf(B))$. As $a\mapsto \mathrm{exp}(\langle a,0\rangle)$ is the zero map, the underlying $B$-module of $\mathcal{L}$ is trivial.
Moreover, $B=\Hom_{R}(A,R)$ is equipped with the $V^\sharp$-action on the dual. More precisely, we can write the regular representation $A=\varinjlim_{n} A_{\leq n}$ with $A_{\leq n}:=\sum_{m\leq n} \Gamma_R^m(E)$ as a colimit of $V^\sharp$-stable subrepresentations, which are finite projective over $R$ (as in the proof of \cite[Theorem 3.5.15]{bhatt2022absolute}, whose computation also explains why the $A_{\leq n}$ are $V^\sharp$-stable). Now, the coactions
$
\Hom_R(A_{\leq n},R)\to \Hom_R(A_{\leq n},R)\otimes_R A
$
combine to a coaction
\begin{equation}\label{eq:coaction-of-A-on-B}
B\to B\widehat{\otimes}_RA,
\end{equation}
where the tensor product is completed with respect to the adic topology on $B$ (and not just $p$-adic). Let $Z\subseteq \widehat{V^\vee}$ be the zero section. Note that $Z$ is cut out by a regular sequence if $E$ is finite free.

Let $d$ be the rank of $E$, assumed to be constant. We denote by
\[
\mathrm{pr}_{1}\colon \widehat{V^\vee}\times BV^\sharp\to \widehat{V^\vee} \quad \text{and }\quad
\mathrm{pr}_{2}\colon \widehat{V^\vee}\times BV^\sharp\to BV^\sharp
\]
 the projections, where the fibre products are over $\Spf(R)$. 
 
 In the following, let us denote by $\mathcal{D}(\widehat{V^\vee})$ the derived category of quasi-coherent sheaves on the formal scheme $\widehat{V^\vee}$, i.e., derived $(p,\mathrm{Sym}^{>0}_R(E^\vee))$-adically complete $B$-modules.

\begin{definition}
	\label{sec:representations-g-definition-of-functor-phi}
	We define the functor
	\[
	\Phi_{\Spf(R)}\colon \mathcal{D}(BV^\sharp)\xrightarrow{F_R} \mathcal{D}(\widehat{V^\vee})\xrightarrow{R\Gamma_Z} \mathcal{D}(V^\vee)\cong \mathrm{Mod}_{\Symvw(E^\vee)}(\mathcal{D}(\Spf(R)))
	\]
	as the composition of the following two functors: The first is the ``Fourier transform''
	\[
	F_R: \mathcal{D}(BV^\sharp)\to \mathcal{D}(\widehat{V^\vee}),\quad M\mapsto R\mathrm{pr}_{1\ast}(\mathrm{pr}^\ast_{2}(M)\widehat{\otimes} \mathcal{L}\widehat{\otimes}_{R} \det(E^\vee)[d]),
	\]
	where the first $\widehat{\otimes}$ is over $\mathcal{O}_{\widehat{V^\vee}\times BV^\sharp}$. The second functor in the definition of $\Phi_{\Spf(R)}$ is the functor
	\[
	R\Gamma_Z\colon \mathcal{D}(\widehat{V^\vee})\to \mathcal{D}(V^\vee).
	\]
	of cohomology with support in $Z$. This is the derived $p$-completion of the usual functor of local cohomology as defined in \cite[Tag 0952]{StacksProjectAuthors2017}. Explicitly, according to  \cite[Tag 0952]{StacksProjectAuthors2017}, for a set of generators $\delta_1,\dots,\delta_r\in \mathrm{Sym}^{1}_R(E^\vee)$ of $I:=\mathrm{Sym}^{>0}_R(E^\vee)$, we can define this as
	\[\textstyle R\Gamma_Z(K):=R\lim_n\Big([B/p^n\to \prod_iB_{\delta_i}/p^n\to \prod_{i,j}B_{\delta_i\delta_j}/p^n\to \dots \to B_{\delta_1\cdots \delta_r}/p^n]\otimes^{L}_BK\Big) \]
\end{definition}

It is straightforward to get a $p$-completed version of \cite[Tag 0A6X]{StacksProjectAuthors2017} for this:
\begin{lemma}
	The functor $R\Gamma_Z\colon \mathcal{D}(\widehat{V^\vee})\to \mathcal{D}(V^\vee)$ is fully faithful. Its essential image is given by those $M\in \mathrm{Mod}_{\Symvw(E^\vee)}(\mathcal{D}(\Spf(R)))$ for which each $\delta\in E^\vee$ acts locally nilpotently on $H^\ast(M\otimes_{\Z_p}^{L}\F_p)$.
\end{lemma}
\begin{proof}
	For every $n \in \N$, let $\widehat{V^\vee_n}:=\widehat{V^\vee}\times \Spec(\Z/p^n)$ as well as $V^\vee_n:=V^\vee\times \Spec(\Z/p^n)$ and $Z_n:=Z\times \Spec(\mathbb Z/p^n)$. Then we have exact functors
	\[ R\Gamma_{Z_n}:\mathcal D(\widehat{V^\vee_n})\to \mathcal D(V^\vee_n)=\mathrm{Mod}_{\Symv(E^\vee)/p^n}(\mathcal{D}(R/p^n))\]
	and by definition we have 
	$R\Gamma_Z(K)= R\lim_nR\Gamma_{Z_n}(K\otimes^L \Z/p^n)$.
	Here we define $\mathcal{D}(\widehat{V^\vee_n})=\widehat{D(B/p^n)}$ as the full subcategory of $D(B/p^n)$ of derived complete objects for the adic topology on $B/p^n$.
	By  \cite[Tag 0A6X]{StacksProjectAuthors2017}, the functor  $R\Gamma_{Z_n}$ is then fully faithful with essential image given by the complexes $M$ which are torsion for the ideal $\mathrm{Sym}^{>0}_R(E^\vee)/p^n\subseteq \mathcal S(E)/p^n$. Equivalently, this means that the action of each $\delta\in E^\vee$ is locally nilpotent on $H^\ast(M)$. As usual, via the exact triangle \[M\otimes^L\Z/p^{n-1}\to M\to M\otimes^L\Z/p\]
	 we see inductively that this is equivalent to asking that each $\delta\in E^\vee$ acts locally nilpotently on $H^\ast(M\otimes^L\Z/p)$. Using that $R\Gamma_{Z_{n}}(M)\otimes^L\Z/p^{n-1}=R\Gamma_{Z_{n-1}}(M\otimes^L\Z/p^{n-1})$, we get the desired description in the limit over $n$, as $\mathrm{Mod}_{\Symvw(E^\vee)}(\mathcal{D}(\Spf(R)))$ embeds into
	 $\mathrm{Mod}_{\Symv(E^\vee)}(\mathcal{D}(\Spec(R)))$ as the full subcategory of derived $p$-complete objects. Similarly, the full faithfulness follows from that of $R\Gamma_{Z_n}$ using first that for any two derived complete complexes $N$ and $M$ in $\mathcal D(\widehat{V^\vee})$ we have \[\mathrm{RHom}_{\mathcal D(\widehat{V^\vee})}(M,N)=R\lim_n\mathrm{RHom}_{\mathcal D(\widehat{V^\vee_n})}(M\otimes^L\Z/p^n,N\otimes^L\Z/p^n),\]
	 then full faithfulness of $R\Gamma_{Z_n}$, and finally the analogous equality in $\mathcal{D}(V^\vee)$.
\end{proof}
\begin{definition}
As the functor $\Phi_{\Spf(R)}$ is natural in $\Spf(R)$, it descends to a functor
	\[
\Phi_{BH}\colon \mathcal{D}(BG)\to \mathrm{Mod}_{\Symvw(E^\vee)}(\mathcal{D}(BH)),
\]
 for the stack $BH$ and the $BV^\sharp$-gerbe $BG$ over it. This will be the functor mentioned in \Cref{sec:representations-g-representations-of-bg-over-bh}. 
\end{definition}

Next we want to make $\Phi_{\mathrm{Spf}(R)}$ explicit. Let $M\in \mathcal{D}(BV^\sharp)$. By construction, 
\[
\Phi_{\Spf(R)}(M)=R\Gamma_Z(R\Gamma(BV^\sharp,M\widehat{\otimes}_R^L B\widehat{\otimes}_R \det(E^\vee)[d])),
\]
where the tensor product is $(p,I)$-adically completed. In this formula, the $BV^\sharp$-module structure  of $M\widehat{\otimes}_R^L B\widehat{\otimes}_R \det(E^\vee)[d]$ is the diagonal $A$-comodule structure given by the $A$-comodule structure on $M$ and the $A$-comodule structure on $B$ explained in \eqref{eq:coaction-of-A-on-B}. The right-hand side carries the $B$-module structure induced by the $B$-module structure on $M\widehat{\otimes}_R^L B\widehat{\otimes}_R \det(E^\vee)$ (with $M$ seen here as an $R$-module). Now,
\[
\begin{array}{rl}
	 R\Gamma_Z(R\Gamma(BV^\sharp,M\widehat{\otimes}_R^L B\widehat{\otimes}_{R}\det(E^\vee)[d]))
	&= R\Gamma(BV^\sharp, R\Gamma_Z(M\widehat{\otimes}_R^L B\widehat{\otimes}_R\det(E^\vee)[d])) \\
	&=  R\Gamma(BV^\sharp, M \widehat{\otimes}_R R\Gamma_Z(B \widehat{\otimes}_R \det(E^\vee))[d]).
\end{array}
\]
Here, the last tensor product is $p$-completed.
To continue, we need the next lemma. 

\begin{lemma}
	\label{sec:representations-g-rgamma-z-for-b-}
	We have $R\Gamma_Z(B \widehat{\otimes}_R \det(E^\vee))\cong A[-d]$ with the regular $A$-coaction and its natural $B$-action $B\widehat{\otimes}_R A\xrightarrow{\mathrm{coact} \otimes \Id} B\widehat{\otimes}_R A\widehat{\otimes}_R A\xrightarrow{\mathrm{eval}\otimes \Id} A$. 
\end{lemma}
\begin{proof}
  Consider the projective bundle $Y:=\mathbb{P}(E^\vee)\to \Spf(R)$, which parametrizes line bundle quotients of $E^\vee$. Let $\mathcal{O}(1)$ be the universal quotient on $Y$. Then there exists a canonical isomorphism $E^\vee\cong H^0(Y,\mathcal{O}(1))$, and thus a canonical isomorphism $B\cong \prod_{n\in \Z} H^0(Y,\mathcal{O}(n))$. 
  
 One the other hand, we claim that we have a canonical isomorphism
  \[\textstyle R\Gamma_Z(B\widehat{\otimes}_R \det(E^\vee))= \widehat{\bigoplus}_{n\in \Z} H^{d-1}(Y,\mathcal{O}(n)\otimes_R \det(E^\vee))[-d].\]
  To see this, we may work locally on $R$ and assume that $E$ is free, that $r=d$ and that $\delta_1,\dots,\delta_r$ are a basis of $E$. Then applying \cite[Tag 0913]{StacksProjectAuthors2017} to the complex in the definition of $R\Gamma_Z(-)$ shows
  \[ \textstyle R\Gamma_Z(B)=\varprojlim_{m}(\varinjlim_k  (\delta_1^{-k}\cdots \delta_d^{-k}B)/B)/p^m[-d]\]
  where the inverses are taking inside $B_{\delta_1\cdots\delta_d}$.
  The usual computation of cohomology on $\mathbb P^{d-1}$ identifies the $R$-submodule of $(\delta_1^{-k}\cdots \delta_d^{-k}B)/B$ spanned by monomials of degree $n\geq -k$ with $H^{d-1}(Y,\mathcal{O}(n))$. As this isomorphism is natural in linear transformations in the variables $\delta_1,\dots,\delta_d$, it is independent of the choice of basis and therefore glues to all of $\Spf(R)$.  Finally, since $\det(E^\vee)$ is an invertible $R$-module, $-\hat{\otimes}_R \det(E^\vee)$ commutes with $R\Gamma_Z(-)$. This proves the claim.
  
  The Euler sequence $0\to \Omega^1_{Y/R}\to E^\vee\otimes_R \mathcal{O}(-1)\to \mathcal{O}\to 0$ on $Y$ shows that the dualizing sheaf on $Y$ is $\det(E^\vee)\otimes_R \mathcal{O}(-d)$. By Serre duality on $Y$, we finally obtain the canonical isomorphism
  \[\textstyle\widehat{\bigoplus}_{n\in \Z} H^{d-1}(Y,\mathcal{O}(n)\otimes_R \det(E^\vee))= \widehat{\bigoplus}_{n\in \Z}H^0(Y,\mathcal{O}(n))^\vee=\big(\prod_{n\in \Z}H^0(Y,\mathcal{O}(n)) \big)^\vee=B^\vee=A. \]
  
  It remains to see that the $A$-coaction is the regular one. For this we have to test the equality of two natural maps $A\to A\widehat{\otimes}_R A$, so we can reduce to the case that $E$ is free on a basis $x_1, \ldots, x_d$. Let $\delta_1,\ldots, \delta_d\in E^\vee\subseteq B$ be the dual basis. We note that each $A$-coaction on an $R$-module $M$ yields endomorphisms $\Theta_i$ obtained as the composition  of the coaction $M\to M\widehat{\otimes}_R A$ with $\mathrm{Id}\otimes \delta_i$. For the regular $A$-coaction, $\Theta_i=\partial_{x_i}$ is derivation with respect to $x_i$ (this follows from the Taylor formula, cf.\ \cite[Proof of Proposition 3.7.1]{bhatt2022absolute}). Using the above formula for $R\Gamma_Z(B\otimes_R \mathrm{det}(E^\vee))$ and contemplating the multiplication by $\delta_i$ on $A$ via the isomorphism to $R\Gamma_Z(B\otimes_R \mathrm{det}(E^\vee))$, shows that again $\delta_i$ acts by derivation with respect to $x_i$ (note that $\det(E^\vee)=R\delta_1\cdots \delta_d$). 
  
  By the same argument, we can more generally also identify the multiplication of each monomial in $B$. Together, these determine the $V^\sharp$-action. This completes the proof of the lemma.
\end{proof}

Consequently,
\begin{equation}\label{eq:explicit-descr-Phi}
	\Phi_{\Spf(R)}(M)= R\Gamma(BV^\sharp, M\widehat{\otimes}_RA[-d][d]) = M,
\end{equation}
where the last equality comes from the fact that $A$ is endowed with the regular $A$-coaction. The action of $\mathrm{Sym}^\bullet_R(E^\vee)\subseteq B$ is given by the adjoint of the coaction
$
M\to M\widehat{\otimes}_R A.
$

It is tempting to try to use this explicit formula to prove directly \Cref{sec:representations-g-representations-of-bg-over-bh}, by relating derived $A$-comodules and $B$-modules by adjunction. It is indeed possible, but we only know an argument using nuclear modules in the sense of \cite{clausenscholzecomplex}. Here we offer a different and simpler argument, closer to an argument of Bhatt--Lurie. We need a few more preliminaries to establish \Cref{sec:representations-g-representations-of-bg-over-bh}.

\begin{definition}
	\label{sec:representations-g}
	Let $d$ be the rank of $E$ over $R$, which we assume to be constant.
	Let
	\[
	\mathcal{K}^{\bullet}:=(0\to \wedge^d_R(E^\vee)\widehat{\otimes}_R B\to \wedge^{d-1}_R(E^\vee)\widehat{\otimes}_R B\to \ldots \to E^\vee\widehat{\otimes}_R B\to B)
	\]
	be the Koszul complex in $\mathcal{D}(\widehat{V^\vee})$ of the finite projective $B$-module $E^\vee\widehat{\otimes}_RB$ with its morphism $E^\vee\widehat{\otimes}_RB\to B$ induced by the natural map $E^\vee\to B$. See \cite[Tags 0621, 062K]{StacksProjectAuthors2017} for the definition of the transition maps of $\mathcal{K}^{\bullet}$. It has a quasi-isomorphism $\mathcal{K}^\bullet\to R\cong B/I$.
	We define
	\[
	\mathcal{E}^\bullet:=\Hom_{R,\mathrm{cont}}(\mathcal{K}^\bullet,R)=(A\to A\widehat{\otimes}_RE \to A\widehat{\otimes}_R\wedge^2_R(E)\to \ldots \to A\widehat{\otimes}_R\wedge^d_R(E)\to 0),
	\]
	which is a resolution $R\to \mathcal{E}^\bullet$ of the trivial $V^\sharp$-representation $R$ by coinduced $A$-comodules.
\end{definition}

\begin{lemma}
	\label{sec:representations-g-3-cohomology-on-bv-sharp}
	Let $M\in \mathcal{D}(BV^\sharp)$. Then there exists a canonical resolution
	\[
	M\to (M\widehat{\otimes}_R A\to M\widehat{\otimes}_R A\widehat{\otimes}_R E\to M\widehat{\otimes}_R A\widehat{\otimes}_R \wedge^2_R(E)\to \ldots)
	\]
	and a canonical quasi-isomorphism
	\[
	R\Gamma(BV^\sharp,M)\to (M\to M\widehat{\otimes}_R E\to M\widehat{\otimes}_R \wedge^2_R(E)\to \ldots ).
	\]
\end{lemma}
We will refer to the natural map $M\to M\widehat{\otimes}_R E$ as the coaction of $E$ on $M$.
\begin{proof}
	The first statement follows by tensoring $R\cong \mathcal{E}^\bullet$ with $M$. For the the second statement, we use the  projection formula for the morphism $\Spf(R)\to BV^\sharp$ to see that $M\widehat{\otimes}_R A\cong |M|\widehat{\otimes}_R A$, where $|M|$ is given the trivial action.
\end{proof}

\begin{remark}
	\label{sec:representations-g-4-module-isomorphic-to-trivial-representation}
	As in \cite[Corollary 3.5.14]{bhatt2022absolute}, \Cref{sec:representations-g-3-cohomology-on-bv-sharp} implies that if $M\in \mathcal{D}(BV^\sharp)$ satisfies $M\cong R$ as $R$-modules and the coaction $M\to M\otimes_R E$ is zero, then $M\cong R$ as $A$-comodules.
\end{remark}

\begin{lemma}
	\label{sec:representations-g-2-category-generated-by-trivial-representation}
	 $\mathcal{D}(BV^\sharp)$ is generated under shifts and colimits by the trivial representation $R$.
\end{lemma}
\begin{proof}
	Using the isomorphism $M\cong M\widehat{\otimes}_R\mathcal{E}^\bullet$ and the projection formula for the faithfully flat morphism $\eta\colon \Spf(R)\to BV^\sharp$, the proof of \cite[Proposition 3.5.15]{bhatt2022absolute} applies here. Hence, it suffices to show that the regular $A$-comodule $A=\eta_\ast(R)$  admits a filtration by sub-comodules with graded pieces isomorphic to $R$ as an $A$-comodule. This is a concrete calculation in $A$ using \Cref{sec:representations-g-4-module-isomorphic-to-trivial-representation}. Alternatively, this statement passes to direct summands in $E$, hence can be reduced to $E=R^d$. Taking tensor products it suffices to assume $d=1$, where it follows from \cite[Proposition 2.9]{analytic_HT}.
\end{proof}

\begin{lemma}
	\label{sec:representations-g-1-phi-fully-faithful}
	The functor $\Phi_{\Spf(R)}\colon \mathcal{D}(BV^\sharp)\to \mathrm{Mod}_{\Symvw(E^\vee)}(\mathcal{D}(\Spf(R)))$
	is fully faithful.
\end{lemma}
\begin{proof}
We want to see that for any $M,N\in \mathcal{D}(BV^\sharp)$, the map
	\[
	R\Hom_{\mathcal{D}(BV^\sharp)}(M,N)\to R\Hom_{\mathcal{D}(V^\vee)}(\Phi_{\Spf(R)}(M),\Phi_{\Spf(R)}(N))
	\]
	is an isomorphism. As $\Phi_{\Spf(R)}$ commutes with filtered colimits (by the explicit formula for $\Phi_{\Spf(R)}$ given above, $\Phi_{\Spf(R)}$ is the identity on underlying $R$-modules) we may assume by \Cref{sec:representations-g-2-category-generated-by-trivial-representation} that $M=R$. By \Cref{sec:representations-g-3-cohomology-on-bv-sharp} the left-hand side is calculated by the complex
	\[
	N\to N\widehat{\otimes}_R E\to N\widehat{\otimes}_R \wedge^2_R(E)\to \ldots,
	\]
	and the right-hand side as well by \eqref{eq:explicit-descr-Phi} and the Koszul resolution of $R$ as $\Symvw(E^\vee)$-module.
\end{proof}

\begin{proof}[Proof of \Cref{sec:representations-g-representations-of-bg-over-bh}]
	We first assume that $H=\{1\}$.
	Fully faithfulness was proven in \Cref{sec:representations-g-1-phi-fully-faithful}. By \Cref{sec:representations-g-2-category-generated-by-trivial-representation} the essential image is generated by $\Phi_{\Spf(R)}(R)\cong R$ and shifts and colimits. In particular, the essential image is contained in the subcategory of (automatically $p$-complete) complexes $M\in \mathrm{Mod}_{\Symvw(E^\vee)}(\mathcal{D}(R))$ such that each $\delta\in E^\vee$ acts locally nilpotently on $H^\ast(M\otimes_{\Z_p}^L\F_p)$. By  fully faithfulness, to show that $M$ lies in the essential image of $\Phi_{\Spf(R)}$, we may pass to a finite Zariski cover of $\Spf(R)$ as this will express $M$ as a finite limit of complexes in the essential image. Hence, we may assume that $E$ is finite free. By \cite[Tag 0A6X]{StacksProjectAuthors2017} we can conclude that the category of such $M$ is generated under colimits by $R\Gamma_Z(B)$. But as in the proof of \Cref{sec:representations-g-rgamma-z-for-b-} we see that
	$
	R\Gamma_Z(B)=\Phi_{\Spf(R)}(A)[-d]$
	(the factor $\mathrm{det}(E)^\vee$ disappears as $E$ was assumed to be trivial).
	
	Now assume that $H\neq 1$. Using that $\Phi_{(-)}$ is natural in $R$ we can conclude by descent along the faithfully flat map $\Spf(R)\to BH$ that  $\Phi_{BH}$
	is fully faithful. The description of the essential image can be checked after pullback along $\Spf(R)$ because the cohomology of $H$ can be calculated as a finite limit, cf.\ \cite[Proposition 2.7]{analytic_HT} and \Cref{sec:representations-g-3-cohomology-on-bv-sharp}. We may thus reduce to $H=\{1\}$.
\end{proof}

\begin{remark}
  \label{sec:representations-g-1-canonical-higgs-field-on-gerbes}
  Let $\mathcal{Z}$ be any $V^\sharp$-gerbe over $\Spf(R)$, and let $R[E]=R\oplus \varepsilon E$ be the trivial square-zero extension of $R$ by $E$. Then there exists a canonical automorphism
  \[
    \rho\colon \Id_{\mathcal{Z}\times_{\Spf(R)}\Spf(R[E])}\to \Id_{\mathcal{Z}\times_{\Spf(R)}\Spf(R[E])}. 
  \]  Namely, the trivial $V^\sharp$-torsor acts as the identity on $\mathcal{Z}$, and on the trivial $V^\sharp$-torsor acts the natural element in $V^\sharp(R[E])=\Hom_{R}(\Gamma_R^\bullet(E)^\wedge_p, R[E])$ coming from the projection
  \[\textstyle\Gamma_R^\bullet(E)^\wedge_p=\widehat{\bigoplus\limits_{n\geq 0}}\Gamma_R^n(E)\to R\oplus {\varepsilon E}
  \]
  where $\varepsilon E\cong \Gamma_R^1(E)$.
  If $\mathcal{F}\in \mathcal{D}(\mathcal{Z})$ is any complex, then $\rho^\ast$ yields an $R[E]$-linear automorphism
  \[
    \mathcal{F}\otimes_{R}R[E]\isomarrow  \mathcal{F}\otimes_{R}R[E],
  \]
  that reduces to $\Id_{\mathcal{F}}$ after base change along $R[E]\to R,\ E\mapsto 0$. Equivalently, this defines a map
  \[
    \Id_{\mathcal{F}}+\varepsilon \theta_{\mathcal{F}}:\mathcal{F}\to \mathcal{F}\otimes_{R}R[E]=\mathcal{F}\oplus \varepsilon\mathcal{F}\otimes_{R}E\quad
 \text{for some} \quad \theta_{\mathcal{F}}\colon \mathcal{F}\to \mathcal{F}\otimes_{R}E \text{ in } \mathcal{D}(\mathcal{Z}).
\]
This construction of $\theta_{\mathcal{F}}$ is analogous to the construction of the Sen operator in \cite[Construction~3.5.4]{bhatt2022absolute}.
  \begin{definition}
  	We call $\theta_{\mathcal{F}}$ the ``canonical Higgs field'' attached to $\mathcal{F}$. 
  \end{definition}
  	In comparison with \cite[Theorem 3.5.8]{bhatt2022absolute} the equivalence \Cref{sec:representations-g-representations-of-bg-over-bh} (if $H=\{1\}$) would ideally be stated as saying that the pullback along any section $\eta\colon \Spf(R)\to \mathcal{Z}$ yields a fully faithful functor $\mathcal{F}\mapsto (\eta^\ast(\mathcal{F}), \eta^\ast(\theta_{\mathcal{F}})$ from $\mathcal{D}(\mathcal{Z})$ to a derived category of Higgs fields, whose essential image is given by the condition of topological nilpotence of the Higgs field. However, this would require setting up enough coherences for the Higgs field condition $\theta_{\mathcal{F}}\wedge \theta_{\mathcal{F}}=0$, which we avoided by using the Fourier transform. Up to these coherence issues, \Cref{sec:representations-g-representations-of-bg-over-bh} realizes this desideratum in the sense that the constructed $\mathcal{S}_p(E^\vee)$-action restricts to the adjoint of the canonical Higgs field.
\end{remark}

\subsection{The isogeny category of perfect complexes}
\label{sec:isog-categ-perf}

We retain setup and notation of 
\S\ref{sec:representations-g_a} with $H=\{1\}$ the trivial group. The aim of this subsection is to describe the $p$-isogeny category
$
\mathcal{P}erf(BV^\sharp)\tf
$
by describing its essential image under the functor $\Phi_{\Spf(R)}\tf$, as follows.
\begin{proposition}
	\label{sec:representations-g-1-isogeny-category-fully-faithful}
	The functor
	\[
	\mathcal{D}(BV^\sharp)\overset{\Phi_{\mathrm{Spf}(R)}} \longrightarrow \mathrm{Mod}_{\Symvw(E^\vee)}(\mathcal{D}(\Spf(R)))\to \mathrm{Mod}_{\Symv(E^\vee)[\frac{1}{p}]}(\mathcal{D}(\Spec(R\tf)))
	\]
	induces a fully faithful functor
	\[
	\Psi_{\mathrm{Spf}(R)} \colon \mathcal{P}erf(BV^\sharp)\tf\to \mathcal{P}erf(\Symv(E^\vee)\tf).
	\]
	The idempotent completion of the essential image of $\Psi_{\mathrm{Spf}(R)}$ consists of those $M\in \mathcal{P}erf(\Symv(E^\vee)\tf)$ that are perfect over $R\tf$ and satisfy the following condition: After base change $R\to S$ to any $p$-complete valuation ring, the action of each $\delta\in E^\vee$ on the cohomology of $M\otimes_{R[\frac{1}{p}]}^L S\tf$ is topologically nilpotent for the natural topology induced by the non-archimedean field $S\tf$.
\end{proposition}
We note that the characterisation of the essential image is a local description in terms of the v-topology. We expect that it should be possible to also give such a description for weaker topologies. However, for more general $S$ it is not immediately clear how to define what ``locally nilpotent'' means for an operator on cohomology as this need not be finitely presented in general.

The following proof, relying on \cite{thomason2007higher}, was suggested to us by Peter Scholze.
\begin{proof}
	We set $T:=\Symv(E^\vee)=\mathrm{Sym}^\bullet_R(E^\vee)$.
Recall that the category of $p$-complete objects in $\mathcal{D}(T)$ is equivalent to the category $\mathrm{Mod}_{\Symvw(E^\vee)}(\mathcal{D}(\Spf(R)))$. 	By \Cref{sec:representations-g-representations-of-bg-over-bh} applied with $H=\{1\}$, the category $\mathcal{P}erf(BV^\sharp)$ is therefore equivalent via $\Phi_{\mathrm{Spf}(R)}$ to the full subcategory
	$
	\mathcal{C}\subseteq \mathcal{D}(T)
	$
	of objects $N$ such that $N$ is perfect over $R$ and such that each $\delta\in E^\vee$ acts locally nilpotently on $H^\ast(N\otimes_{\Z_p}^L\F_p)$. Here we use that any such object is already $p$-complete by $p$-completeness of $R$.
	
	 By \Cref{sec:representations-g-2-general-perfectness-result} below, we see that $\mathcal{C}\subseteq \mathcal{P}erf(T)$.
	We can conclude that the functor 
	\[
	\Psi_{\mathrm{Spf}(R)} \colon \mathcal{C}\tf\to \mathcal{P}erf(T\tf)
	\]
	is fully faithful: Indeed, it suffices to check that $\mathcal{P}erf(T)\subseteq \mathcal{P}erf(T\tf)$ is fully faithful. But
	\[
	R\Hom_{T[\frac{1}{p}]}(N_1\tf,N_2\tf)\cong R\Hom_T(N_1,N_2)\tf
	\]
	for any $N_1,N_2\in \mathcal{P}erf(T)$
	by reduction to the case that $N_1=N_2=T$.

	Clearly, each object $M$ in the essential image of $\Psi_{\mathrm{Spf}(R)}$ satisfies the conditions of the statement, i.e., $M$ is perfect over $R\tf$ and the action of each $\delta\in E^\vee$ is topologically nilpotent on the cohomology of $M$ in the sense of the statement.
	As the category of such $M\in \mathcal{P}erf(T\tf)$ is idempotent complete, these assertions extend to the idempotent completion $\mathcal{C}^\prime$ of $\mathcal{C}\tf$. 
	
	Let now $M\in \mathcal{P}erf(T\tf)$ be as in the statement.
	Set $X:=\Spec(T)$ with its open subscheme $U:=\Spec(T\tf)$. Let $Z\subseteq X$ be the closure of the support of $M$ in $X$. 
	By \cite[Proposition (5.2.2)]{thomason2007higher} we can extend $M\oplus M[1]$, considered as a perfect complex on $U$ with support on $Z\cap U$, to a perfect complex $N$ on $X$ with support in $Z$. The topological nilpotence condition ensures:
	\begin{claim*}
	The subscheme $Z\cap \Spec(T/p)$ is contained (as a set) in the zero section. 
	\end{claim*}
	\begin{proof}[Proof of claim:]
	The subset $Z\cap U$ is pro-constructible in $X$ and hence its closure is given by the set of specializations of elements in $Z$. 
	Each specialization $x\in \Spec(T/p)$ of some point in $Z\cap U$ is witnessed by some map $T\to S^\prime$ with $S^\prime$ a valuation ring. Let $S$ be the $p$-adic completion of $S^\prime$. Then $\Spec(S\tf)$ maps still to $Z$, while the closed point of $\Spec(S)$ maps to $x$.
	
	 To see that $x$ lies in the zero section we may replace $R$ by $S$, and we may assume that $S\tf$ is algebraically closed. 
	Then $E^\vee=R\delta_1\oplus\ldots \oplus R\delta_d$ is trivial and the assumptions that $S\tf$ is algebraically closed and that $M$ is perfect over $R\tf$ guarantee that $Z\cap \Spec(T\tf)$ is (as a set) a finite union of closed subsets of the form  $W=V(\delta_1-s_1,\ldots, \delta_d-s_d)$ with $s_1,\ldots, s_d\in S\tf$. The topological nilpotence condition implies that $s_1,\dots,s_d$ in fact lie in the maximal ideal of $S$. 
	But the mod $p$ fiber  of each $W$ is contained in the zero locus of $X$.
	\end{proof}
	We have thus found a perfect complex $N$ on $X$ with support in $Z$ such that $N\tf\cong M\oplus M[1]$. As the mod $p$ fiber of $Z$ is contained in the zero section, the action of each $\delta\in E^\vee$ on $H^\ast(N\otimes_{\Z_p}^L \F_p)$ is locally nilpotent.
	It remains to see that $N$ is 
	perfect over $R$. By $p$-completeness of $R$ and $N$ it suffices to check this for $N\otimes_{R}^LR/p$, cf.\ \cite[Tag 09AW]{StacksProjectAuthors2017}.
	Now we can apply \Cref{sec:representations-g-2-general-perfectness-result} below.
\end{proof}

\begin{lemma}
	\label{sec:representations-g-2-general-perfectness-result}
	Let $R$ be a ring, $E$ a finite projective $R$-module and $T:=\mathrm{Sym}^\bullet_R(E^\vee)$.
	\begin{enumerate}
		\item If $N\in \mathcal{D}(T)$ is perfect over $R$, then $N$ is perfect over $T$.
		\item If the support of $N\in \mathcal{P}erf(T)$ in $\Spec(T)$ is finite over $\Spec(R)$, then $N\in \mathcal{P}erf(R)$. 
	\end{enumerate}
	
\end{lemma}
\begin{proof}
	Both claims are local on $R$, hence we may assume that $E^\vee$ is trivial, i.e., $T\cong R[\delta_1,\ldots, \delta_d]$ is a polynomical ring.
	For (1), using induction on $d$, we can (by replacing $R$ by $R[\delta_1,\ldots, \delta_{d-1}]$) reduce to the case that $d=1$.
	Let $\delta:=\delta_1$. If $N\in \mathcal{D}(T)$ is arbitrary, then we get a fiber sequence
	\[
	N\otimes_R^LT\xrightarrow{\delta\otimes \Id_T-\Id_N\otimes \delta}N\otimes_R^L T\to N
	\]
	as we check by reduction to $N=T$ via colimits. If $N\in \mathcal{P}erf(R)$, then this shows $N\in \mathcal{P}erf(T)$.
	
	For (2), we embed $\Spec(T)$ into the projective bundle 
	\[f\colon \mathbb{P}(E^\vee\oplus R)\to \Spec(R).\]
	Since the support of $N$ is closed in $\mathbb{P}(E^\vee\oplus R)$ (by finiteness of the support over $\Spec(R)$), we can extend $N$ to a perfect complex $N^\prime$ on $\mathbb{P}(E^\vee\oplus R)$ by extending it by zero on the complement of the support using \cite[Tag 08DP]{StacksProjectAuthors2017}.  Then $Rf_\ast(N^\prime)$ is a perfect complex on $R$ by \cite[Tag 0B91]{StacksProjectAuthors2017}, but $Rf_\ast(N^\prime)$ is also just the complex $N$ seen as an $R$-module. This finishes the proof.
      \end{proof}

      \begin{remark}
        \label{sec:isog-categ-perf-1-remark-for-isogeny-category-in-g-pi-case}
        The proof of \Cref{sec:representations-g-1-isogeny-category-fully-faithful} also applies to the group scheme $\Spf(R)\times_{\Spf(\Z_p)}\Gm^\sharp$, and hence to $G_A\times_{\Spf(\O_K)}\Spf(R)$ for $G_A$ as in \S\ref{sec:p-adic-fields-new-identification-of-galois-action-for-p-adic-fields} and some morphism $\O_K\to R$. Similarly to \cite[Theorem 2.5]{analytic_HT}, we thus obtain from this a fully faithful functor
        \[\mathcal{P}erf(\Spf(\O_K)^\HT\times_{\Spf(\O_K)}\Spf(R))\to \mathcal{P}erf(R[\Theta_\pi])\tf\]
        whose essential image can be characterised as follows: After base change $R\to S$ to any $p$-complete valuation ring, the action of $\Theta_\pi^p-E'(\pi)^{p-1}\Theta_\pi$ on the cohomology of $M\otimes_{R[\frac{1}{p}]}^L S\tf$ is topologically nilpotent for the natural topology induced by the non-archimedean field $S\tf$.
      \end{remark}

      \begin{lemma}
        \label{sec:isog-categ-perf-1-multiplication-by-x-on-v-sharp}
        Let $x\in R$. Then multiplication $x\colon V^\sharp\to V^\sharp$ induces a functor
        \[
          x^\ast\colon \mathcal{P}erf(BV^\sharp)\to \mathcal{P}erf(BV^\sharp),
        \]
        which is fully faithful after inverting $x$, and its essential image are those perfect complexes $(M,\theta_M)$ on $BV^\sharp$, such that $(M,\theta_M)=(M,x\theta_M^\prime)$ for some $(M, \theta_M^\prime)\in \mathcal{P}erf(BV^\sharp)$. 
      \end{lemma}
      \begin{proof}
        Fully faithfulness follows from \Cref{sec:representations-g-3-cohomology-on-bv-sharp}. The essential image is clear.
      \end{proof}

\subsection{Application to the smoothoid case}
\label{sec:application-ht-stack-smoothoid-case}
Let $X$ be a qcqs smoothoid $p$-adic formal scheme over $\Z_p^\cycl$ with sheaf of $p$-completed differentials $\Omega^1_{X}:=\Omega^1_{X|\Z_p^\cycl}$ from \Cref{def:absolute-diff-of-smoothoid}. Let $(A_0,I_0)$ be the perfect prism associated with $\Z_p^\cycl$, i.e., $A_0/I_0\cong \Z_p^\cycl$. This describes the Breuil--Kisin twist $\{1\}$ on $\Z_p^\cycl$-modules as $I_0/I_0^2\otimes_{\Z_p^\cycl}(-)$.
As in \S\ref{sec:smoothoid-case}, $\mathcal{T}_{X}^\sharp\{1\}:=\mathcal{T}_{X|\Z_p^\cycl}^\sharp\{1\}$ is the $p$-completed PD-envelope of the zero section of the tangent bundle $\mathcal{T}_{X}\{1\}$ of $X$, i.e., locally $X=\Spf(R)$ and
  \[
    \mathcal{T}^\sharp_{X}\{1\}:=\Spf(\Gamma_R^\bullet(\Omega^1_{R}\{-1\})^\wedge_p),\quad \mathcal{T}_{X}\{1\}:=\Spf(\mathcal{S}_p(\Omega^1_{R}\{-1\})),
  \]
  where $\Gamma_R^\bullet $ denotes the PD-algebra, and $\mathcal{S}_p(-)$ the $p$-completed symmetric algebra. The main player of this subsection is the relative classifying stack 
 $B\mathcal{T}^\sharp_{X}\{1\}$.
   We use it to describe complexes on the Hodge--Tate stack in some cases and describe the isogeny category of the category of perfect complexes on it in terms of the generic fiber $\X$ of $X$. Indeed, we can now prove \Cref{sec:smoothoid-case-1-complexes-on-x-ht-smoothoid-case-introduction}:

  \begin{theorem}
    \label{sec:smoothoid-case-1-complexes-on-ht-in-split-case}
  Any section $X\to X^\HT$ of the projection $X^\HT\to X$ induces an isomorphism
  \[
    X^\HT\cong B\mathcal{T}_{X}^\sharp\{1\}
  \]
  of $\mathcal{T}_{X}^\sharp\{1\}$-gerbes.
  In particular, it induces a fully faithful functor
		\[
		\mathcal{D}(X^\HT)\hookrightarrow \mathcal{D}(\mathcal{T}_{X}^\vee\{-1\}).
		\]
		Its essential image is given by complexes $\mathcal{M}$ on $\mathcal{T}_{X}^\vee\{-1\}$ such that locally on any affine open $U:=\Spf(R)\subseteq X$ each $D\in T_{R}\{1\}$ acts nilpotently on $H^\ast(R\Gamma(U,\mathcal{M})\otimes_{R}^LR/p)$.
              \end{theorem}
              This generalizes \cite[Corollary 6.6]{bhatt2022prismatization}, which gives the statement for vector bundles.
              \begin{proof}
                By the proof of \cite[Proposition 5.12]{bhatt2022prismatization}, $X^\HT\to X$ is a gerbe for $\mathcal{T}_{X}^\sharp\{1\}$ (note that $X^\HT$ is equivalently the Hodge--Tate stack of the relative prismatization of $X$ over $A_0$ as $\Z_p^\cycl$ is perfectoid). So any splitting $X\to X^\HT$ induces $X^\HT\cong B\mathcal{T}_{X}^\sharp\{1\}$. The second part follows from \Cref{sec:representations-g-representations-of-bg-over-bh} as the functor in \Cref{sec:representations-g-definition-of-functor-phi} is natural in $\Spf(R)$, hence can be glued over affine pieces of $X$.
              \end{proof}
             \begin{remark}
             	The
              convergence condition for Higgs fields appearing in the description of the essential image of \Cref{sec:smoothoid-case-1-complexes-on-ht-in-split-case} also appears in the work of Tsuji in the context of the Higgs site \cite[\S IV.3.6, Thm~IV.3.4.16]{abbes2016p} as well as the work of Tian  \cite[Def.\ 5.9, Thm~5.12]{tian2021finiteness}. In our context, its appearance is explained geometrically by the isomorphism $X^\HT\cong B\mathcal{T}_{X}^\sharp\{1\}$.
              
              Via reduction mod $p$, there is also a relation to the nilpotent Higgs bundles appearing in the mod $p$ Simpson correspondence of Ogus-Vologodsky \cite{ogus_vologodsky}  over $\F_p$.  In particular, they also consider the mod $p$ version of the PD-cotangent sheaf $\mathcal{T}_{X}^\sharp\{1\}$ from \Cref{sec:smoothoid-case-1-complexes-on-ht-in-split-case}, see \cite[\S2.3]{ogus_vologodsky}. 
            \end{remark}

                       Recall that the Hodge--Tate stack does not split in general, but it does when $X$ admits a smooth lift to $A_0$ equipped with a $\delta$-structure, e.g. when $X$ is affine.

              \begin{corollary}
	\label{computation-cohomology-koszul-resolution}
	Let $X=\mathrm{Spf}(R)$ be a smooth $p$-adic formal scheme over $\Z_p^\cycl$. Let $\mathcal{M} \in \mathcal{D}(X^{\mathrm{HT}})$, seen as an object $M\in \widehat{D(R)}$ with an $R$-linear action of $S:= \mathcal{S}_p(T_R \{1\})$ via \Cref{sec:smoothoid-case-1-complexes-on-ht-in-split-case}. Then
	$$
	R\Gamma(X^{\mathrm{HT}}, \mathcal{M}) \cong R\mathrm{Hom}_S(R, M) \cong R\mathrm{Hom}_S( ...\to \wedge^2 (T_{R} \{1\}) \otimes_R S \to T_{R} \{1\} \otimes_R S \to S, M).
	$$
	In particular, this cohomology is computed as cohomology of the Dolbeault complex
	$$
	\mathrm{Dol}(\mathcal{M}):= M \overset{\theta_M} \longrightarrow M \otimes_R \Omega_{R}^1 \{-1\} \overset{\theta_M} \longrightarrow M \otimes_R \Omega_{R}^2\{-2\} \to \dots
	$$ 
\end{corollary}
\begin{proof}
	The first isomorphism follows from fully faithfulness in \Cref{sec:smoothoid-case-1-complexes-on-ht-in-split-case} applied to morphisms from the unit object on $X^{\mathrm{HT}}$, the second from the Koszul resolution of $R$ as an $S$-module.
\end{proof}

\begin{remark}
\label{remark:recovering-scholze-computation-pushforward}
Let $X$ be a smooth $p$-adic formal scheme over the ring of integers $\mathcal{O}_L$ of a perfectoid field extension $L$ of $\Q_p$ containing $\Q_p^\cycl$, and $\X$ its generic fiber. As a consequence of \Cref{computation-cohomology-koszul-resolution} and \Cref{sec:introduction-2-statement-main-theorem}, we recover Scholze's computation \cite[Proposition 3.23]{scholze2013perfectoid} that for $j\geq 0$,
$$
\Omega_{\X}^j(-j) \cong R^j\nu_\ast \mathcal{O}_{\X_{\mathrm{pro\acute{e}t}}}
$$
where $\nu_\ast : {\X_{\mathrm{pro\acute{e}t}}} \to {\X_{\mathrm{\acute{e}t}}}$ is the natural morphism of sites. The proof below is not very different from the original but explains perhaps more clearly how differentials enter the picture; it is quite similar to Bhatt--Lurie's new proof of Hodge--Tate comparison for prismatic cohomology using $X^\HT$.

We first note that we can replace the pro-\'etale site by the v-site and the Tate twist by a Breuil-Kisin twist. We will still denote by $\nu$ the natural morphism of topoi $\widetilde{\X_v} \to \widetilde{\X_{\mathrm{\acute{e}t}}}$. 
Let $\eta:\X_\et\to X_{\mathrm{Zar}}$ be the natural morphism to the Zariski-site of $X$.

Now by construction
$
\alpha_X^\ast \mathcal{O}_{X^\HT} = \mathcal{O}_{\X_v}
$.
Hence  \Cref{sec:introduction-2-statement-main-theorem} yields by naturality of $\alpha^\ast_X$ a natural isomorphism
\[  R^j(\eta_\ast\circ \nu_\ast) \mathcal{O}_{\X_v}\isomarrow  R^j\pi_\ast \mathcal{O}_{X^\HT}\tf\]
on $X$
where $\pi:X^\HT\to X$ is the structure map.
Since we know a priori from \cite[Lemmas~4.5~and~5.5]{scholze2013p} that $R^j\nu_\ast \mathcal{O}_{\X_v}$ is a finite locally free sheaf, pullback along $\eta$ induces an isomorphism
\[  R^j\nu_\ast \mathcal{O}_{\X_{v}}\to \eta^\ast R^j\pi_\ast \mathcal{O}_{X^\HT}. \]

On the other hand, the fact that $\pi$ is a  $\mathcal{T}_{X}^\sharp\{1\}$-gerbe implies that
\[R^j\pi_\ast \mathcal{O}_{X^\HT}=R^j\pi'_\ast \mathcal O_{B\mathcal{T}_{X}^\sharp\{1\}}\]
for the trivial gerbe $\pi': B\mathcal{T}_{X}^\sharp\{1\}\to X$,
see \cite[($\ast$) in the proof of Corollary 2.7.2.1]{bhatt2022F-gauges}.  When $X$ is affine, \Cref{computation-cohomology-koszul-resolution} then yields an isomorphism 
\[ \eta^\ast R^j\pi'_\ast \mathcal O_{B\mathcal{T}_{X}^\sharp\{1\}}=\Omega_{\X}^j(-j),\]
which is natural in $X$ due to naturality in \Cref{sec:representations-g-representations-of-bg-over-bh}. By gluing, we thus obtain the last isomorphism also for general $X$. All in all, this yields the desired isomorphism $R^j\nu_\ast \mathcal{O}_{\X_v}=\Omega_{\X}^j(-j)$.

Note that for affine $X$ with a lift to $A_0$, we additionally get the splitting of the complex $R\nu_\ast \mathcal{O}_{\X_v}$.  
\end{remark}

In order to describe the isogeny category $\mathcal{P}erf(X^\HT)\tf$, we now define the notions of Higgs perfect complexes and Hitchin-smallness. Recall our notation $\omega:=(\zeta_p-1)^{-1}$.

 \begin{definition}
  \label{sec:appl-Hodge--Tate-definition-higgs-perfect-complex}
	Let $\mathcal{Y}$ be a smoothoid analytic adic space over $\Q_p^\cycl$.\footnote{I.e., $\mathcal{Y}$ is locally for its analytic topology smooth over some perfectoid space over $\Q_p^\cycl$, cf.\ \cite[Definition 2.2.]{heuer-sheafified-paCS}. Similarly to \Cref{sec:fully-faithf-smooth-differentials-are-smoothoid} each smoothoid space has  finite locally free sheaf of differentials $\Omega^1_{\mathcal{Y}}$, cf.\ \cite[Definition 2.10]{heuer-sheafified-paCS}.}
	\begin{enumerate}
		\item  Recall that $T_{\mathcal{Y}}$ is the $\mathcal{O}_{\mathcal{Y}}$-linear dual of $\Omega_{\mathcal{Y}}^1=\Omega_{\mathcal{Y}|\Q_p^\cycl}^1$. 
		A \textit{Higgs perfect complex} on $\mathcal{Y}$ is a complex on the ringed site $(\mathcal{Y},\mathrm{Sym}^\bullet_{\O_{\mathcal{Y}}}(T_{\mathcal{Y}}\{1\}))$ that is already perfect over  $(\mathcal{Y},\O_{\mathcal{Y}})$.
		\item Assume that $\mathcal{Y}=Y^\rig$ is the rigid generic fiber of a smoothoid $p$-adic formal scheme $Y$ over $\Z_p^\cycl$. Then a Higgs perfect complex $M$ on $\mathcal{Y}$ is called \textit{Hitchin--small} if for any point $\xi:\Spa(C,C^+)\to \mathcal{Y}$ valued in a non-archimedean affinoid field $(C,C^+)$, the action of each
		$\delta\in \omega \cdot T_{Y}\{1\}_\xi \subseteq T_{\mathcal{Y}}\{1\}_\xi$
		 on each cohomology group of $M\otimes_{\O_{\mathcal{Y}}}^L C$ is topologically nilpotent.
		 \item More generally, for $z\in \O(\mathcal Y)^\times$, we call a Higgs perfect complex $z$-Hitchin-small if instead $\delta$ is topologically nilpotent for any $\delta\in \frac{\omega}{z}\cdot T_{Y}\{1\}_\xi$.
	\end{enumerate}
	      We let $\mathcal{H}ig_{\mathcal{Y}}$ be the $\infty$-category of Higgs perfect complexes on $\mathcal{Y}$, and $\mathcal{H}ig_{\mathcal{Y}}^{\Hsm}$ its full subcategory of Hitchin-small objects. Similarly, let $\mathcal{H}ig_{\mathcal{Y}}^{z\text{-}\Hsm}$ be the full subcategory of $z$-Hitchin-small objects.
      \end{definition}
      
  \begin{example}
  	\label{rk:hitchin-small-vector-bundles}
  	A Higgs perfect complex for which the underlying $\mathcal{O}_{\mathcal{Y}}$-module is a vector bundle is the same thing as a vector bundle $M$ on $\mathcal{Y}$ together with a Higgs field
  	$
  	\theta_M : M \to M \otimes_{\mathcal{O}_{\mathcal{Y}}} \Omega_{\mathcal{Y}}^1\{-1\}.
  	$
  	It is Hitchin-small if and only if $\delta\circ \theta_M$ is topologically nilpotent for any $\delta\in \omega.T_{Y|\Z_p^\cycl}\{1\}_\xi$. For example, this is the case if $M$ admits a formal model $\mathfrak M\to \mathfrak M\otimes \Omega_{X}^1\{-1\}$ that is $\equiv 0 \bmod (\zeta_p-1)$.
  \end{example}
  \begin{remark}
  	The normalisation by $(\zeta_p-1)$ in the definition of Hitchin smallness is chosen to make it easiest to state the global correspondence in  \S\ref{sec:globalisation}. In particular, it gives the cleanest comparison to Faltings' notion of smallness.
  	We also note that the additional factor of $(\zeta_p-1)$ is \textit{not} explained by the difference between Breuil-Kisin and Tate twists: Indeed, $\omega\cdot T_{Y}\{1\}=\omega^2 \cdot T_{Y}(1)$.
  	 From the perspective of this section, $\omega$-Hitchin smallness would be the more intrinsic notion.
  \end{remark}

  \begin{remark}
    \label{sec:appl-smooth-case-notion-of-hitchin-smallness-is-intrinsic}
    A priori, the notion of Hitchin-smallness in \Cref{sec:appl-Hodge--Tate-definition-higgs-perfect-complex} depends on the formal model $Y$ because the lattices $\Omega^1_{Y/\Z_p^\cycl}\{-1\}_{\xi}\subseteq \Omega^1_{\mathcal{Y}/\Q_p^\cycl}\{-1\}_\xi$ do. However, these lattices are in fact intrinsic to $\mathcal{Y}$ by \cite[\S8]{Bhatt2018}: It suffices to check this when $Y=\Spf(R)$ is affine and small. Then 
    \[
      \Omega_{R/\Z_p^\cycl}^1\{-1\}\cong H^1(L\eta_{(\zeta_p-1)}(R\Gamma(\mathcal{Y}_{\textrm{pro\'et}},\mathcal{O}^+)))
    \]by \cite[Theorem 8.7]{Bhatt2018}, and this is
    is intrinsic to $\mathcal{Y}$, defined as a submodule of $\Omega^1_{R[1/p]/\Q_p^\cycl}\{-1\}\cong H^1(\mathcal{Y}_{\textrm{pro\'et}},\mathcal{O}^+)$ via the natural map $L\eta_{(\zeta_p-1)}R\Gamma(\mathcal{Y}_{\textrm{pro\'et}},\mathcal{O}^+)\to R\Gamma(\mathcal{Y}_{\textrm{pro\'et}},\mathcal{O}^+)$.  
  \end{remark}

By \Cref{sec:representations-g-representations-of-bg-over-bh}, complexes on $B\mathcal{T}^\sharp_{X}\{1\}$ generalize the natural integral version of $\omega$-Hitchin-small Higgs perfect complexes.
Note that we have a natural, fully faithful functor
\[
  \Xi_{X}\tf\colon \mathcal{P}erf(B\mathcal{T}^\sharp_{X}\{1\})\tf\to \mathcal{H}ig_{\mathcal{X}}^{\omega\text{-}\Hsm}.
\]
This functor is not far from being an equivalence:
\begin{proposition}
  \label{sec:appl-Hodge--Tate-hig-sen-via-integral-stack-smoothoid-case}
  \begin{enumerate}
  	\item Suppose that $X$ is affine. Then $\Xi_{X}\tf$ is an equivalence of categories, up to passing to the idempotent completion $\mathcal{P}erf(B\mathcal{T}^\sharp_{X}\{1\})\tf^{\mathrm{idem}}$ of $\mathcal{P}erf(B\mathcal{T}^\sharp_{X}\{1\})\tf$.
  	\item In general, $\Xi_{X}\tf$ induces an equivalence onto $\mathcal{H}ig_{\mathcal{X}}^{\omega\text{-}\Hsm}$ from the global sections of the stackification of the functor on $X_{\mathrm{Zar}}$ defined by $U\mapsto {\mathcal{P}erf(B\mathcal{T}^\sharp_{U}\{1\})\tf}^{\mathrm{idem}}$.
  \end{enumerate}

\end{proposition}
\begin{proof}
	Part (1) follows from \Cref{sec:representations-g-1-isogeny-category-fully-faithful}. For (2), by descent of perfect complexes on $\X$, the functor $U\mapsto \mathcal{H}ig_{\mathcal{X}}^{\omega-\Hsm}$ on $X_{\mathrm{Zar}}$ is a stack of $\infty$-categories.
	Now $\Xi_{(-)}\tf$ constitutes a natural transformation of functors. By (1) it follows that $\Xi_{X}\tf$  is an equivalence after stackification.
\end{proof}

We can now put everything together and summarise our main results in the smoothoid case so far by way of a derived generalisation of Faltings' ``local $p$-adic Simpson functor'' in our setting:
\begin{theorem}[derived local $p$-adic Simpson functor]\label{t:local-p-adic-Simpson-functor-geometric}
	Let $X$ be a smoothoid formal scheme over $\Z_p^\cycl$. Assume that $X^\HT$ admits a splitting $s:X\to X^\HT$.\footnote{If $X=\Spf(R)$ is affine, such a choice is induced by a prismatic lift $(A,I)$ with $A/I=R$, or by a toric chart.} Then for the generic fibre $\X$ of $X$, there is a fully faithful functor $\LS_s$ from the category of $\omega$-Hitchin-small Higgs perfect complexes on $\X$ into the category of perfect complexes on $\X_v$ which is natural in $s$ and fits into a diagram
	\[\begin{tikzcd}[column sep=0.2cm]
		{\mathcal{H}ig_{\mathcal{X}}^{\omega\text{-}\Hsm}} \arrow[rr,hook,"\LS_s"] &                          & {\calPerf(\X_v)} \\
		& {\calPerf(X^\HT)\tf} \arrow[lu,"{\Xi_X[\frac{1}{p}]}"{yshift=2pt,xshift=-3pt},xshift=-0.3cm,hook'] \arrow[ru,"	\alpha_X^\ast"'{yshift=2pt,xshift=3pt},hook] &   
	\end{tikzcd}\]
    \end{theorem}
\begin{proof}
	If $X$ is affine, then by combining \Cref{sec:smoothoid-case-1-complexes-on-ht-in-split-case} and \Cref{sec:appl-Hodge--Tate-hig-sen-via-integral-stack-smoothoid-case}, we see that $s$ induces a fully faithful functor $\Xi_X[\tfrac{1}{p}]$. After replacing $\calPerf(X^\HT)\tf$ with the stackification of its idempotent completion, its image can be identified with $\mathcal Hig_{\X}^{\omega\text{-}\Hsm}$ according to \Cref{sec:appl-Hodge--Tate-definition-higgs-perfect-complex}.
	
		The fully faithful functor $\alpha_X^\ast$ exists by  \Cref{sec:tori-over-perfectoid-1-main-theorem-with-perfectoid-base}. It can also be extended to the stackification of the idempotent completion of $\calPerf(X^\HT)\tf$ since $\calPerf(\X_v)$ is already idempotent complete.
	We can thus defined $\LS_s$ by setting $\LS_s:=\alpha_X^\ast\circ\Xi_X[\tfrac{1}{p}]^{-1}$.
	
	As this is clearly natural in $(X,s)$, we obtain the general case by gluing.
\end{proof}
We note that this is a generalisation of Faltings' functor (in the case of good reduction) even in the more classical case of coherent sheaves on smooth rigid spaces:
\begin{corollary}\label{c:LS-coherent-modules}
Assume that $X$ is an affine smooth formal scheme with a prismatic lift $(A,I)$ (e.g.\ induced by a toric chart), inducing a splitting $s$ of $X^\HT$, then we obtain a fully faithful functor
\[\mathrm{LS}_s^{\rm coh} : \left\{ \begin{array}{@{}c@{}l}\omega\text{-Hitchin-small coherent}\\\text{Higgs modules on $\X$}\end{array}\right\}\hookrightarrow \left\{ \begin{array}{@{}c@{}l}\text{perfect } \O_{\X_v}\text{-modules}\\\text{on $(\X_v,\O_{\X_v})$}\end{array}\right\}\]
where $\X$ is the rigid generic fibre.
Moreover, for any $\omega$-Hitchin-small coherent Higgs bundle $\mathcal M$,
\[ \mathrm{Dol}(\mathcal M)=R\Gamma(\X_v,\mathrm{LS}_{s}^{\rm coh}(\mathcal M)).\]
\end{corollary}
Here by a perfect $\O_{\X_v}$-module we mean a perfect complex concentrated in degree $0$.
\begin{proof}
	For any affine smooth rigid space $\X=\mathrm{Spa}(A)$, the ring $A$ is regular, hence perfect complexes concentrated in degree $0$ are equivalent to coherent $\O_{\X}$-modules. The cohomological comparison is immediate from \Cref{computation-cohomology-koszul-resolution} by comparing $R\Hom(\O,-)$ on both sides.
\end{proof}

\subsection{Application to the arithmetic case}
\label{sec:application-ht-stack-arithmetic-case}
Let $K$ be a $p$-adic field. Let $X$ be a smooth formal scheme over $X_0:=\Spf(\O_K)$, not necessarily affine.
By \Cref{sec:crit-fully-faithf-relative-ht-map}, the relative Hodge--Tate map
\[
\pi_{X/X_0}\colon X^\HT\to X\times_{X_0} X_0^\HT
\]
is a gerbe under the affine flat group scheme $\mathcal{T}^\sharp_{X|\O_K}\{1\}$ over $X\times_{X_0}X_0^\HT$. Here, $\{1\}$ refers to twisting by the pullback to $X\times_{X_0}X_0^\HT$ of the canonical line bundle $\mathcal{O}_{\Spf(\Z_p)^\HT}\{1\}$, cf.\ \cite[Example 3.5.2]{bhatt2022absolute}.
We denote by $\mathcal{T}^\vee_{X|\O_K}\{-1\}$ the dual geometric vector bundle to $\mathcal{T}_{X|\O_K}\{1\}$ over $X\times_{X_0}X_0^\HT$.

Assume first that $X=\Spf(R)$ and that $R=A/I$ for some prism $(A,I)$. Let
\[
f:=\overline{\rho_A}\colon X=\Spf(A/I)\to X^\HT
\quad \text{and}\quad
g\colon X\to X^\HT\xrightarrow{\pi_{X/X_0}} X\times_{X_0}X_0^\HT
\]
be the induced morphisms. Set $G:=\underline{\mathrm{Aut}}(f)$, $H:=\underline{\mathrm{Aut}}(g)$ as the induced affine $p$-completely flat group schemes over $X$ of automorphisms of $f,g$, and let $G\to H$ be the natural morphism. Then
\[
BG\cong X^\HT,\quad X\times_{X_0} X_0^\HT\cong BH
\]
and there exists a natural exact sequence
\[
1\to \mathcal{T}_{X|\O_K}^\sharp\{1\}\to G\to H\to 1
\]
of group schemes over $X$ (here $(-)\{1\}=I/I^2\otimes_R(-)$ by our choice of prismatic lift of $R$). The conjugation action of $G$ on $\mathcal{T}_{X|\O_K}^\sharp\{1\}$ equals the natural action of $H$ via the Breuil--Kisin twist. 

We can draw the following conclusion.

\begin{theorem}
	\label{sec:appl-Hodge--Tate-structure-of-x-ht-for-some-prismatic-lift}
	Assume $X=\Spf(R)$ with $R=A/I$ for some prism $(A,I)$. This choice defines an isomorphism $X^\HT=BG$ which induces a fully faithful functor
	\[
	\mathcal{D}(X^\HT)\to \mathrm{Mod}_{\mathcal{S}_p(\Omega^{1,\vee}_{X|\O_K}\{1\})}(\mathcal{D}(X\times_{X_0}X_0^\HT))
	\]
	that is natural in $(A,I)$.
	Its essential image is given by objects $M$ such that each $\delta\in \Omega^{1,\vee}_{X|\O_K}\{1\}$ acts locally nilpotently on $H^\ast(M\otimes_{\Z_p}^L \F_p)$. Here we regard $\Omega^{1,\vee}_{X|\O_K}\{1\}\in \mathcal{D}(X\times_{X_0}X_0^\HT)$.
\end{theorem}

\begin{proof}
	Setting $E=\Omega^1$, 
	this follows from \Cref{sec:representations-g-representations-of-bg-over-bh} by our preceding discussion.
\end{proof}
\begin{proof}[Proof of \Cref{sec:introduction-2-description-of-ht-stack}]
This follows from \Cref{sec:appl-Hodge--Tate-structure-of-x-ht-for-some-prismatic-lift} in conjunction with \cite[Theorem~2.5]{analytic_HT}, which says that the essential image of 
\[ \mathcal{D}(X_0^\HT)\hookrightarrow \mathcal{D}(\O_K[\Theta_\pi])\]
is given by objects $M$ for which the action of $\Theta^p_\pi-E^\prime(\pi)^{p-1}\Theta_\pi$ on $H^\ast(M\otimes_{\Z_p}^L \F_p)$ is locally nilpotent.
To see that this implies the analogous description of $\mathcal{D}(X\times_{X_0}X_0^\HT)$, we note that $X\to X_0$ is affine and therefore $\mathcal{D}(X\times_{X_0}X_0^\HT)$ is naturally identified with $R$-module objects in  $\mathcal{D}(X_0^\HT)$.
\end{proof}

The choice of a uniformizer $\pi \in \mathcal{O}_K$ gives rise to an isomorphism $X_0^\HT\cong BG_\pi$. We let $G_\pi$ act on $\mathcal{T}^\sharp_{X|\O_K}\{1\}$ via multiplication by $E^\prime(\pi)$, cf.\ \cite[Theorem 2.5]{analytic_HT}. Then
\[
 \mathcal{Z}_X:= B\mathcal{T}^\sharp_{X|\O_K}\{1\}
\]
is a gerbe over $X\times_{X_0}BG_\pi$. Our next goal is to describe $\mathcal Perf(\mathcal{Z}_X)\tf$, analogously as in \S\ref{sec:application-ht-stack-smoothoid-case}

 To do so, we first define the notion of Higgs--Sen perfect complexes. Let $\mathcal{Y}$ be a smooth rigid space over $K$ (viewed as an adic space). We define the sheaf of (non-commutative) rings on $\mathcal{Y}$
    \[
      \mathcal{HS}_{\mathcal{Y}}:=\mathrm{Sym}^\bullet_{\O_{\mathcal{Y}}}(\Omega^{1,\vee}_{\mathcal{Y}|K}\{1\})[\Theta] \quad \text{such that } [\Theta,\delta]=-\delta \text{ and } [\delta, \delta^\prime]=0 \text{ for any }\delta, \delta^\prime\in \Omega^{1,\vee}_{\mathcal{Y}|K}\{1\}.\] 
    \begin{definition}
      \label{sec:appl-Hodge--Tate-definition-higgs-sen-complexes}
      Let $\mathcal{Y}\to \Spa(K,\O_K)$ be a smooth rigid space over $K$.
      \begin{enumerate}
      \item A Higgs--Sen perfect complex is a complex in the derived category of the ringed site $(\mathcal{Y}, \mathcal{HS}_{\mathcal{Y}})$, which is already perfect over $(\mathcal{Y},\O_{\mathcal{Y}})$.
      \item A Higgs--Sen perfect complex $(M, \theta_M\colon M\to M\otimes_{\O_\mathcal{Y}}\Omega^1_{\mathcal{Y}|K}\{-1\}, \vartheta_M \colon M\to M)$ with $\vartheta_M$ denoting the action of $\Theta$ is called Hitchin-small if for any point $y\colon \Spa(C,C^+)\to\mathcal Y$ valued in a non-archimedean field, the eigenvalues of $\vartheta_M$ are all in $\Z+ \delta_{\mathcal{O}_K/\Z_p}^{-1}\cdot \mathfrak{m}_{\overline{K}}$ where $\delta_{\mathcal{O}_K/\Z_p}^{-1}$ is the inverse different of $\O_K$.
      \end{enumerate}
      We let $\mathcal{H}igSen_\mathcal{Y}$ be the $\infty$-category of Higgs--Sen perfect complexes on $\mathcal{Y}$, and $\mathcal{H}igSen_{\mathcal{Y}}^{\Hsm}$ its full subcategory of Hitchin-small Higgs--Sen perfect complexes.
    \end{definition}
 
Let $e:=E^\prime(\pi)$, then $\delta_{\mathcal{O}_K/\Z_p}=(e)$. For $\Theta_\pi:=e\vartheta_M$, the condition in (2) is equivalent to saying that the action of  $\Theta_\pi^p-e^{p-1}\Theta_\pi$ on $y^\ast M$ is topologically nilpotent  (cf.\ \cite[\S 2.2]{analytic_HT}): This follows from
\[\Theta^p_\pi-e^{p-1}\Theta_\pi=e^p(\vartheta_M^p-\vartheta_M).\]

\begin{remark}
	\label{affinoid-case-andreychev}
	If $\mathcal{Y}=\Spa(B,B^+)$ is affinoid, then by \cite[Theorem 1.4]{andreychev2021pseudocoherent}, $\mathcal Perf(\mathcal{Y})$ is equivalent to $\mathcal Perf({B})$. From here one can deduce that a Higgs--Sen perfect complex on $\mathcal{Y}=\Spa(B,B^+)$ is a perfect complex over $B$ together with an action of the global sections of $\mathcal{HS}_{\mathcal{Y}}$.
\end{remark}
	\begin{remark}
			The name is motivated as follows: When $M$ vector bundle, by \Cref{sec:appl-Hodge--Tate-properties-of-t-modules} below, the Hitchin fibration for Higgs--Sen modules should be defined by sending a triple $(M,\theta_M,\vartheta_M)$ to the characteristic polynomial of $\vartheta_M$. Then a Higgs--Sen vector bundle is Hitchin-small if and only if its image under the Hitchin fibration lands in an open subspace depending only on $\delta_{\mathcal{O}_K/\Z_p}^{-1}$.
	\end{remark}

    First, we analyze the local situation. Fix a uniformizer $\pi$ as before. Assume that $X=\Spf(R)$ is affine of pure dimension $n$. We define a non-commutative ring
    \[
      T:=R[\Theta_\pi, \theta_1, \ldots, \theta_n] \quad\text{with}\quad
      [\Theta_\pi,\theta_i]=-E^\prime(\pi)\theta_i,\quad [\theta_i,\theta_j]=0 \quad 
      \text{for }i,j=1,\ldots, n.
    \]

    \begin{proposition}
      \label{sec:appl-Hodge--Tate-properties-of-t-modules}
      Let $M\in \mathcal{D}(T\tf)$ be such that $M$ is perfect over $R\tf$.
      \begin{enumerate}
        \item Each canonical truncation of $M$ is perfect over $R\tf$.
      \item Each $\theta_i, i=1,\ldots, n$ acts nilpotently on each cohomology object of $M$.
      \item The functor $R\tf\otimes_R(-)$ induces a functor
        \[
         \Xi \colon \{ N\in \mathcal{D}(T),\ \textrm{ perfect over }R \} \to \{ M\in \mathcal{D}(T\tf),\ \textrm{ perfect over }R\tf \} 
       \]
       that is fully faithful after passing to the isogeny category on the left, i.e., inverting $p$.
       \item If $E^\prime(\pi)^{-1}\Theta_\pi$ has all eigenvalues  in $\Z+ \delta_{\mathcal{O}_K/\Z_p}^{-1}\cdot \mathfrak{m}_{\overline{K}}$ on each cohomology object on $M$ (in the sense of \Cref{sec:appl-Hodge--Tate-definition-higgs-sen-complexes}), then $M$ lies in the essential image of $\Xi\tf$ after passing to the idempotent completion on the source.
      \end{enumerate}
      
    \end{proposition}
    The proof of nilpotence follows an argument of Min--Wang \cite[Remark~4.2]{MinWang22}.
    \begin{proof}
      The first part is implied by regularity of $R$, which itself follows from smoothness of the formal scheme $X$ over $\O_K$.
      For (2), we may assume by (1) that $M$ is concentrated in a single degree. Then the claim is true with $R\tf$ replaced by any noetherian $\Q$-algebra. Namely, it suffices to argue for reduced noetherian $\Q$-algebras, and the noetherian induction reduces to the case of generic points, and then further to some algebraically closed field $F$ over $\Q$. In this case, we get a morphism
  \[
    \Phi\colon \langle \theta_1,\dots,\theta_n,\Theta_\pi\rangle_F/([\Theta_\pi,\theta_i]=-E^\prime(\pi)\theta_i,[\theta_i,\theta_j]=0, i,j=1,\ldots, n)\to \mathrm{End}(M),
  \]
  of $F$-Lie algebras. Let $L$ be the image of $\Phi$. Since $E^\prime(\pi)$ is invertible in $F$, the commutator $[L,L]$ is generated by the images of $\theta_1,\ldots, \theta_n$. Note that $L$ is solvable since $[L,L]$ is abelian as we just saw. This implies that $[L,L]$ acts nilpotently on any finite dimensional representation, cf.\ \cite[\S 4.1, Cor A]{humphreys2012introduction}. Since $M$ is finite dimensional, this proves (2).
  
  We now prove (3). The algebra $T$ is in fact a Hopf algebra over $R$ (it is the universal enveloping algebra of the $R$-version of the Lie algebra appearing in the proof of (2)), and each $N\in \mathcal{D}(T)$, which is perfect over $R$, is dualizable for the induced symmetric monoidal structure
  \[
    -\otimes_R^L-\colon \mathcal{D}(T)\times \mathcal{D}(T)\to \mathcal{D}(T).
  \]
  Now, $R\Hom_T(R,N)$ is calculated by the fiber of the lift of multiplication by $\Theta_\pi$ on $T/(\theta_1,\dots,\theta_n)$ to the (shifted by $n$) Koszul complex for the $\theta_1,\ldots, \theta_n$ on $N$. As this holds similarly also for
  $
    R\Hom_{T[\frac{1}{p}]}(R\tf,N\tf)
  $
  we can conclude fully faithfulness of $\Psi$ on the isogeny categories. (We caution that the lifts of multiplication by $\Theta_\pi$ to the Koszul complexes aren't termwise multiplication by $\Theta_\pi$, but rather more complicated expressions, due to the non-commutative nature of $T$. Whatever they are, they are the same in both cases, which is enough for our argument.)
  For (4), we may by (1) and (3) assume that $M$ is concentrated in degree $0$. By (2), there is a finite filtration of $M$ as a $T\tf$-module such that $\theta_1,\ldots, \theta_n$ act by $0$ on graded pieces. As all these submodules are perfect over $R\tf$ (by regularity), we may assume that $M$ is killed by $\theta_1,\ldots, \theta_n$. By the remark after \Cref{sec:appl-Hodge--Tate-definition-higgs-sen-complexes}, the assumption on the eigenvalues of $E^\prime(\pi)^{-1}\Theta_\pi$ allows us to apply \Cref{sec:isog-categ-perf-1-remark-for-isogeny-category-in-g-pi-case}, which yields a $\Theta_\pi$-stable retract $N$ of a perfect $R$-module such that $M= N\otimes_R R[1/p]$, and equipping this $N$ with the trivial $\theta_1, \ldots, \theta_n$-action does the job.   
\end{proof}

Consider the functor
\[
  \Xi_{X,\pi}\tf\colon \mathcal{P}erf(\mathcal{Z}_X)\tf\to \mathcal{H}igSen_{\X}^{\Hsm}
\]
defined as follows: By \Cref{sec:representations-g-representations-of-bg-over-bh} and the Sen theory from \cite[Theorem 2.5]{analytic_HT} complexes on $\mathcal{Z}_X$ embed fully faithfully in perfect complexes over $(X,\mathcal{HS}_{X,\pi})$ which are already perfect over $(X,\mathcal{O}_X)$, where 
\[
      \mathcal{HS}_{X,\pi}:=\mathrm{Sym}^\bullet_{\O_X}(T_{X|\O_K}\{1\})[\Theta_\pi]
    \]
    such that $[\Theta_\pi,\delta]=-E^\prime(\pi)\delta$ and $[\delta, \delta^\prime]=0$ for any $\delta, \delta^\prime\in T_{X|\O_K}$. The isogeny category of this category itself naturally maps to the category of perfect complexes over $(\X,\mathcal{HS}_{\X,\pi})$ which are already perfect over $(\X,\mathcal{O}_{\X})$, where $\mathcal{HS}_{\X,\pi}$ is the generic fibre of $ \mathcal{HS}_{X,\pi}$ on $\X$. Sending $\Theta_\pi$ to $E^\prime(\pi)\Theta$ defines an isomorphism between $\mathcal{HS}_{\X,\pi}$ and $\mathcal{HS}_{\X}$, giving the desired functor $\Xi_{X,\pi}$.
This functor is not very far from being an equivalence:

\begin{proposition}
  \label{sec:appl-Hodge--Tate-hig-sen-via-integral-stack}
  \begin{enumerate}
  	\item 
  If $X$ is affine, then $\Xi_{X,\pi}\tf$ is an equivalence of categories (up to idempotent completion). 
  
 \item For general $X$, this identifies  
  $\mathcal{H}igSen_{\X}^{\Hsm}$ with the global sections of the stackification of the functor on $X_{\mathrm{Zar}}$ defined by $U\mapsto \mathcal{P}erf({\mathcal{Z}_U})\tf^{\mathrm{idem}}$ where $-^{\mathrm{idem}}$ denotes idempotent completion.
\end{enumerate}
\end{proposition}
\begin{proof}
  Part (1) follows from \Cref{sec:appl-Hodge--Tate-properties-of-t-modules}. Indeed, by \Cref{sec:representations-g-representations-of-bg-over-bh} and \cite[Theorem 2.5]{analytic_HT} the category $\mathcal{P}erf(\mathcal{Z}_X)$ embeds into the category $\mathcal{D}(T)$ from above with essential image given by objects $M$ which are perfect over $R$ (and in particular $p$-complete) and such that $\Theta_\pi^p-E'(\pi)^{p-1}\Theta_\pi$ and $\theta_1,\ldots, \theta_n$ act locally nilpotently on $H^\ast(M\otimes_{\Z_p}^L \F_p)$. By \Cref{sec:appl-Hodge--Tate-properties-of-t-modules} the nilpotence condition on the $\theta_i$ is actually automatic and thus the claim follows from \Cref{sec:appl-Hodge--Tate-properties-of-t-modules} by expressing $\mathcal{H}igSen^{\mathrm{H}\text{-}\sm}_{\mathcal{X}}$ similarly as a full subcategory of $\mathcal{D}(T[1/p])$.
  
  By descent of perfect complexes on $\X$ the functor
  $U\mapsto \mathcal{H}igSen_{U^\rig}^{\Hsm}$ on $X_{\mathrm{Zar}}$ is
  a stack of $\infty$-categories. The $\Xi_{(-)}$ constitute a natural
  transformation of functors. It follows from the first part that this
  is an equivalence up to stackification and idempotent completions.
\end{proof}

Like in the geometric case of the last subsection, we can now put our work in the arithmetic case so far together and arrive at the following local $p$-adic Simpson functor in the arithmetic setting:
\begin{theorem}\label{t:local-p-adic-Simpson-functor-arithmetic}
	Let $X$ be an affine smooth $p$-adic formal scheme over $\O_K$. Assume that we are given a prismatic lift $(A,I)$ such that $A/I=R$ (for example induced by a toric chart), inducing a splitting $s$ of $X^\HT \to X$. Then for the generic fibre $\X$ of $X$, there is a fully faithful functor
	\[\begin{tikzcd}[column sep=0.1cm]
		\mathcal{H}igSen_{\X,\pi}^{\Hsm} \arrow[rr,hook,"\LS_{s}"] &                          & {\calPerf(\X_v)} \\
		& {\calPerf(X^\HT)\tf} \arrow[lu,"\Xi_{X,\pi}\tf",hook'] \arrow[ru,"\alpha"',hook] &   
	\end{tikzcd}\]
	from Hitchin-small Higgs-Sen perfect complexes on $\X$ into perfect complexes on $\X_v$, natural in $s$. 
\end{theorem}
\begin{proof}
	We have the equivalence $\Xi_{X,\pi}\tf$ by combining \Cref{sec:appl-Hodge--Tate-structure-of-x-ht-for-some-prismatic-lift} and 
	\Cref{sec:appl-Hodge--Tate-hig-sen-via-integral-stack}.
	The fully faithful functor $\alpha$ comes from \Cref{sec:tori-over-perfectoid-1-main-theorem-with-perfectoid-base}.
	We can thus define $\LS_s$ as $\alpha\circ\Xi_{X,\pi}\tf^{-1}$.
\end{proof}
Exactly as in the geometric case, \Cref{c:LS-coherent-modules}, we can deduce for discrete objects:
\begin{corollary}
	In the situation of \Cref{t:local-p-adic-Simpson-functor-arithmetic}, $\mathrm{LS}_s$ restricts to a fully faithful functor 
	\[\mathrm{LS}_s^{\rm coh}: \left\{ \begin{array}{@{}c@{}l}\text{Hitchin-small coherent}\\\text{Higgs-Sen modules on $\X$}\end{array}\right\}\hookrightarrow \left\{ \begin{array}{@{}c@{}l}\text{perfect }\O_{\X_v}\text{-modules}\\\text{on $(\X_v,\O_{\X_v})$}\end{array}\right\}.\]
\end{corollary}

\section{Globalization in terms of square-zero lifts}\label{sec:globalisation}
Let $X$ be a $p$-adic formal scheme. We assume that $X$ is smoothoid and lives over a perfectoid base ring $R_0$, or that $X$ is smooth over the ring of integers of a $p$-adic field. In \Cref{t:local-p-adic-Simpson-functor-geometric}, resp. \Cref{t:local-p-adic-Simpson-functor-arithmetic}, we have constructed a local $p$-adic Simpson functor for perfect complexes under the condition that $X^\HT$ is split. The reason we call this functor ``local'' is that the assumption that $X^\HT \to X$ splits is a restrictive one. It is satisfied when $X$ is affine, but rarely in general. 

The goal of this section is to relax this assumption and thus to construct a ``global'' $p$-adic Simpson correspondence. In the $p$-adic Simpson literature, e.g.\ in  \cite{faltings2005p}\cite{abbes2016p}\cite{wang2021p}, when the base $R_0$ is the ring of integers $\mathcal{O}_C$ of an algebraically closed perfectoid extension of $\Q_p$, this is usually achieved by switching from toric charts to the datum of lifts of $X$ along 
$A_{\inf}(\O_C)/(\ker \theta)^2\to \O_C$. We will now make an analogous switch of perspective in our setting, in terms of $X^\HT$. As we will explain, there is an intrinsic motivation to do so, given by the inherent relation between Cartier--Witt divisors and square-zero extensions. As a special case, this also leads to a geometric reinterpretation of the aforementioned construction of Faltings, Abbes--Gros and Wang, see \S\ref{sec-comparison-with-prev-constructions}.

More generally, for applications to the arithmetic case of smooth formal schemes $X$ over a $p$-adic base $\O_K$, it is beneficial to consider lifts of $X$ to larger square-zero extensions of $\O_C$ inside $A_{\inf}(\O_C)/(\ker \theta)^2\tf$. In a second step, these will allow us to also glue the local correspondences in the arithmetic case. As an application, we get for the base-change $X_{\O_\C}$ a \textit{canonical} global $p$-adic Simpson functor for perfect complexes at the expense of a stricter convergence condition. 

\subsection{Square-zero lifts of perfectoid base rings}
We start by discussing the more general variants of the square-zero thickening $A_{\inf}(\O_C)/\xi^2\to \O_C$ that we need. 
 \begin{definition}\label{def:sq-zero-pushouts}
 	Let $S$ be a $p$-torsion free perfectoid ring with associated perfect prism $(A,I)$. Set
 	$A_2:=A/I^2$,
 	this is a square-zero extension of $S$ by $S\{1\}=I/I^2$.
 	Let $x\in S\tf$ be such that $S\subseteq xS$ and consider the $S$-submodule $xS\{1\}\subseteq S\{1\}\tf$ defined as the image of the $S$-module map $S\{1\}\xrightarrow{\cdot x}S\{1\}\tf$. Then we define a square-zero extension $A_2(x)\to S$ as the pushout
 \[
 \begin{tikzcd}
 	0 \arrow[r] & S\{1\} \arrow[d] \arrow[r] & A_2 \arrow[d] \arrow[r] & S \arrow[d] \arrow[r] & 0 \\
 	0 \arrow[r] & xS\{1\} \arrow[r]          & A_2(x) \arrow[r]         & S \arrow[r]           & 0
 \end{tikzcd}\]
inside $A_2\tf$. Note that this is naturally an $A_2$-algebra. As $S$ is $p$-torsionfree, we have $A_2(1)=A_2$.
\end{definition}

\begin{definition}\label{def:x-lift}
	For any $p$-adic formal scheme $X$ over $S$, an $x$-lift is a flat lift $\tilde{X}$ of $X$ to $A_2(x)$.
\end{definition}

Such lifts arise naturally in our setup: Let $K$ be a $p$-adic field with residue field $k$. Let $C=\widehat{\overline{K}}$. Recall that by henselian lifting, there is then a canonical lift\footnote{This canonical lift can uniquely be characterized by the requirements that it is natural in $K$ and that it is induced by the composition $W(k)\tf=W(\mathcal{O}_K^\flat)\tf\to W(\mathcal{O}_C^\flat)\tf$ if $K=W(k)\tf$ is unramified.} of $K\to C$ to a morphism
\[s:K\to B_{\dR}^+/\xi^2=A_{\inf}(\O_C)/\xi^2\tf.\]
However, the map $s$ does not in general send $\O_K$ into  $A_{\inf}(\O_C)/\xi^2$. Rather, we have the following:

\begin{proposition}\label{p:non-integral-lift-of-O-K-to-A2}
	Let $A_2:=A_\inf(\O_C)/\xi^2$. Then under the canonical lift $s:K\to B_{\dR}^+/\xi^2=A_2\tf$, the  $A_2$-submodule generated by $s(\O_K)$ is precisely $\delta_{\O_K|W(k)}^{-1}\xi A_2+A_2$. In particular, there is a canonical lift
	$s:\O_K\to A_2(e^{-1})$
	where for any uniformizer $\pi \in \mathcal{O}_K$ with minimal polynomial $E$ over $W(k)$, we let $e=E'(\pi)$ be the induced generator of the different ideal $\delta_{\O_K|W(k)}$.
\end{proposition}
\begin{proof}
	If $K=W(k)$ is unramified, then $s$ is induced by the lift
	$\O_K=W(k)\to W(\O_C^\flat)\to A_2$,
	using that $k\cong \mathcal{O}_K^\flat$.
	So the statement holds in this case. In general, the map $s:K\to A_2\tf$ is therefore uniquely determined by $s(\pi)$, which needs to satisfy $E(s(\pi))=0$.

        Let $(\mathfrak{S}=W(k)[[u]],(E(u)))$ be the Breuil-Kisin prism associated with $\pi\in \mathcal{O}_K$, and consider the morphism of prisms $f\colon \mathfrak S\to A_{\inf}(\O_C)$, $u\mapsto [\pi^\flat]$. Then we have $(\xi)=(f(E(u)))=(E([\pi^\flat]))$ in $A_{\inf}(\O_C)$ by \cite[Lemma 3.5]{Bhatta}, hence there is a unit $v\in A_{\inf}(\O_C)$ such that $v\xi=E([\pi^\flat])$.
	
	Since $\theta(s(\pi))=\pi=\theta([\pi^\flat])$, we can write
	$s(\pi)-[\pi^\flat]=c\xi$
	for some $c\in K$. Then inside $A_2\tf$,
	\[ 0=E([\pi^\flat]+c\xi)=E([\pi^\flat])+E'([\pi^\flat])c\xi=(v+E'([\pi^\flat])c)\xi.\]
	Thus $c\xi=-vE'([\pi^\flat])^{-1}\xi$. Since $v$ is a unit in $A_2$, it follows that
	\[ s(\pi)\cdot A_2+A_2=(c\xi+[\pi^\flat])A_2+A_2=\delta_{K|W(k)}^{-1}\xi A_2+A_2\]
	and the statement follows because this is already a subalgebra of $A_2\tf$.
\end{proof}
\begin{corollary}\label{c:canonical-e-lifts}
	Let $X$ be a $p$-adic formal scheme over $K$. Then for $x:=e^{-1}$, the base-change $X_{\O_C}$ has a canonical Galois-equivariant $x$-lift in the sense of \Cref{def:x-lift} defined by
	\[\widetilde{X}_{\O_C}:=X\times_{\Spf(\O_K)}\Spf(A_2(x)).\]
\end{corollary}
\subsection{Square-zero lifts of $p$-adic formal schemes}
Fix a $p$-torsionfree perfectoid base ring $R_0$ over $\Z_p$. For simplicity, let us make the harmless additional assumption that $R_0$ contains a primitive $p$-th root of unity $\zeta_p\in R_0$. Let  $S$ be a $p$-torsionfree perfectoid base ring over $R_0$ and fix $(A,I)$ and $x\in S\tf$ as in \Cref{def:sq-zero-pushouts}. In the following, we consider $\mathbb{G}_a$ as an affine formal group over $S$ and let $x\mathbb{G}_a$ be the affine formal group over $S$ with a morphism $\Ga\to x\Ga$, with $x\Ga$ abstractly isomorphic to $\Ga$ such that $\Ga\to x\Ga$ identifies with the multiplication by $x$. In particular, for any $x$ and $p$-torsionfree $S$-algebra $T$ the $T$-valued points of $x\Ga$ are given by $x\O(T)\subseteq \O(T)\tf$. 
\begin{definition}
\label{def:gerb-of-x-lifts}
Let $\mathcal{R}_x$ be the square-zero extension of $\mathbb{G}_a$ by $x\mathbb{G}_a\{1\}[1]$ classifying the obstruction to lifting an $S$-algebra to $A_2(x)$, i.e.\ the $\infty$-sheaf of animated $S$-algebras that sends ($p$-nilpotent) animated $S$-algebras $T$ to the animated square zero extension $T\oplus T\{1\}[1]$ with $S$-algebra structure
    \[
      S\xrightarrow{\delta_x} S\oplus  S\{1\}[1]\to  T\oplus  T\{1\}[1],
    \]
    where the second map is the base change from $S$ to $T$ of the trivial extension of $S$ by $ S\{1\}[1]$, and the first map classifies the square zero extension $A_2(x)$ of $S$ by $S\{1\}[1]$, cf.\ \cite[Section 5.1.9]{cesnavicius2019purity}.
\end{definition}
Let now $X$ be any $p$-adic formal scheme over $S$.
\begin{definition}
Let $\Lft_{X,x}$ be the stack sending any ($p$-nilpotent) $S$-algebra $T$ to the morphisms $\Spf(\mathcal{R}_x(T))\to X$ over $S$. 
\end{definition}
When $x=1$, we just write $\mathcal{R}:=\mathcal{R}_1$ and $\Lft_{X}=\Lft_{X,1}$.
By construction of $\mathcal{R}_x$ we have for any open subscheme $U\subseteq X$ a natural equivalence of groupoids between
    $\Lft_{X,x}(U)$
    and the groupoid of $x$-lifts of $U$ (cf.\ \Cref{def:x-lift}). Indeed, $x$-lifts define a stack for the Zariski topology on open subschemes of $X$, and so does $U\mapsto \Lft_{X,x}(U)$. Hence it suffices to argue for $U=\Spf(T)\subseteq X$ open and affine and to show that the equivalence is natural with respect to restrictions between open subschemes. In the affine case, given an $x$-lift $\tilde{U}$ of $U$ this lift of $T$ to $A_2(x)$ is classified by an $S$-algebra map $\delta_{\tilde{U}}\colon T\to \mathcal{R}_x(T)=T\oplus T\{1\}[1]$ by derived deformation theory. Conversely, each such $S$-algebra morphism $\delta\colon T\to \mathcal{R}_x(T)$ defines an $x$-lift by taking the pullback of the morphism $(\mathrm{Id}_T,0)\colon T\to T\oplus T\{1\}[1]$ (cf.\ \cite[Section 5.1.9]{cesnavicius2019purity}). 
    
    Moreover, for $U=X$ the groupoid $\Lft_{X,x}(X)$ identifies with the fiber of the map
    \[
      \mathrm{Map}_{\O_X}(L_{X/\Z_p},x\cdot \Ga\{1\}[1])\to \mathrm{Map}_{\O_X}(\O_X\otimes_{S}L_{S/\Z_p},x\cdot \Ga\{1\}[1])
    \]
    over $L_{S/\Z_p}\to I(x)[1]$, which classifies the square zero extension $A_2(x)$. Here $x\cdot \Ga\{1\}\cong I(x)\otimes_{S}\O_X$.
\\

Let now $0\neq x\in S\tf$ be such that $S \subset xS$  and let $z\in S$ be such that $z\in S\tf^\times$ and $\omega  xz \in S$ with $\omega=(\zeta_p-1)^{-1}$ as before, e.g.\, $x=1$ and $z=\zeta_p-1$. These assumptions imply that the maps
    \[
      \can\colon \Ga^\sharp\to x \Ga
\quad \text{and} \quad
      x \Ga\xrightarrow{\cdot z} \Ga^\sharp
    \]
    of formal group schemes over $S$ are well-defined due to the natural divided powers on $xz$.
    
    The relation of our discussion to $X^\HT$ is furnished by the following observation of Bhatt-Lurie:
\begin{proposition}\label{p:pushout-def-u}
  Let
    $
      \overline{W}(-):=\mathrm{cone}\big(I_0\otimes_{A_0}W(-)\to W(-)\big).
    $
    This is a square zero extension of $\Ga$ by $\Ga^\sharp\{1\}[1]$ on $p$-nilpotent $S$-algebras. There exists a unit\footnote{A priori $u\in R_0^\times$, but naturality forces $u\in \Z_p^\times$ (e.g.\ consider $\Z_p^\cycl$ with Galois action). Conjecturally, $u=1$.} $u\in \Z_p^\times$ such that the pushout of $\overline{W}(-)$ along
    $
      u\cdot \can\colon \Ga^\sharp\{1\}[1]\to x\Ga\{1\}[1]$
    is the square zero extension $\mathcal{R}_x(-)$ of $\Ga$ by $x\Ga\{1\}[1]$.
\end{proposition}
\begin{proof}
When $x=1$, this is  \cite[Construction 5.10, Remark 5.11]{bhatt2022prismatization}. The general case follows because the natural map $\mathcal{R}(-)\to \mathcal{R}_x(-)$ is clearly the pushout along $\can\colon \Ga\to x \Ga$.
\end{proof}

\begin{proposition}
	 The pushout of the $B\mathcal{T}^\sharp_{X|S}\{1\}$-gerbe $X^\HT\to X$ along \[u\cdot \can\colon \mathcal{T}^\sharp_{X|S}\{1\}\to x\mathcal{T}_{X|S}\{1\}\] is canonically isomorphic to $\Lft_{X,x}$. Here $x\mathcal{T}_{X|S}\{1\}:=\mathcal{T}_{X|S}\otimes_{\Ga} x\Ga\{1\}$.
\end{proposition}
\begin{proof}
	This can be seen as in the proof of \cite[Proposition 5.12]{bhatt2022prismatization}. In fact, the same arguments as in \cite[Proposition 5.12]{bhatt2022prismatization} show that $\Lft_{X,x}$ is an $x\mathcal{T}_{X|S}\{1\}$-gerbe over $X$, and visibly the composition $X^\HT\to \Lft_{X}\to \Lft_{X,x}$ is linear over $u\cdot \can\colon \mathcal{T}^\sharp_{X|S}\{1\}\to \mathcal{T}_{X|S}\{1\}\to x\mathcal{T}_{X|S}\{1\}$.
\end{proof}
\begin{definition}
      Let $z_\ast(X^\HT)$ be the pushforward of the $\mathcal{T}^\sharp_{X|S}\{1\}$-gerbe $X^\HT$ along the multiplication map $z\colon \mathcal{T}^\sharp_{X|S}\{1\}\to \mathcal{T}^\sharp_{X|S}\{1\}$. Then there is a diagram
\[\begin{tikzcd}
	{X^\HT} & {\Lft_{X,x}} \\
	& {z_\ast X^\HT}
	\arrow[from=1-1, to=1-2]
	\arrow[from=1-2, to=2-2]
	\arrow[from=1-1, to=2-2]
\end{tikzcd}
\quad \text{which is linear over}\quad \begin{tikzcd}
	{\mathcal{T}^\sharp_{X|S}\{1\}} & {x\cdot \mathcal{T}_{X|S}\{1\}} \\
	& {\mathcal{T}^\sharp_{X|S}\{1\}.}
	\arrow["{u\cdot \can}", from=1-1, to=1-2]
	\arrow["{u^{-1}z}", from=1-2, to=2-2]
	\arrow["z"', from=1-1, to=2-2]
      \end{tikzcd}\]
\end{definition}

    \begin{remark}
      \label{sec:appl-Hodge--Tate-change-of-canonical-higgs-field}
      Let $\mathcal{F}\in \mathcal{D}(z_\ast(X^\HT))$, and let $
        \theta_{\mathcal{F}}\colon \mathcal{F}\to \mathcal{F}\otimes_{\O_X} \Omega^1_{X|S}\{-1\}
     $
      be its canonical Higgs field, cf.\ \Cref{sec:representations-g-1-canonical-higgs-field-on-gerbes}. Let $g\colon X^\HT\to z_\ast(X^\HT)$ be the morphism introduced above. Then the canonical Higgs field of $g^\ast\mathcal{F}$ is the composition
      \[
        g^\ast\mathcal{F}\xrightarrow{g^\ast\theta_{\mathcal{F}}} g^\ast\mathcal{F}\otimes_{\O_X} \Omega^1_{X}\{-1\}\xrightarrow{\Id\otimes z} g^\ast\mathcal{F}\otimes_{\O_X} \Omega^1_{X}\{-1\}.
      \]
      Using that $z\in S\tf^\times$, we can conclude that
      \[
        g^\ast\colon \mathcal{P}erf(z_\ast(X^\HT))\to \mathcal{P}erf(X^\HT)
      \]
      is fully faithful on isogeny categories and (up to stackification via open subsets of $X$) its essential image in (the stackification of) $\mathcal{P}erf(X^\HT)\tf$ is given by those $\mathcal{E}\in \mathcal{P}erf(X^\HT)\tf$ whose canonical Higgs field $\theta_{\mathcal{E}}$ satisfies that $\frac{1}{z}\theta_{\mathcal{E}}$ is topologically nilpotent (in the sense of \Cref{sec:representations-g-1-isogeny-category-fully-faithful}). Indeed, these statements can be checked locally on $X$, and follow from \Cref{sec:isog-categ-perf-1-multiplication-by-x-on-v-sharp} if $X^\HT$ is split. 
    \end{remark}

    \begin{theorem}
      \label{sec:appl-Hodge--Tate-functor-for-x-lift}
      \begin{enumerate}
      \item Each $x$-lift $\tilde{X}$ induces a natural isomorphism 
      of $\mathcal{T}^\sharp_{X}\{1\}$-gerbes over $X$
      \[
        z_\ast(X^\HT)\cong B\mathcal{T}^\sharp_{X}\{1\}.
      \]
      \item  Let
      $\Phi_{\tilde{X},x,z}\colon X^\HT\to z_\ast(X^\HT)\cong B\mathcal{T}^\sharp_{X}\{1\}$ be the induced morphism. Then the pullback
      \[
        \Phi_{\tilde{X},x,z}^\ast\colon \mathcal{P}erf(B\mathcal{T}^\sharp_{X}\{1\})\to \mathcal{P}erf(X^\HT)
      \]
      is linear over $z\colon \mathcal{T}^\sharp_{X}\{1\}\to \mathcal{T}^\sharp_{X}\{1\}$. It
      is fully faithful after inverting $p$. 
      \item The essential image in $\mathcal{P}erf(X^\HT)\tf$ is up to stackification on $X_{\mathrm{Zar}}$ given by those $\mathcal{E}$ whose canonical Higgs field $\theta_{\mathcal{E}}$ satisfies that $\frac{1}{z}\theta_{\mathcal{E}}$ is topologically nilpotent, i.e., locally on $X$ the Higgs field $\frac{1}{z}\theta_{\mathcal{E}}$ is topologically nilpotent for any choice a splitting of $X^\HT$.
      \end{enumerate}
    \end{theorem}
    \begin{proof}
      The first statement is clear because $\mathcal{L}ft_{X,x}$ maps to $z_\ast(X^\HT)$. The fully faithfulness in (2) follows from \Cref{sec:isog-categ-perf-1-multiplication-by-x-on-v-sharp}. Part (3) follows from \Cref{sec:appl-Hodge--Tate-change-of-canonical-higgs-field}.
    \end{proof}

\subsection{Applications to the global $p$-adic Simpson correspondence, geometric case}
We can finally harvest the fruit of our work.
We first consider the smoothoid case. We assume that the base ring $R_0$ contains a primitive $p$-th root of unity $\zeta_p\in R_0$.  Let $x,y\in R_0\tf^\times$ be such that $|x|\geq 1$ and $|y|\leq |\zeta_p-1|$ and  let $z=y/x$. We are particularly interested in setting $x=1$ and $y=z=\zeta_p-1$. The following proposition is a direct corollary of \Cref{sec:appl-Hodge--Tate-functor-for-x-lift} and for $x=1$ settles \Cref{sec:smoothoid-case-1-complexes-on-the-Hodge--Tate-stack-for-choice-of-lift-introduction}.

    \begin{proposition}
      \label{sec:appl-Hodge--Tate-embedding-from-colimit}
     Let $X$ be a qcqs smoothoid formal scheme over $R_0$ with $x$-lift $\tilde{X}$. Then
      \[
        \Phi^\ast_{\tilde{X},x, z}\colon \mathcal{P}erf(B\mathcal{T}^\sharp_{X}\{1\})\tf\to \mathcal{P}erf(X^\HT)\tf.
      \]
      is fully faithful, with essential image given by those objects $\mathcal{E}$ for which the scaled canonical Higgs field $\tfrac{1}{z}\theta_{\mathcal{E}}$ is topologically nilpotent locally on $X$.
    \end{proposition}
    \begin{proof}
      This is a reformulation of \Cref{sec:appl-Hodge--Tate-functor-for-x-lift}.
    \end{proof}

We can thus prove \Cref{sec:smoothoid-case-1-derived-local-p-adic-simpson-correspondence-introduction}, regarding the global $p$-adic Simpson correspondence. 
 
\begin{theorem}
  \label{sec:smoothoid-case-1-derived-local-p-adic-simpson-correspondence}
 Let $X$ be a smoothoid $p$-adic formal scheme over $R_0$, with generic fiber $\X$.
  \begin{enumerate}
 \item  Each $x$-lift $\tilde{X}$ of $X$ to $A_2(x)$ induces a fully faithful functor, natural in $\tilde{X}$,
  \[
    \mathrm{S}_{\tilde{X},x,z}\colon \mathcal{H}ig_{\mathcal{X}}^{\omega z\text{-}\Hsm} \to \mathcal{P}erf(\X_v).
  \]
  \item If $X$ is affine with prismatic lift $(A,I)$ inducing a splitting $s$ of $X^\HT$ and $\tilde{X}=\mathrm{Spf}(A/I^2)$, set $y:=\zeta_p-1$ and $z=y/x$, then $\mathrm{S}_{\tilde{X},x,z}$ is the composition of $\mathcal{H}ig_{\mathcal{X}}^{\omega z\text{-}\Hsm}\hookrightarrow \mathcal{H}ig_{\mathcal{X}}^{\omega\text{-}\Hsm}$ with $\LS_s$ from \Cref{t:local-p-adic-Simpson-functor-geometric}.
 \item The essential image of $\mathrm{S}_{\tilde{X},x,z}$ is contained in the essential image of $\alpha_X^\ast$ from \Cref{sec:tori-over-perfectoid-1-main-theorem-with-perfectoid-base}, up to idempotent completion and replacing $X$ by an affine open cover. Conversely, any $\mathcal E$ in the essential image of $\alpha_X^\ast$ admits a canonical Higgs field $\Theta$ with values in $\Omega^1_{\X}\{-1\}$; then $\mathcal E$ lies in the essential image of $\mathrm{S}_{\tilde{X},x,z}$ if and only if $\frac{1}{z}\Theta$ is topologically nilpotent.
 \end{enumerate}
\end{theorem}
\begin{proof}
	By naturality in $\tilde{X}$, it will suffice to construct $\mathrm{S}_{\tilde{X},x,z}$ locally, and we thus may assume that $X$ is qcqs.
  By \Cref{sec:appl-Hodge--Tate-hig-sen-via-integral-stack-smoothoid-case} and \Cref{sec:tori-over-perfectoid-1-main-theorem-with-perfectoid-base}, it suffices to construct a natural functor
  \[
    \mathcal{P}erf(B\mathcal{T}^\sharp_{X}\{1\})\to \mathcal{P}erf(X^\HT),
  \]
  which is fully faithful on isogeny categories, with the prescribed compatibilities of (2) and (3). Here we use that $\mathcal{P}erf(\X_v)^{\mathrm{idem}}=\mathcal{P}erf(\X_v)$. \Cref{sec:appl-Hodge--Tate-embedding-from-colimit} tells us that we can take $\Phi^\ast_{\tilde{X},x,z}$.
\end{proof}

In what follows, when $X$ has a lift $\tilde{X}$ to $A_2$, we will simply write
$
\mathrm{S}_{\tilde{X}}:= \mathrm{S}_{\tilde{X},1,\zeta_p-1}.
$

\subsection{Applications to the global $p$-adic Simpson correspondence, arithmetic case}
 Let $K$ be a $p$-adic field and $X$ a smooth formal scheme over $X_0:=\Spf(\O_K)$ with rigid generic fibre $\X$.
    The rest of this section is devoted to proving the following, which is \Cref{sec:arithmetic-case-1-derived-local-p-adic-simpson-in-arithmetic-case} in the introduction:

    \begin{theorem}
      \label{sec:arithmetic-case-1-derived-local-p-adic-simpson-in-arithmetic-case-plus-description-of-perf-on-x-ht}
      The choice of a uniformizer $\pi \in \mathcal{O}_K$ gives rise to a fully faithful functor
      \[
       \mathrm{S}_\pi \colon \mathcal{H}igSen_{\X}^{\Hsm} \to \mathcal{P}erf(\X_v).
     \]
    \end{theorem}
    Characterizing the essential image of $\mathrm S_{\pi}$ is more difficult than in \Cref{sec:smoothoid-case-1-derived-local-p-adic-simpson-correspondence}: For $X=\Spa(\O_K)$, we gave a description in \cite[Theorem 1.3]{analytic_HT}. But this does not immediately globalize  because there is in general no canonical Sen operator on perfect complexes on $\X_v$. For coherent modules, we give a description of the essential image of $\mathrm S_{\pi}$  in \cite[Theorem~5.17, Corollary~5.21]{AHLB-companion}.
    
    Towards a proof of \Cref{sec:arithmetic-case-1-derived-local-p-adic-simpson-in-arithmetic-case-plus-description-of-perf-on-x-ht},
    assume first that $X$ is qcqs.
Let $C:=\widehat{\overline{K}}$ be a completed algebraic closure of $K$ with Galois group $\Gamma$. Let
$ X_{\O_C}:=\Spf(\mathcal{O}_C)\times_{X_0} X
$
be the base change of $X$ to $\O_C$ along $\O_K\to \O_C$. Note that
\[
 X_{\O_C}^\HT\cong X^\HT\times_{X_0^\HT} \Spf(\O_C)
\]
with $\Spf(\O_C)\to X_0^\HT$ the canonical lift of $\Spf(\O_C)\to X_0$. Namely, passing to Hodge--Tate stacks commutes with Tor-independent limits and $\Spf(\O_C)^\HT\cong \Spf(\O_C)$.
Set
\[
  \Phi\colon \mathcal{Z}_{X_{\O_C}}:=\Spf(\O_C)\times_{X_0^\HT}\mathcal{Z}_X \to \mathcal{Z}_X,
\]
with $\mathcal{Z}_X$ defined after \Cref{sec:appl-Hodge--Tate-structure-of-x-ht-for-some-prismatic-lift}. This map is $\Gamma$-equivariant for the trivial action on the target.
\begin{lemma}
  \label{sec:appl-Hodge--Tate-phi-is-fully-faithful-on-isogeny-categories}
  The following pullback functor is fully faithful:
  \[
    \Phi^\ast\colon \mathcal{P}erf(\mathcal{Z}_X)\tf\to \mathcal{P}erf([\mathcal{Z}_{X_{\O_C}}/\underline{\Gamma}])\tf.
  \]
\end{lemma}
\begin{proof}
	We follow the proof of  \Cref{sec:form-reduct-prov-1-fully-faithfulness-for-beta}
	 combined with flat base-change.
\end{proof}

Next, we want to relate $\mathcal{Z}_{X_{\O_C}}$ and $X_{\O_C}^\HT$. This is possible thanks to \Cref{sec:appl-Hodge--Tate-embedding-from-colimit}. Indeed, consider the canonical $\Gamma$-equivariant $e^{-1}$-lift $\widetilde{X}_{\O_C}\to \Spf(A_2(e^{-1}))$ from \Cref{c:canonical-e-lifts}.
By \Cref{sec:appl-Hodge--Tate-functor-for-x-lift}, applied with $x=e^{-1}$ and $y=p$ (so that $z=ep$), we get a $\Gamma$-equivariant morphism 
\[
 \Phi_{\widetilde{X}_{\O_C}}\colon X^\HT_{\O_C}\to B\mathcal{T}^\sharp_{X_{\O_C}}\{1\}\cong \mathcal{Z}_{X_{\O_C}}
\]over $X\times_{X_0}X_0^\HT$,
which is linear over the multiplication by $ep$ on $B\mathcal{T}^\sharp_{X_{\O_C}}\{1\}$. By \Cref{sec:appl-Hodge--Tate-embedding-from-colimit},
\[
  \Phi^\ast_{\widetilde{X}_{\O_C}}\colon \mathcal{P}erf(\mathcal{Z}_{X_{\O_C}})\tf\to \mathcal{P}erf(X^\HT_{\O_C})\tf
\]
is fully faithful, and one deduces the same for the functor 
on $\Gamma$-equivariant objects
\[
  \mathcal{P}erf([\mathcal{Z}_{X_{\O_C}}/\underline{\Gamma}])\tf\to \mathcal{P}erf([X^\HT_{\O_C}/\underline{\Gamma}])\tf.
\]

\begin{proof}[Proof of \Cref{sec:arithmetic-case-1-derived-local-p-adic-simpson-in-arithmetic-case-plus-description-of-perf-on-x-ht}]
  By naturality in $X$, we may assume that $X$ is qcqs.
  By \Cref{sec:appl-Hodge--Tate-hig-sen-via-integral-stack}, pro-\'etale descent for v-perfect complexes and \Cref{sec:tori-over-perfectoid-1-main-theorem-with-perfectoid-base} it suffices to construct a fully faithful functor
  \[
    \mathcal{P}erf(\mathcal{Z}_X)\tf\to \mathcal{P}erf([X^\HT_{\O_C}/\underline{\Gamma}])\tf.
  \]
  Here, we can take the composition of $\Phi^\ast_{\widetilde{X}_{\O_C}}\circ \Phi^\ast$ with the functor $(M,\theta_M,\Theta_\pi)\mapsto (M,(ep)^{-1}\theta_M, \Theta_\pi)$ on Higgs--Sen perfect complexes. 
\end{proof}

\begin{remark}
  \label{sec:appl-Hodge--Tate-remark-on-normalization}
 The normalization by $(ep)^{-1}$ in this proof  makes the Higgs field of a perfect complex on $\mathcal{Z}_X$ compatible with the canonical Higgs field of the associated v-perfect complex.
\end{remark}

    \subsection{Comparison with previous constructions}\label{sec-comparison-with-prev-constructions}
    Let $X$ be a qcqs smoothoid formal scheme over a perfectoid base $S$ with sheaf of $p$-completed differentials $\Omega^1_{X}$ (cf \Cref{def:absolute-diff-of-smoothoid}) and rigid generic fibre $\X$. Let $(A_0,I_0)$ be the perfect prism associated to $S$, i.e., $A_0/I_0\cong S$. Let $\tilde{X}$ be a flat lift of $X$ to $A_0/I_0^2$. 
 To compare $\mathrm{S}_{\tilde{X}}$ to former constructions, we first make it more explicit. Set $z=y=\zeta_p-1$. Let $\psi:=\Phi_{\tilde{X},1,z}\colon X\to z_\ast(X^\HT)\cong B\mathcal{T}_{X}^\sharp\{1\}$ be the morphism from \Cref{sec:appl-Hodge--Tate-functor-for-x-lift}. 
 \begin{definition}
 The pullback of $\psi_\ast(\O_X)$ along $j:X^\HT\to z_\ast(X^\HT)$ followed by pullback to $\X_v$ defines a ring sheaf 
$\mathcal{B}^+_{\tilde{X}}$
  on $\X_v$ with a Higgs field
  $$
  \Theta_{\mathcal{B}_{\tilde{X}}^+}\colon \mathcal{B}_{\tilde{X}}^+\to \mathcal{B}_{\tilde{X}}^+\otimes_{\O_{\X_v}} \mu^\ast \Omega^1_{X}\{-1\}
  $$
  where $\mu$ is the pullback from the Zariski site of $X$ to the v-site of $\X$. We then set \[\mathcal{B}_{\tilde{X}}:=\mathcal{B}_{\tilde{X}}^+\tf=\alpha_X^{\ast}(j^{\ast}\psi_\ast(\O_X)\tf)\] with its associated Higgs field $\Theta_{\mathcal{B}_{\tilde{X}}}$.
  \end{definition}
  
   \begin{lemma}\label{explicit-formula_ls}
   	Let $\mathcal{M}=(M,\theta_M)\in \mathcal{P}erf(B\mathcal{T}^\sharp_{X}\{1\})$, considered as an object of $\mathcal{H}ig_{\mathcal{X}}^{\Hsm}$  via \Cref{sec:appl-Hodge--Tate-hig-sen-via-integral-stack-smoothoid-case}). Then there is a natural isomorphism
   	\[\mathrm{S}_{\tilde{X}}(\mathcal{M}) \cong \mathrm{Dol}(\mathcal{B}_{\tilde{X}} \otimes_{\mathcal{O}_{\X}}^L \mu^\ast M, \Theta_{\mathcal{B}_{\tilde{X}}} \otimes \mathrm{Id} + \mathrm{Id} \otimes \theta_M).\]
   	In particular, if $M$ is a vector bundle on $\X$, then  
   	\[\mathrm{S}_{\tilde{X}}(\mathcal M) \cong \ker(\mathcal{B}_{\tilde{X}} \otimes_{\mathcal{O}_{\X}} \mu^\ast M \xrightarrow{\Theta_{\mathcal{B}_{\tilde{X}}} \otimes \mathrm{Id} + \mathrm{Id} \otimes \theta_M}  \mathcal{B}_{\tilde{X}} \otimes_{\mathcal{O}_{\X}} \mu^\ast M \otimes_{\mathcal{O}_{\X}}^L \nu^\ast \Omega_{\X}^1\{-1\}).\]
   \end{lemma}
   \begin{proof}
   	Let $V:=\mathrm{S}_{\tilde{X}}(\mathcal{M})$.
   	Let $\mathcal{Y}=\mathrm{Spa}(S\tf,S) \to \X$ be an affinoid perfectoid object of the v-site of $\X$. We want to describe the sections of $V$ on $\mathcal{Y}$:
   The induced map $Y:=\mathrm{Spf}(S) \to X$ lifts uniquely to a map $Y \to X^\HT$, which we can compose with the natural map $X^\HT \to z_\ast(X^\HT)$ to get a map $f: Y \to z_\ast(X^\HT)$. The  definition of $V$ and the projection formula for $f$ give
   		\[ R\Gamma(\mathcal{Y},V)\cong R\Gamma(z_\ast(X^\HT),Rf_{\ast}f^{\ast}\O\otimes_{\mathcal{O}_{z_\ast(X^\HT)}} \mathcal M)\tf.\]
   		Since the pullback of $f$ along $\psi: X \to z_\ast(X^\HT)$ is the affine morphism $Y \times_X \mathcal{T}_{X}^\sharp\{1\} \to X$, we have $Rf_\ast f^{\ast}\O=f_{\ast}\O$. Moreover, as $z_\ast(X^\HT)\cong B\mathcal{T}_{X}^\sharp\{1\}$, the cohomology of any object in $\mathcal{D}(z_\ast(X^\HT))$ can be described as Dolbeault cohomology (by the same argument as for \Cref{computation-cohomology-koszul-resolution}). Hence,
		\[
		R\Gamma(\mathcal{Y},V)\cong R\Gamma(X,\mathrm{Dol}(f_\ast \mathcal{O} \otimes_{\mathcal{O}_{z_\ast(X^\HT)}}^L \mathcal{M}))\tf= R\Gamma(\mathcal{Y},\mathrm{Dol}(\mathcal{B}_{\tilde{X}} \otimes_{\mathcal{O}_\X}^L \mu^\ast M, \Theta_{\mathcal{B}_{\tilde{X}}} \otimes \mathrm{Id} + \mathrm{Id} \otimes \theta_M)).\]
		We deduce the desired formula.
		\end{proof}
\begin{definition}
\label{def:faltings-small-higgs-bundles}
A Higgs bundle on $\mathcal X$ is called \textit{Faltings--small} if Zariski-locally on $X$, it admits a model $(\mathfrak{M},\theta_{\mathfrak{M}})$ consisting of a vector bundle $\mathfrak{M}$ on $X$ and a Higgs field $\theta_\mathfrak{M} :\mathfrak{M}\to \mathfrak{M}\otimes \Omega_{X}^1\{-1\}$ such that the reduction of $(\mathfrak{M},\theta_{\mathfrak{M}})$ mod $p^{\alpha}$ is isomorphic to the trivial Higgs bundle for some $\alpha >1/(p-1)$. We denote the category of Faltings--small Higgs bundles on $\X$ by $\Higgs^{\Fsm}(\X)$.
\end{definition}

  \begin{remark}\label{r:Fsm-vs-Hsm}
    	It is clear that any Faltings--small Higgs bundle is Hitchin--small. 
    	We do not currently know if any Hitchin--small bundle is also Faltings--small: the former notion seems more general, for example any nilpotent Higgs bundle is Hitchin--small. But e.g.\ in the affine case, it turns out that Hitchin-small bundles are also Faltings-small, see  \cite[Cor.\ IV.3.6.4]{abbes2016p}. 
    \end{remark}

In \cite{wang2021p}, Wang constructs a fully faithful embedding of the category of Faltings--small Higgs bundles on $\X$ into the category of v-vector bundles on $\X$, which we shall denote by $\mathrm{S}_{\tilde{X}}^W$. Our next goal is to prove that our functor $\mathrm{S}_{\tilde{X}}$ restricted to Faltings--small objects agrees with $\mathrm{S}_{\tilde{X}}^W$. To be able to compare setups, we assume in the following that $S=\Z_p^\cycl.$
\begin{proposition} 
\label{prop:comparison-ls-with-wang}
Let $X$ be a qcqs smoothoid $p$-adic formal scheme over $S$ with generic fiber $\X$ and $\tilde{X}$ a lift as before. Let $\mathcal{M}=(M,\theta_M)\in\Higgs^{\Fsm}(\X)$. Then there is a natural isomorphism
$$
\mathrm{S}_{\tilde{X}}^W(\mathcal{M}) \cong \mathrm{S}_{\tilde{X}}(\mathcal{M}).
$$
\end{proposition}

\begin{remark}
Wang's Simpson functor $\mathrm{S}_{\tilde{X}}^W$ also agrees with Faltings' Simpson functor (\cite{faltings2005p}): cf. \cite[Remark 5.5]{wang2021p}. Hence our functor is also compatible with Faltings'.
\end{remark}

To construct $\mathrm{S}_{\tilde{X}}^W$, Wang defines a period sheaf $\O\mathbb{C}^\dagger$ depending on $\tilde{X}$ with a Higgs field\footnote{Compared to \cite{wang2021p}, we consider a Higgs field valued in the Breuil--Kisin twist of the differentials rather than the Tate twist (i.e. we renormalize by $z=\zeta_p-1$) to be in accordance with our conventions in the rest of the paper and in particular our definition of Higgs bundles. We note that this matches up the smallness condition in \Cref{def:faltings-small-higgs-bundles} with \cite[Definition 5.2]{wang2021p} because $p^{\nu_p(\rho_k)}\widehat{\Omega}^1_{\mathfrak{X}}(-1)=\widehat{\Omega}^1_{\mathfrak{X}}\{-1\}$ in the notation of loc.\ cit.\.} 
$$
\Theta_{\O\mathbb{C}^\dagger} : \O\mathbb{C}^\dagger \to \O\mathbb{C}^\dagger \otimes_{\mathcal{O}_{\X}} \nu^\ast \Omega_{\X}^1\{-1\}
$$
(cf.\ \cite[Definition 2.27]{wang2021p}) and then sets
$$
\mathrm{S}_{\tilde{X}}^W(M,\theta_M) \cong \ker(\O\mathbb{C}^\dagger\otimes_{\mathcal{O}_{\X}} \mu^\ast M \xrightarrow{\Theta_{\O\mathbb{C}^\dagger} \otimes \mathrm{Id} + \mathrm{Id} \otimes \theta_M}  \O\mathbb{C}^\dagger \otimes_{\mathcal{O}_{\X}} \mu^\ast M \otimes_{\mathcal{O}_{\X}} \nu^\ast \Omega_{\X}^1\{-1\}).
$$
Comparing with \Cref{explicit-formula_ls}, we see that our task is to compare the period sheaves $\mathcal{B}_{\tilde{X}}$ and $\O\mathbb{C}^\dagger$ and their Higgs fields (thereby providing a geometric description of the latter, in the spirit of \cite[\S II.9]{abbes2016p}).

The morphism 
\[
  \eta\colon X\to \Lft_{X}
\]induced by our lift $\tilde{X}$
is a $\mathcal{T}_{X}\{1\}$-torsor over $\Lft_{X}$. Equivalently, this torsor defines an extension
\[
  0\to \mathcal{O}_{\Lft_{X}}\to E^+_{\Lft_{X}} \to \Omega^1_{X}\{-1\}\otimes_{\O_X} \mathcal{\O}_{\Lft_{X}}\to 0
\]
on $\Lft_{X}$. Pulling back this extension via $u\cdot \can\colon X^\HT\to \Lft_{X}$ yields an extension
\begin{equation}
  \label{sec:appl-Hodge--Tate-wang-extension-on-x-ht}
  0\to \mathcal{O}_{X^\HT}\to E^+_{X^\HT}\to \Omega^1_{X}\{-1\}\otimes_{\O_X} \mathcal{\O}_{X^\HT}\to 0.  
\end{equation}

\begin{proposition}
  \label{sec:appl-Hodge--Tate-comparison-to-wangs-extension}
  The pullback of \eqref{sec:appl-Hodge--Tate-wang-extension-on-x-ht} to $\mathcal{O}^+_{\X_v}$-vector bundles is the twist by $(\zeta_p-1)\mathcal{O}_{\X_v}^+(-1)=\mathcal{O}_{\X_v}^+\{-1\}$ of Wang's ``integral Faltings extension''
  \[
    \textstyle 0\to \frac{1}{(\zeta_p-1)}\mathcal{O}^+_{\X_v}(1)\to \mathcal{E}^+\to  \Omega^1_{X} \otimes_{\mathcal{O}_X} \mathcal{O}^+_{\X_v} \to 0
  \]
  associated to the lift $\tilde{X}$ from \cite[Theorem 2.9]{wang2021p}.
\end{proposition}
\begin{proof}
Assume $X=\mathrm{Spf}(R)$ is affine. Let $S$ be a perfectoid ring with a map $\mathrm{Spf}(S) \to X$. By definition, a map $\mathrm{Spf}(S)\to \Lft_{X}$ is a morphism $R\to \mathcal{R}(S)$ of animated rings, which is also $\overline{A_0}$-linear (for the $\overline{A_0}$-linear structure on $\mathcal{R}(S)$ described above). The canonical lift $\tilde{S}$ of $S$ defines an $\overline{A_0}$-morphism $S\to \mathcal{R}(S)$ of animated rings, which we can precompose with $R \to S$. By definition, the pullback of $X\to \Lft_{X}$ to $S$ compares this composition $R \to \mathcal{R}(S)$ with the chosen $\overline{A_0}$-linear morphism $R\to \mathcal{R}(R)$ induced by $\tilde{R}$ (composed with $\mathcal{R}(R)\to \mathcal{R}(S)$). The isomorphisms between these two morphisms $R\to \mathcal{R}(S)$ identify with the maps of lifts $\tilde{R}\to \tilde{S}$. The functor sending $S$ as above to such maps of lifts is by deformation theory a torsor under $\mathrm{Hom}(\Omega^1_{X} \otimes_{\mathcal{O}_X} \mathcal{O}^+_{\X_v}, \mathcal{O}_{\X_v}^+\{1\})$ which is by construction (cf. \cite[p.12]{wang2021p}) the twist by $\mathcal{O}_{\X_v}^+\{-1\}$ of Wang's extension.
\end{proof}

\begin{remark}
\label{relation-with-original-faltings-extension}
If instead $X$ is a smooth $p$-adic formal scheme over $\mathcal{O}_K$, with $K$ a finite, unramified extension of $\Q_p$, then we can also recover from the Hodge--Tate stack $X^\HT$ the Faltings extension of the generic fiber $\X$ of $X$, as originally defined by Faltings\footnote{We thank Peter Scholze for a related discussion.}. Indeed, consider the natural morphism induced by the structure map of the Hodge--Tate stack and functoriality of its construction:
$$
X^\HT \to X \times_{\mathrm{Spf}(\mathcal{O}_K)} \mathrm{Spf}(\mathcal{O}_K)^\HT.
$$
It realizes the source as a $T_{X/\mathcal{O}_K}^\sharp \{1\}$-gerbe over the target. Let $Y$ denote its pushout along the natural map $T_{X/\mathcal{O}_K}^\sharp \{1\} \to T_{X/\mathcal{O}_K} \{1\}$. This gerbe canonically splits: indeed, as $K$ is unramified the first Breuil-Kisin twist has no cohomology on $\mathrm{Spf}(\mathcal{O}_K)^\HT$, so both the relevant $H^2$ and $H^1$ vanish. This canonical splitting makes $X \times_{\mathrm{Spf}(\mathcal{O}_K)} \mathrm{Spf}(\mathcal{O}_K)^\HT$ a $T_{X/\mathcal{O}_K}\{1\}$-torsor over $Y$, which can be pulled back to a $T_{X/\mathcal{O}_K}\{1\}$-torsor over $X^\HT$, corresponding to a class in $H^1(X^\HT,  T_{X/\mathcal{O}_K}\{1\})$. (We slighthly abuse notation by still denoting $T_{X/\mathcal{O}_K}\{1\}$ the pullback of $T_{X/\mathcal{O}_K}\{1\}$ to $X^\HT$.) After further pullback to $\X_v$ and inversion of $p$, this corresponds to an extension
$$
0 \to T_{\X/K} \otimes_{\mathcal{O}_{\X}} \mathcal{O}_{\X_v}(1)  \to \mathcal{E} \to \mathcal{O}_{\X_v} \to 0.
$$
As proved in \cite[II.10.19]{abbes2016p} (which in fact proves a finer, integral, statement), this recovers the Faltings extension by dualizing and twisting by $\mathcal{O}_{\X_v}(1)$. 
\end{remark}

Consider the commutative diagram
\[\begin{tikzcd}
	X \\
	{\Lft_{X}} & {z_\ast(X^\HT)\cong B\mathcal{T}^\sharp_{X}\{1\}}
	\arrow["{h}"', from=2-1, to=2-2]
	\arrow["\eta"', from=1-1, to=2-1]
	\arrow["\psi", from=1-1, to=2-2]
      \end{tikzcd}\]
    with $\eta, \psi$ induced by the chosen lift $\tilde{X}$. Then there exists a natural morphism
    \[
      c\colon h^\ast\psi_\ast(\mathcal{O}_X)\to \eta_\ast(\O_X).
    \]
    We now construct an overconvergent version of $c$.
    For $w\in \mathfrak m_{\Z_p^\cycl}$ consider the commutative diagram
\[\begin{tikzcd}
	X & {\Lft_{X}} & {\Lft_{X,w}} \\
	& {z_\ast(X^\HT)} & {(wz)_\ast(X^\HT)}
	\arrow["{g_w}", from=1-2, to=1-3]
	\arrow["{f_w}", from=2-2, to=2-3]
	\arrow["h", from=1-2, to=2-2]
	\arrow["\eta", from=1-1, to=1-2]
	\arrow["\psi", from=1-1, to=2-2]
	\arrow["{h_w}", from=1-3, to=2-3]
      \end{tikzcd}\] with $\Lft_{X,w}$ the pushforward of $\Lft_{X}$ along $\cdot w:\mathcal{T}_{X}\{1\}\to \mathcal{T}_{X}\{1\}$. Set $\eta_w:=g_w\circ \eta$, and $\psi_w:=f_w\circ \psi$. By \Cref{sec:appl-Hodge--Tate-comparison-to-wangs-extension}, Wang's $\O\mathbb{C}^{\dagger,+}$ is the colimit in sheaves on $\X_v$ of the pullback of the ind-object $"\varinjlim_{w\to 1}" g^\ast_w\eta_{w,\ast}\O_X$ on $\Lft_{X}$ to $\X_v$. By construction, we have natural maps, compatible with $c$,
    \[
      c_w\colon h_w^\ast\psi_{w,\ast}\O_X\to \eta_{w,\ast}\O_X.
    \]

    \begin{definition}
    \label{def:more-period-sheaves}
    For $w\in \mathfrak{m}_{\Z_p^\cycl}$, we let $\mathcal{B}_{\tilde{X},w}^+$ be the pullback of $\psi_{w,\ast}\O_X$ to $\X_v$. We also let
    $$\textstyle
    \mathcal{B}^{\dagger,+}_{\tilde{X}}=\varinjlim_{w\to 1} \mathcal{B}_{\tilde{X},w}^+
    $$ 
    It comes with an injection
    $
     c^{\dagger,+}\colon \mathcal{B}^{\dagger,+}_{\tilde{X}} \to \mathcal{O}\mathbb{C}^{\dagger,+},
    $
    which is not compatible with the Higgs fields on both sides, but satisfies $\Theta_{\mathcal{O}\mathbb{C}^{\dagger,+}}=u^{-1}\Theta_{\mathcal{B}^{\dagger,+}_{\tilde{X}}}$,
 as all $h_w$ are linear over $u^{-1}z:\mathcal{T}_{X}^\sharp \{1\}\to \mathcal{T}^\sharp_{X}\{1\}$ (with $u$ from \Cref{p:pushout-def-u}) and we have renormalized Wang's Higgs field by $z=\zeta_p-1$. Finally, let
 $$
 \mathcal{B}_{\tilde{X},w}=\mathcal{B}_{\tilde{X},w}^+\tf, ~~ \mathcal{B}^{\dagger}_{\tilde{X}}=\mathcal{B}^{\dagger,+}_{\tilde{X}}\tf.
 $$
    \end{definition}
    
\begin{proof}[Proof of \Cref{prop:comparison-ls-with-wang}]
Let $\mathcal{M}=(M,\theta_M)$ be a Faltings--small Higgs bundle on $\X$. Let $\alpha >\tfrac{1}{p-1}$ be as in \Cref{def:faltings-small-higgs-bundles} for $\mathcal{M}$. One sees that $\mathcal{M}$ comes from a vector bundle on $(wz)_\ast X^\HT$ for any $w \in \mathfrak{m}_{\Z_p^\cycl}$ such that $v_p(w)<\alpha-\tfrac{1}{p-1}$. We have natural injective maps
$$
 \mathcal{B}_{\tilde{X},w} \to  \mathcal{B}_{\tilde{X}}^\dagger \to  \mathcal{B}_{\tilde{X}} 
 $$
 inducing natural injective maps (to simplify notation, we omit the index of $\Theta$ in the superscript)
$$
(M\otimes_{\O_X}\mathcal{B}_{\tilde{X},w})^{\Id\otimes \Theta + \tfrac{1}{z}\theta_M\otimes \Id=0} \to (M\otimes_{\O_X}\mathcal{B}^{\dagger}_{\tilde{X}})^{\Id\otimes \Theta + \tfrac{1}{z}\theta_M\otimes \Id=0} \to (M\otimes_{\O_X}\mathcal{B}_{\tilde{X}})^{\Id\otimes \Theta + \tfrac{1}{z}\theta_M\otimes \Id=0}
$$ 
which are isomorphisms by \Cref{sec:isog-categ-perf-1-multiplication-by-x-on-v-sharp}. Therefore, one can replace $\mathcal{B}_{\tilde{X}}$ in \Cref{explicit-formula_ls} by $\mathcal{B}_{\tilde{X}}^\dagger$.
    
 We also have an injective map
 $$
 \mathcal{B}_{\tilde{X}}^\dagger \to \mathcal{O}\mathbb{C}^{\dagger} 
 $$
 induced by the inclusion $c^{\dagger,+}$ of \Cref{def:more-period-sheaves}. This map gives an injection  
 \[(M\otimes_{\O_X}\mathcal{B}^{\dagger}_{\tilde{X}})^{\Id\otimes \Theta + \tfrac{1}{z}\theta_M\otimes \Id=0}\to (M\otimes_{\O_X}\mathcal{O}\mathbb{C}^{\dagger})^{\Id\otimes \Theta + \theta_M\otimes \Id=0},\] by using $\Id\otimes c^{\dagger,+}$ and multiplication by $u^{-1}$. To check that this map is an isomorphism we may argue locally and assume that $\tilde{X}$ is induced by a prismatic lift: all lifts are locally isomorphic and the statement is natural in $\tilde{X}$. If $\eta^\prime\colon X\to X^\HT$ is induced by the prismatic lift, and $j:X^\HT\to z_{\ast}X^\HT$ is the natural map as before, then there is an injection $j^{\ast}\psi_\ast(\mathcal{O}_X)\to \eta_\ast^\prime(\mathcal{O}_X)$. It yields an injection \[\mathcal{B}_{\tilde{X}}\hookrightarrow \mathcal{C}:=\mathcal C^+\tf=\alpha_X^{\ast}(\eta_\ast^\prime(\mathcal{O}_X)\tf),\]
 where $\mathcal{C}^+$ is the pullback of $\eta_\ast^\prime(\mathcal{O}_X)$ to the v-site of the generic fiber $\mathcal X$. By \Cref{sec:isog-categ-perf-1-multiplication-by-x-on-v-sharp} the injection $\mathcal{B}_{\tilde{X}}\to \mathcal{C}$ induces an isomorphism on the $\mathcal{H}^0$ of the Dolbeault complexes for $(M,\theta_M)$. As we have seen above, the same holds for $\mathcal{B}^\dagger_{\tilde{X}}\to \mathcal{C}$. This implies the same statement for the map induced by $\mathcal{O}\mathbb{C}^\dagger\to \mathcal{C}$ on Dolbeault complexes for $(M,\theta_M)$, and thus for the one induced by $\mathcal{B}^\dagger_{\tilde{X}}\to \mathcal{O}\mathbb{C}^\dagger$ because we already know injectivity on the $\mathcal{H}^0$. This finishes the proof that
$\mathrm{S}_{\tilde{X}}^W(\mathcal{M}) \cong \mathrm{S}_{\tilde{X}}(\mathcal{M})$.
\end{proof}

  \begin{remark}
    \label{sec:appl-Hodge--Tate-comparison-to-wang-locally-and-explicit}
    We can make the above formulas for $\mathrm{S}_{\tilde{X}}$ more explicit in the setup of \S\ref{sec:tori-over-geometric}, i.e., $X=\Spf(R)$ for some prism $(A,I)$ over $(A_0,I_0)$, and there is a fixed morphism $(A,I)\to (A_\infty, I_\infty)$. We use the lift $\tilde{X}=\Spf(A/I^2)$ with its natural morphism from $\Spf(A_\infty/I_\infty^2)$. In this situation we evaluate $\mathcal{B}^+_{\tilde{X}}, \mathcal{O}\mathbb{C}^{\dagger,+}$, the pullback $\mathcal{D}^+$ of $\eta_\ast(\O)$ to $\X_v$ and the overconvergent variants on $\X_\infty=X_\infty^\rig$. Comparing $\eta\colon X\to \Lft_{X}$ with the given splitting $\eta^\prime\colon X\to X^\HT$ realises all these rings as subrings of $B^+_{A,R_\infty}$, which itself identifies with $\O(G_{A,R_\infty})$. Note that $B^+_{A,R_\infty}=\mathcal{C}^+(\mathcal{X}_\infty)$ in the notation of the proof of \Cref{prop:comparison-ls-with-wang}. Set $E:=\Omega^1_{R}\{-1\}\otimes_RR_\infty$. Then we have the following:
    \begin{enumerate}
    \item $\O(G_{A,R_\infty})=\Gamma_{R_\infty}^\bullet(E)^\wedge_p$ is the $p$-completed PD-algebra on $E$,
    \item $\mathcal{B}^+_{\tilde{X}}(\X_\infty)=\Gamma_{R_\infty}^\bullet((\zeta_p-1)E)^\wedge_p$,
    \item $\mathcal{D}^+(\X_\infty)=\mathcal{S}_p(E)$ is the $p$-completed symmetric algebra on $E$ considered as an $R_\infty$-module. It contains $\Gamma_{R_\infty}^\bullet((\zeta_p-1)E)^\wedge_p$ because $(\zeta_p-1)$ admits divided powers,
    \item $\mathcal{B}^+_{\tilde{X},w}(\X_\infty)=\Gamma_{R_\infty}^\bullet(w(\zeta_p-1)E)^\wedge_p$ for some $w\in \mathfrak{m}_{\Z_p^\cycl}$,
      \item $\mathcal{B}^{\dagger,+}_{\tilde{X}}(\X_\infty)=\varinjlim_{w\to 1}\Gamma_{R_\infty}^\bullet(w(\zeta_p-1)E)^\wedge_p$,
     \item $\mathcal{O}\mathbb{C}^{\dagger,+}(\X_\infty)=\varinjlim_{w\to 1} \mathcal{S}_p(w E)$, which is contained in $\O(G_{A,R_\infty})$ and contains $\mathcal{B}^{\dagger,+}_{\tilde{X}}(\X_\infty)$.    
     \end{enumerate}
     In summary, the proof of \Cref{prop:comparison-ls-with-wang} now used the following diagram 
     of injections:
\[\begin{tikzcd}[row sep = 0.3cm]
	{(5)} & {(2)} \\
	{(6)} & {(3)} & {(1)}
	\arrow[hook, from=1-1, to=1-2]
	\arrow[hook, from=1-1, to=2-1]
	\arrow[hook, from=2-1, to=2-2]
	\arrow[hook, from=1-2, to=2-3]
	\arrow[hook, from=2-2, to=2-3]
      \end{tikzcd}\]
  \end{remark}


\begin{thebibliography}{10}
	
	\bibitem{abbes2016p}
	A.~Abbes, M.~Gros, and T.~Tsuji.
	\newblock {\em The $p$-adic {S}impson {C}orrespondence (AM-193)}, volume 193.
	\newblock Princeton University Press, 2016.
	
	\bibitem{andre2018conjecture}
	Y.~Andr{\'e}.
	\newblock La conjecture du facteur direct.
	\newblock {\em Publications math{\'e}matiques de l'IH{\'E}S}, 127(1):71--93,
	2018.
	
	\bibitem{andreychev2021pseudocoherent}
	G.~Andreychev.
	\newblock {P}seudocoherent and perfect complexes and vector bundles on analytic
	adic spaces.
	\newblock {\em Preprint, arXiv:2105.12591}, 2021.
	
	\bibitem{analytic_HT}
	J.~Ansch{\"u}tz, B.~Heuer, and A.-C. Le~Bras.
	\newblock v-vector bundles on $p$-adic fields and {S}en theory via the
	{H}odge-{T}ate stack.
	\newblock {\em Preprint, arXiv:2211.08470}, 2022.
	
	
	\bibitem{AHLB-companion}
	J.~Anschütz, B.~Heuer, and A.-C. {Le Bras}.
	\newblock The small $p$-adic {S}impson correspondence in terms of moduli
	spaces.
	\newblock {\em Preprint, arXiv:2312.07554, to appear in Math.\ Res.\ Lett.}
	
	\bibitem{anschutz2021fourier}
	J.~Ansch{\"u}tz and A.-C. Le~Bras.
	\newblock {A} {F}ourier {T}ransform for {B}anach-{C}olmez spaces.
	\newblock {\em Preprint, arXiv:2111.11116, to appear in J. Eur. Math. Soc.},
	2021.
	
	
	\bibitem{berthelot2015notes}
	P.~Berthelot and A.~Ogus.
	\newblock {\em Notes on crystalline cohomology.(MN-21)}, volume~21.
	\newblock Princeton University Press, 2015.
	
	\bibitem{bhatt_lectures_on_prismatic_cohomology}
	B.~Bhatt.
	\newblock {L}ectures on prismatic cohomology.
	\newblock available at
	\url{http://www-personal.umich.edu/~bhattb/teaching/prismatic-columbia/}.
	
	\bibitem{bhatt2022absolute}
	B.~Bhatt and J.~Lurie.
	\newblock {A}bsolute prismatic cohomology.
	\newblock {\em Preprint, arXiv:2201.06120}, 2022.
	
	\bibitem{bhatt2022F-gauges}
	B.~Bhatt and J.~Lurie.
	\newblock Prismatic F-gauges.
	\newblock {\em Lecture notes available at
		\url{https://www.math.ias.edu/~bhatt/teaching/mat549f22/lectures.pdf}}, 2022.
	
	\bibitem{bhatt2022prismatization}
	B.~Bhatt and J.~Lurie.
	\newblock {T}he prismatization of $ p $-adic formal schemes.
	\newblock {\em Preprint, arXiv:2201.06124}, 2022.
	
	\bibitem{Bhatt2018}
	B.~Bhatt, M.~Morrow, and P.~Scholze.
	\newblock {I}ntegral $p$-adic {H}odge theory.
	\newblock {\em Publications math{\'e}matiques de l'IH{\'E}S}, 128(1):219--397,
	2018.
	
	\bibitem{Bhatta}
	B.~Bhatt and P.~Scholze.
	\newblock {Prisms and prismatic cohomology}.
	\newblock {\em Annals of Mathematics}, 196(3):1135 -- 1275, 2022.
	
	\bibitem{BrinonSen}
	O.~Brinon.
	\newblock Une g\'{e}n\'{e}ralisation de la th\'{e}orie de {S}en.
	\newblock {\em Math. Ann.}, 327(4):793--813, 2003.
	
	\bibitem{cesnavicius2019purity}
	K.~\v{C}esnavi\v{c}ius and P.~Scholze.
	\newblock Purity for flat cohomology.
	\newblock {\em Ann. of Math. (2)}, 199(1):51--180, 2024.
	
	\bibitem{clausenscholzecomplex}
	D.~{Clausen} and P.~{Scholze}.
	\newblock {L}ectures on {C}omplex {G}eometry.
	\newblock available at
	\url{https://people.mpim-bonn.mpg.de/scholze/Complex.pdf}.
	
	\bibitem{drinfeld2020prismatization}
	V.~Drinfeld.
	\newblock Prismatization.
	\newblock {\em Selecta Mathematica}, 30(3):49, 2024.
	
	\bibitem{faltings2005p}
	G.~Faltings.
	\newblock A $p$-adic {S}impson correspondence.
	\newblock {\em Advances in Mathematics}, 198(2):847--862, 2005.
	
	\bibitem{gao2020integral}
	H.~Gao.
	\newblock Integral $p$-adic {H}odge theory in the imperfect residue field case.
	\newblock {\em Preprint, arXiv:2007.06879}, 2020.
	
	\bibitem{gao2021padic}
	H.~Gao.
	\newblock On $p$-adic {S}impson and {R}iemann-{H}ilbert correspondences in the
	imperfect residue field case.
	\newblock {\em Preprint, arXiv:2108.07029}, 2021.
	
	\bibitem{he2022sen}
	T.~He.
	\newblock Sen operators and {L}ie algebras arising from {G}alois
	representations over $ p $-adic varieties.
	\newblock {\em Preprint, arXiv:2208.07519}, 2022.
	
	\bibitem{G-torsors-perfectoid-spaces}
	B.~Heuer.
	\newblock {$G$}-torsors on perfectoid spaces.
	\newblock {\em Preprint, arXiv:2207.07623}, 2022.
	
	\bibitem{heuer-sheafified-paCS}
	B.~Heuer.
	\newblock Moduli spaces in $p$-adic non-abelian {H}odge theory.
	\newblock {\em Preprint, arXiv:2207.13819}, 2022.
	
	\bibitem{HWZ}
	B.~Heuer, A.~Werner, and M.~Zhang.
	\newblock $p$-adic {S}impson correspondences for principal bundles in abelian
	settings.
	\newblock 2023.
	\newblock Preprint, arXiv:2308.13456.
	
	\bibitem{humphreys2012introduction}
	J.~E. Humphreys.
	\newblock {\em Introduction to {L}ie algebras and representation theory},
	volume~9 of {\em Graduate Texts in Mathematics}.
	\newblock Springer-Verlag, New York-Berlin, 1978.
	\newblock Second printing, revised.
	
	\bibitem{Hyodo-HT-imperfect-res}
	O.~Hyodo.
	\newblock On the {H}odge-{T}ate decomposition in the imperfect residue field
	case.
	\newblock {\em J. Reine Angew. Math.}, 365:97--113, 1986.
	
	\bibitem{laumon1996transformation}
	G.~Laumon.
	\newblock Transformation de {F}ourier g{\'e}n{\'e}ralis{\'e}e.
	\newblock {\em Preprint, alg-geom/9603004}, 1996.
	
	\bibitem{LiuZhu_RiemannHilbert}
	R.~Liu and X.~Zhu.
	\newblock Rigidity and a {R}iemann-{H}ilbert correspondence for {$p$}-adic
	local systems.
	\newblock {\em Invent. Math.}, 207(1):291--343, 2017.
	
	\bibitem{lurie_spectral_algebraic_geometry}
	J.~Lurie.
	\newblock {S}pectral algebraic geometry.
	\newblock Available at \url{http://www.math.harvard.edu/~lurie/}.
	
	\bibitem{min2021hodge}
	Y.~Min and Y.~Wang.
	\newblock {O}n the {H}odge--{T}ate crystals over {$\mathcal O_K$}.
	\newblock {\em Preprint, arXiv:2112.10140}, 2021.
	
	\bibitem{MinWang22}
	Y.~Min and Y.~Wang.
	\newblock $p$-adic {S}impson correpondence via prismatic crystals.
	\newblock {\em Preprint, arXiv:2201.08030}, 2022.
	
	\bibitem{ogus_vologodsky}
	A.~Ogus and V.~Vologodsky.
	\newblock Nonabelian {H}odge theory in characteristic $p$.
	\newblock {\em Publications math{\'e}matiques}, 106(1):1--138, 2007.
	
	\bibitem{OhkuboSenTheory}
	S.~Ohkubo.
	\newblock A note on {S}en's theory in the imperfect residue field case.
	\newblock {\em Math. Z.}, 269(1-2):261--280, 2011.
	
	\bibitem{camargo2022locally}
	J.~E. Rodriguez~Camargo.
	\newblock Locally analytic completed cohomology of shimura varieties and
	overconvergent {$BGG$} maps.
	\newblock {\em Preprint, arXiv:2205.02016}, 2022.
	
	\bibitem{scholze2013p}
	P.~Scholze.
	\newblock {$p$}-adic {H}odge theory for rigid-analytic varieties.
	\newblock {\em Forum Math. Pi}, 1:e1, 77, 2013.
	
	\bibitem{scholze2013perfectoid}
	P.~Scholze.
	\newblock {\em {P}erfectoid spaces: {A} survey}.
	\newblock Current developments in mathematics 2012. Int. Press, Somerville, MA,
	2013.
	\newblock 193--227.
	
	\bibitem{StacksProjectAuthors2017}
	T.~{Stacks Project Authors}.
	\newblock \textit{{S}tacks {P}roject}.
	\newblock \url{http://stacks.math.columbia.edu}, 2024.
	
	\bibitem{thomason2007higher}
	R.~W. Thomason and T.~Trobaugh.
	\newblock Higher algebraic {K}-theory of schemes and of derived categories.
	\newblock In {\em The {G}rothendieck {F}estschrift}, pages 247--435. Springer,
	2007.
	
	\bibitem{tian2021finiteness}
	Y.~Tian.
	\newblock Finiteness and duality for the cohomology of prismatic crystals.
	\newblock {\em Journal für die reine und angewandte Mathematik (Crelles
		Journal)}, 2023(800):217--257, 2023.
	
	\bibitem{Tsuji-localSimpson}
	T.~Tsuji.
	\newblock Notes on the local {$p$}-adic {S}impson correspondence.
	\newblock {\em Math. Ann.}, 371(1-2):795--881, 2018.
	
	
	\bibitem{wang2021p}
	Y.~Wang.
	\newblock A p-adic {S}impson correspondence for rigid analytic varieties.
	\newblock {\em Algebra \& Number Theory}, 17(8):1453--1499, 2023.
	
	\bibitem{YamauchiSen}
	T.~Yamauchi.
	\newblock A generalization of {S}en-{B}rinon's theory.
	\newblock {\em Manuscripta Math.}, 133(3-4):327--346, 2010.
	
\end{thebibliography}
\end{document}